\documentclass[oneside,reqno,10pt]{amsart}

\usepackage[margin=1in]{geometry}
\usepackage{mathtools}
\usepackage{amsmath}
\usepackage{amsthm}
\usepackage{amssymb}
\usepackage{tikz}
\usepackage{tikz-cd}
\usepackage[all,cmtip]{xy}
\usepackage{float}
\usepackage{indentfirst}
\usepackage[mathscr]{euscript}
\usepackage{stmaryrd}
\usepackage{hyperref}
\usepackage{csquotes}
\usepackage{enumitem}
\usepackage{graphicx}
\usepackage{subcaption}

\DeclareFontFamily{U}{mathx}{}
\DeclareFontShape{U}{mathx}{m}{n}{<-> mathx10}{}
\DeclareSymbolFont{mathx}{U}{mathx}{m}{n}
\DeclareMathAccent{\widecheck}{0}{mathx}{"71}

\makeatletter
\def\blfootnote{\gdef\@thefnmark{}\@footnotetext}
\makeatother

\theoremstyle{plain}
\newtheorem{thm}{Theorem}[section]
\newtheorem{prop}[thm]{Proposition}
\newtheorem{lem}[thm]{Lemma}

\theoremstyle{definition}
\newtheorem{dfn}[thm]{Definition}
\newtheorem{exa}[thm]{Example}
\newtheorem*{ack}{Acknowledgements}

\theoremstyle{remark}
\newtheorem{rmk}[thm]{Remark}

\numberwithin{equation}{section}

\newcommand{\W}{\mathcal{W}}
\newcommand{\F}{\mathcal{F}}
\newcommand{\RW}{\mathcal{RW}}
\newcommand{\Hom}{\mathrm{Hom}}

\newcommand{\Z}{\mathbb{Z}}
\newcommand{\R}{\mathbb{R}}

\title{Calabi--Yau structures on Rabinowitz Fukaya categories}

\author{Hanwool Bae}
\address{Center for Quantum Structures in Modules and Spaces, Seoul National University, Seoul 08826, Republic of Korea}
\email{hanwoolb@gmail.com}

\author{Wonbo Jeong}
\address{Research Institute of Mathematics, Seoul National University, Seoul 08826, Republic of Korea}
\email{wonbo.jeong@gmail.com}

\author{Jongmyeong Kim}
\address{Center for Quantum Structures in Modules and Spaces, Seoul National University, Seoul 08826, Republic of Korea}
\email{thereisnoroyalroadtogeometry@gmail.com}

\begin{document}

\begin{abstract}
In this paper, we prove that the derived Rabinowitz Fukaya category of a Liouville domain $M$ of dimension $2n$ is $(n-1)$-Calabi--Yau assuming the wrapped Fukaya category of $M$ admits an at most countable set of Lagrangians that generate it and satisfy some finiteness condition on morphism spaces between them.

\end{abstract}
\maketitle
\tableofcontents

\section{Introduction}\label{section:introduction}

	Rabinowitz Floer homology is a Floer homology associated to a contact hypersurface of a symplectic manifold, which was first introduced by Cieliebak and Frauenfelder \cite{cie-fra09}. In particular, in \cite{cfo10}, Cieliebak, Oancea and Frauenfelder proved that the Rabinowitz Floer homology $RFH_*(M)$ of a Liouville domain $M$ of dimension $2n$ fits into a long exact sequence 
\begin{equation}\label{eq:les}
\dots \to SH^{-*}(M) \to SH_*(M) \to RFH_*(M) \to SH^{-*+1}(M) \to \cdots.
\end{equation}	
Here, the map $SH^{-*}(M) \to SH_*(M)$ factors through
$$ \begin{tikzcd}
SH^{-*}(M) \arrow{r}\arrow{d}{c^*} & SH_*(M)\\ H_{*+n}(M) \arrow{r} & H_{*+n}(M,\partial M) \arrow{u}{c_*},
\end{tikzcd} $$
where the vertical arrows $c^* : SH^{-*}(M) \to H_{*+n}(M)$ and $c_* : H_{*+n}(M,\partial M) \to SH_*(M)$ are induced by quotient maps and inclusion maps between Floer complexes in certain action windows.

In \cite{ven18, cie-oan20}, this map was understood as a continuation map. Indeed, they consider Floer homologies for Hamiltonians that are linear with negative slopes and positive slopes at infinity and then consider the continuation map from the former one to the latter one. This leads to a new approach to defining Rabinowitz Floer complex as a mapping cone of a continuation map between two Floer complexes.

Adapting this idea to an open string analogue of Rabinowitz Floer homology, Ganatra, Gao and Venkatesh \cite{ggv22} recently introduced {\em Rabinowitz wrapped Fukaya category} of a Liouville domain. For a Liouville domain $M$,  objects of the Rabinowitz wrapped Fukaya category $\RW(M)$ are certain Lagrangian submanifolds of $M$ and the morphism space between two Lagrangian submanifolds $L_0$ and $L_1$ of $M$ is given by the mapping cone of the continuation map from the Floer complex $CW_*(L_0,L_1)$ for wrapped Floer homology to that $CW^*(L_0,L_1)$ for wrapped Floer cohomology. Namely,
the Rabinowitz Floer complex between $L_0$ and $L_1$ is given by
$$ RFC^*(L_0,L_1) =\mathrm{Cone}( c : CW_{n-*}(L_0,L_1) \to CW^*(L_0,L_1) )$$
for the continuation map $c$.
Furthermore, if $M$ is Weinstein, they showed that its Rabinowitz wrapped Fukaya category $\RW(M)$ is quasi-equivalent to the formal punctured neighborhood of infinity $\widehat{\W}_{\infty}(M)$ of the wrapped Fukaya category $\W(M)$. In particular, the latter one $\widehat{\W}_{\infty}(M)$ is further shown to be quasi-equivalent to the quotient $\W(M)/\F(M)$ by the compact Fukaya category $\F(M)$ when it is a Koszul dualizing subcategory of $\W(M)$.

Meanwhile, in our previous work \cite{bjk22}, we consider plumbings of $T^*S^n$ along a tree for $n\geq 3$. The main result is that, for such a Weinstein manifold, the derived wrapped Fukaya category, the derived compact Fukaya category, and the collection of cocore disks $(D^{\pi}\mathcal{W}, D^{\pi}\mathcal{F}, L)$ form a Calabi--Yau triple. As a consequence, the quotient category $D^{\pi}\mathcal{W}/D^{\pi}\mathcal{F}$ becomes an $(n-1)$-Calabi--Yau (generalized) cluster category.

Here, Calabi--Yau property is defined as follows.
\begin{dfn}\label{dfn:calabiyau}
	A triangulated category $\mathcal{C}$ over a field $\mathbb{K}$ is said to be {\em $n$-Calabi--Yau} for some $n \in \mathbb{Z}$ if, for all objects $X$ and $Y$ of $\mathcal{C}$,
	\begin{itemize}
		\item $\dim \Hom_{\mathcal{C}} (X,Y)<\infty$ and
		\item there exists a non-degenerate pairing
		$$\beta_{X,Y} : \Hom_{\mathcal{C}}(Y,X[n]) \times \Hom_{\mathcal{C}}(X,Y) \to \mathbb{K},$$
		which is bifunctorial.	
	\end{itemize}
	The bifunctoriality means that $\beta_{X,Y}$ defines a natural transformation between two functors from 
	$\mathcal{C}^{op} \times \mathcal{C}$ to $\mathrm{Vect}\, \mathbb{K}$ given by
	$$(X,Y) \mapsto \Hom_{\mathcal{C}}(X,Y)$$
	and 
	$$(X,Y) \mapsto \Hom_{\mathcal{C}}(Y,X[n])^{\vee}.$$	
\end{dfn}

Since the compact Fukaya category is known to be Koszul dual to the wrapped Fukaya category in this case \cite{etg-lek17,ekh-lek17}, this means that the derived Rabinowitz wrapped Fukaya category of the plumbings carries an $(n-1)$-Calabi--Yau cluster structure.

At this point, the following question arises naturally:
\begin{center}
{\em When does the derived Rabinowitz wrapped Fukaya category admit a Calabi--Yau structure?}
\end{center}
In this paper, we give a partial answer to this question. Indeed, under the assumption that the wrapped Fukaya category of a Liouville domain $M$ with $2c_1(M)=0$ is generated by an at most countable family of Lagrangians of $M$,  we provide equivalent conditions for its derived Rabinowitz wrapped Fukaya category to be Calabi--Yau. This can be stated as follows.

\begin{thm}[Theorem \ref{thm:maintheorem2}]\label{thm:maintheorem}
Let $M$ be a Liouville domain of dimension $2n$ with $2c_1(M) =0$ and let $b\in H^2(M,\Z/2)$ be a background class. Assume that $M$ admits a set of admissible Lagrangians $\{L_i\}_{i \in I}$ for some at most countable set $I$ generating the wrapped Fukaya category $\mathcal{W}_b(M)$ with background class $b$. Then the following are equivalent.
\begin{enumerate}
\item\label{condition1} The derived Rabinowitz wrapped Fukaya category $D^{\pi}\RW(M)$ is $(n-1)$-Calabi--Yau.
\item\label{condition2} For all $i,j \in I$ and $k \in \mathbb{Z}$, $\dim RFH^k(L_i,L_j) <\infty$.
\item\label{condition3} For all $i,j \in I$ and $k \in \mathbb{Z}$, $\dim HW^k(L_i,L_j) <\infty$.
\end{enumerate}
\end{thm}
Here, $RFH^*(L_i,L_j)$ denotes the homology of the Rabinowitz Floer complex $RFC^*(L_i,L_j)$. See Remark \ref{rmk:wrappedfukayacategory} or \cite{abo11} for a definition of the wrapped Fukaya category $\W_b(M)$.

It was proved in \cite[Theorem 1.3]{cho20} that, under the assumption that $2c_1(M)=0$, $2c_1(M,L)=0$ and $L$ is oriented, there is a graded open-closed TQFT structure on the pair of a (degree shifted) Rabinowitz Floer homology $RF\mathbb{H}_*(\partial M)$ of a Liouville domain $M$ and a (degree shifted) Rabinowitz Floer homology ${RF}\mathbb{H}_*(\partial L)$ of a Lagrangian submanifold $L$ of $M$ having a cylindrical Legendrian end. In particular, if $RF\mathbb{H}^*(\partial M)$ and $RF\mathbb{H}^*(\partial L)$ are of finite type, meaning that they are finite-dimensional in each degree, then this implies that there exist non-degenerate pairings on $RF\mathbb{H}^*(\partial M)$ and $RF\mathbb{H}^*(\partial L)$. In terms of our cohomological convention, this implies the existence of a non-degenerate pairing between $RFH^*(L, L)$ and $RFH^{n-1-*} (L,L)$.

Theorem \ref{thm:maintheorem} can be viewed as a partial extension of the result of \cite{cho20} to a categorical level under the finite type assumption in the sense that we prove that the duality pairing extends to any pair of objects of the derived Rabinowitz wrapped Fukaya category. This extension was available since we worked with $A_{\infty}$-category and operators on chain level. We guess that many experts would have already expected Theorem \ref{thm:maintheorem} to be true.

Let us now return to how to prove Theorem \ref{thm:maintheorem}. We prove it by showing \eqref{condition1} $\Rightarrow$ \eqref{condition2} $\Rightarrow$ \eqref{condition3} $\Rightarrow$ \eqref{condition1}. Here, we provide the proofs of the implications \eqref{condition1} $\Rightarrow$ \eqref{condition2} and \eqref{condition2} $\Rightarrow$ \eqref{condition3} and briefly explain the strategy  for the proof of the implication \eqref{condition3} $\Rightarrow$ \eqref{condition1}, which will be proved in Section \ref{section:cy}.

The first implication \eqref{condition1} $\Rightarrow$ \eqref{condition2} follows from the definition of Calabi--Yau property and
the fact that $\Hom_{D^{\pi} \RW(M)}(L_i,L_j[k]) \cong RFH^k (L_i,L_j)$ by definition of the derived Rabinowitz wrapped Fukaya category $D^{\pi}\RW(M)$.

For the second implication \eqref{condition2} $\Rightarrow$ \eqref{condition3}, consider the long exact sequence:
\begin{equation}\label{eq:openles}
\dots \to HW_{n-*}(L_i,L_j) \xrightarrow{c_*} HW^*(L_i,L_j) \to RFH^*(L_i,L_j) \to HW_{n-*-1}(L_i,L_j) \xrightarrow{c_*} \cdots,
\end{equation}	
which is a consequence of the definition of the Rabinowitz Floer complex $RFC^*(L_i,L_j)$. In Section \ref{section:rw}, we will see that the continuation map $c : CW_{n-*}(L_i,L_j) \to CW^*(L_i,L_j)$ factors through
$$ \begin{tikzcd}
CW_{n-*}(L_i,L_j)  \arrow{r}{c} \arrow{d}{} & CW^*(L_i,L_j) \\CF^*(L_i,L_j;H_{-\epsilon}) \arrow{r}{c_{-\epsilon,\epsilon}} &  CF^*(L_i,L_j;H_{\epsilon}) \arrow{u},
\end{tikzcd} $$
where $c_{-\epsilon,\epsilon}:CF^*(L_i,L_j;H_{-\epsilon}) \to CF^*(L_i,L_j;H_{\epsilon})$ is a continuation map between Floer complexes of $L_i$ and $L_j$ for Hamiltonians $H_{\pm \epsilon}$ linear at infinity with slope $\pm \epsilon$ for a sufficiently small $\epsilon>0$. This means that the image of the continuation $c_* : HW_{n-*}(L_i,L_j) \to HW^*(L_i,L_j)$  is finite-dimensional since $\dim HF^*(L_i,L_j;H_{\pm \epsilon})$ is finite. Therefore, if $HW^k(L_i,L_j)$ is infinite-dimensional, then $RFH^k(L_i,L_j)$ is infinite-dimensional as well.

Finally, the proof of the last implication \eqref{condition3} $\Rightarrow$ \eqref{condition1} relies on the duality between the wrapped Floer cohomology $HW^*(L_0,L_1)$ and the wrapped Floer homology $HW_*(L_1,L_0)$ for Lagrangians $L_0$ and $L_1$, which comes from the natural duality between $CF^*(L_0,L_1;H_a)$ and $CF^{n-*}(L_1,L_0;-H_a)$ for a Hamiltonian $H_a$ linear at infinity with slope $a \in \mathbb{R}$. 

The duality between $HF^*(L_0,L_1;H_a)$ and $HF^{n-*}(L_1,L_0;H_{-a})$ can be obtained as the composition of the pair-of-pants product
$$HF^{n-*}(L_1,L_0; H_{-a -\epsilon}) \otimes  HF^*(L_0,L_1;H_a) \to HF^{n}(L_0,L_0; H_{-\epsilon})$$
and the integration  
$$HF^n(L_0,L_0; H_{-\epsilon}) \cong H^n(L_0,\partial L_0;\mathbb{K}) \to \mathbb{K}$$
for $\epsilon>0$ sufficiently small so that the following are isomorphic:
$$HF^{n-*}(L_1,L_0;H_{-a}) \cong HF^{n-*}(L_1,L_0;H_{-a-\epsilon}).$$

Since both of these operators originate from operators on chain complex level, we construct chain maps
\begin{align*}
\overline{\alpha} &: CW^*(L_0,L_1) \to CW_{*}(L_1,L_0)^{\vee}, \\
\overline{\gamma} &: CW_{*}(L_0,L_1) \to CW^*(L_1,L_0)^{\vee},
\end{align*}
which give the duality explained above. Then we construct a chain map
$$\overline{\beta} : RFC^*(L_0,L_1) \to RFC^{n-1-*} (L_1,L_0)^{\vee}$$
so that the following diagram commutes:
 \begin{equation*}
	\xymatrix{ 0\ar[r] & CW^*(L_0,L_1)\ar[r]^{} \ar[d]^{\overline{\alpha}}&RFC^*(L_0,L_1) \ar[r]^{} \ar[d]^{\overline{\beta}}&CW_{n-1-*}(L_0,L_1)\ar[r] \ar[d]^{\overline{\gamma}}&0 \\ 0\ar[r]& CW_*(L_1,L_0)^{\vee} \ar[r]^{} &RFC^{n-1-*}(L_1,L_0)^{\vee} \ar[r]^{}& CW^{n-1-*}(L_1,L_0)^{\vee} \ar[r]&0. }
\end{equation*}
We found it easier to use linear Hamiltonians to construct such chain maps, rather than quadratic Hamiltonians as in \cite{ggv22}.

Although the chain map $\overline{\gamma}$ is always a quasi-isomorphism, the same cannot be said for $\overline{\alpha}$ due to the same reason as in \cite[Section 3.3]{cie-oan18}. However, the condition that $HW^k(L_0,L_1)$ is finite-dimensional for each degree $k\in \Z$ ensures that $\overline{\alpha}$ is a quasi-isomorphism as well. As a result, the chain map $\overline{\beta}$ becomes a quasi-isomorphism under the assumption. Finally, we observe that the chain map $\overline{\beta}$ naturally extends to the triangulated envelope and gives a Calabi--Yau paring on the derived Rabinowitz wrapped Fukaya category.\\

Let us now give a brief overview of this paper.
In Section \ref{section:rw}, we construct Rabinowitz Fukaya category using linear Hamiltonians as mentioned above.
In Section \ref{section:cy}, we construct the chain maps $\overline{\alpha}$, $\overline{\beta}$ and $\overline{\gamma}$ that fit into the above commutative diagram and show that these are quasi-isomorphisms under the assumption that the condition \eqref{condition3} of Theorem \ref{thm:maintheorem} is satisfied. Then we show that it extends to the triangulated envelope.
Finally, in Section \ref{section:examples}, we provide three examples of Liouville domains for which  condition \eqref{condition3} of Theorem \ref{thm:maintheorem} is satisfied. This includes cotangent bundles of smooth manifolds with a finite fundamental group, plumbings of $T^*S^n$ along a tree and Weinstein domains with a periodic Reeb flow on the boundary.

\begin{ack}
	We thank Cheol-Hyun Cho, Jungsoo Kang, Yeongrak Kim, Yoosik Kim and Yong-Geun Oh for helpful discussion and comments. We are especially grateful to the anonymous referee for invaluable comments. Hanwool Bae and Jongmyeong Kim were supported by the National Research Foundation of Korea (NRF) grant funded by the Korea government (MSIT) (No.2020R1A5A1016126). Wonbo Jeong was supported by Samsung Science and Technology Foundation under project number SSTF-BA1402-52.
\end{ack}

\section{Rabinowitz wrapped Fukaya category}\label{section:rw}

In this section, we present a construction of Rabinowitz wrapped Fukaya category. For that purpose, first consider the open string analogue \eqref{eq:openles} of the long exact sequence \eqref{eq:les}. Since we use a cohomological convention, the counterparts to the symplectic homology and the symplectic cohomology in the long exact sequence \eqref{eq:les} are wrapped Floer cohomology and wrapped Floer homology as in \eqref{eq:openles}, respectively.

For the Rabinowitz Floer cohomology, in \cite{ggv22}, the authors constructed its open string counterpart in two steps.
The authors first define the Floer complexes for wrapped Floer homology and wrapped Floer cohomology using quadratic Hamiltonians. Then they define the Rabinowitz wrapped Floer complex of two admissible Lagrangians $L_0$ and $L_1$ as the cone of the continuation map from the Floer complex for wrapped Floer homology to that for wrapped Floer cohomology as in \cite{ven18}. Consequently, its cohomology fits into the long exact sequence \eqref{eq:openles}.\\

We first review the construction of wrapped Floer cohomology and wrapped Floer homology using linear Hamiltonians \cite{abo-sei10,ven18} in the following three subsections. Then we use those to construct Rabinowitz Floer complex and Rabinowitz Fukaya category in Subsection \ref{subsection:rabinowitz}.

\subsection{Floer cohomology}\label{subsection:floercohomology}
Let $(M,\omega=d\lambda)$ be a Liouville domain of dimension $2n$ for some positive integer $n$. Its boundary $\partial M$ is a contact manifold with the contact form $\lambda|_{\partial M}$. Let us denote the corresponding Reeb vector field by $R$. Then the completion of $M$ is obtained from $M$ by gluing the symplectization $[1,\infty) \times \partial M$ with the identification $\partial M = \{1 \} \times \partial M$, i.e.,
$$ \widehat{M} = M \cup_{\partial M} ([1,\infty) \times \partial M).$$

Let $b\in H^2(M,\mathbb{Z}/2)$ be a background class and let us assume $2c_1(M) =0 \in H^2(M ,\mathbb{Z}).$

An exact Lagrangian submanifold $L$ of $M$ is called {\em admissible} if
\begin{itemize}
	\item $L$ intersects the boundary  $\partial M$ transversely,
	\item the restriction $\lambda|_L$ vanishes near the boundary $\partial L = L \cap \partial M$,
	\item $L$ is graded in the sense of \cite{sei08}, and
	\item the second Stiefel--Whitney class of $L$ coincides with $b|_L$.
\end{itemize}
These requirements force any admissible Lagrangian to have a cylindrical Legendrian end near the boundary of $M$. For a given admissible Lagrangian $L$, its completion is given by
$$ \widehat{L} = L \cup_{\partial L} ([1,\infty) \times \partial L).$$

We will recall the definition of Floer cohomology of pairs of admissible Lagrangians. For that purpose, let $L_0$ and $L_1$ be admissible Lagrangian submanifolds of $M$. Consider a Hamiltonian 
$H_{\mu} : \widehat{M} \to \mathbb{R}$ such that
\begin{itemize}
	\item $H_{\mu}$ is a $C^2$-small Morse function on $M$,
	\item $H_{\mu}(r, y) =  \mu( r-1)$ for $(r,y) \in [R, \infty) \times \partial M \subset \widehat{M}$
	for some $\mu \in \mathbb{R} \setminus \mathrm{Spec} (\partial L_0,\partial L_1)$ and some $1 \leq R \leq 2$, and
	\item all the Hamiltonian chords from $L_0$ to $L_1$ for $H_{\mu}$ are non-degenerate.
\end{itemize}
We will call such Hamiltonians {\em admissible}. Here the spectrum $\mathrm{Spec}(\partial L_0,\partial L_1)$ of periods of Reeb chords from $\partial L_0$ to $\partial L_1$ is given by
\begin{equation*}
	\mathrm{Spec} (\partial L_0, \partial L_1) = \{  T \in \mathbb{R} \,|\, \mathrm{Fl}_R^T (\partial L_0) \cap \partial L_1 \neq \emptyset\}, 
\end{equation*}
where $\mathrm{Fl}_R^T$ is the time $T$-flow of the Reeb vector field $R$ on $\partial M$. 

Then we consider the space $\mathcal{P}(L_0,L_1)$ of paths from $\widehat{L}_0$ to $\widehat{L}_1$:
$$\mathcal{P} (L_0,L_1) = \left\{ \gamma : [0,1] \to \widehat{M} \, \Big| \,\gamma(0) \in \widehat{L}_0, \gamma(1) \in \widehat{L}_1\right\}$$
and, for some $\mu \in \mathbb{R} \setminus \mathrm{Spec} (\partial L_0, \partial L_1)$, consider the {action functional} $\mathcal{A}_{H_{\mu}} : \mathcal{P} (L_0,L_1) \to \mathbb{R}$ defined by
\begin{equation}\label{eq:action}
	\mathcal{A}_{H_{\mu}} (\gamma) = - \int \gamma^* \lambda + \int_{0}^1 H_{\mu} (\gamma(t) )dt - h_{L_0}(\gamma(0)) + h_{L_1} (\gamma(1)),
\end{equation}
where $h_{L_0} :  \widehat{L}_0 \to \mathbb{R}$ and $h_{L_1} : \widehat{L}_1 \to \mathbb{R}$ are smooth functions that are locally constant outside a compact subset of $M$ such that
$$ \lambda|_{\widehat{L}_0} = dh_{L_0} \text{ and } \lambda|_{\widehat{L}_1} = dh_{L_1}.$$

It is a standard fact that critical points of the action functional $\mathcal{A}_{H_{\mu}}$ are Hamiltonian chords from $\widehat{L}_0$ to $\widehat{L}_1$ for the Hamiltonian $H_{\mu}$. Let us denote the set of such Hamiltonian chords by
\begin{equation*}\label{eq:hamiltonianchords}
	\mathcal{I} ( L_0,L_1;H_\mu).
\end{equation*}
Since the Lagrangians $L_0$ and $L_1$ are assumed to be graded, every Hamiltonian chord $x$ is graded by the Maslov index $\deg x \in \mathbb{Z}$. See \cite{sei08} for details.

Then we define the {\em Floer complex} $CF^*(L_0,L_1;H_{\mu})$ to be the graded $\mathbb{K}$-vector space generated by $\mathcal{I}(L_0,L_1;H_\mu)$. More precisely, we define
\begin{equation*}
	CF^*(L_0,L_1;H_{\mu}) \coloneqq \bigoplus_{x \in \mathcal{I}(L_0,L_1; H_{\mu})} |o_x|_{\mathbb{K}}[-\deg x], 
\end{equation*}
where $|o_x|_{\mathbb{K}}$ is the $\mathbb{K}$-normalized orientation space defined in \cite[Section 3.6]{abo-sei10}. We will just denote by $x$ a generator of $|o_x|_{\mathbb{K}}$ in the rest of this paper.

An $\omega$-compatible time-dependent almost complex structure $J= (J(t))_{t\in [0,1]}$ of contact type gives rise to an $L^2$-metric on the path space $\mathcal{P}(L_0,L_1)$. After choosing a regular time-dependent almost complex structure among those, the differential 
\begin{equation*}
	\delta_{\mu} :CF^*(L_0,L_1;H_{\mu}) \to CF^{*+1}(L_0,L_1;H_{\mu})
\end{equation*}
is defined by counting {\em Floer strips}. Indeed,  let
$$ x_0 \in \mathcal{I} (L_{0},L_{1}; H_{\tau_w}) \text{ and } x_1 \in \mathcal{I} (L_{0}, L_{1}; H_{\tau_w})$$
and consider the space of Floer strips 
\begin{equation}\label{eq:floerstrip}
	\mathcal{M}(x_0,x_1;H_\mu,J) = \{ u : \mathbb{R} \times [0,1] \to \widehat{M} \,|\, u \text{ satisfies } \mathrm{(i)}, \mathrm{(ii)}, \mathrm{(iii)} \text{ and } \mathrm{(iv)} \text{ below}\}/ \mathbb{R}
\end{equation}
\begin{enumerate}[label=(\roman*)]
	\item $u(s,j) \in \widehat{L}_{j}$ for all $s\in\mathbb{R}$ and for $j=0,1$.
	\item $\lim_{s\to -\infty} u(s,t) = x_0(t)$.
	\item $\lim_{s\to \infty} u(s,t) = x_1(t)$.
	\item $\partial_s u + J(t) (\partial_t u - X_{H_{\mu}}(u)) =0$.
\end{enumerate}
where $/ \mathbb{R}$ means the quotient by the natural $\mathbb{R}$-translation.

A standard argument shows that for a generic choice of the datum $(H_{\mu},J)$, the following hold:
\begin{itemize}
	\item Every Hamiltonian chord $x \in \mathcal{I} (L_{0},L_{1}; H_{\mu})$ is non-degenerate.
	\item the pair $(H_{\mu}, J)$ is regular in the sense that the associated Fredholm operator $D\partial_{H_{\mu},J}(u)$ is surjective for all $u\in \mathcal{M}(x_0,x_1;H_\mu,J)$.
\end{itemize}

As a consequence, the space $\mathcal{M}(x_0,x_1;H_\mu, J)$ is a smooth manifold of the dimension
$$ \deg x_0 - \deg x_1 -1$$
for such Floer data, which admits a compactification obtained by adding broken Floer strips. For example, for $x_0 \in \mathcal{I}(L_{0}, L_{1};H_\mu)$ and $x_1 \in \mathcal{I}(L_{0}, L_{1};H_{\mu})$ with 
$\deg x_0 = \deg x_1 +1$, then $\mathcal{M} (x_0,x_1;H_\mu, J)$ is a zero-dimensional compact smooth manifold. As argued in \cite[Section 3.6]{abo-sei10}, in such a case, every element $u$ of $\mathcal{M}(x_0,x_1;H_\mu,J)$ induces a preferred isomorphism
$$|o_{u}| : |o_{x_1}|_{\mathbb{K}} \to |o_{x_0}|_{\mathbb{K}}.$$

Finally we define the differential
\begin{equation}\label{eq:differential}
	\delta_{\mu} : CF^*(L_{0}, L_{1};H_{\mu}) \to CF^{*+1}(L_{0},L_{1};H_{\mu})
\end{equation}
by putting
$$ \delta_{\mu} \coloneqq \sum_{ u \in \mathcal{M}(x_0,x_1;H_{\mu},J)} |o_{u}|: |o_{x_1}|_{\mathbb{K}} \to |o_{x_0}|_{\mathbb{K}}$$
for Hamiltonian chords $x_0,x_1$ given as above and extending this linearly.

The {\em Floer cohomology} $HF^*(L_0,L_1;H_{\mu})$ is defined to be the cohomology
\begin{equation*}
	HF^*(L_0,L_1;H_{\mu}) \coloneqq H^*(CF^*(L_0,L_1;H_{\mu}), \delta_{\mu}).
\end{equation*}

\subsection{Wrapped Floer cohomology}\label{subsection:wrappedfloercohomology}
To define the wrapped Floer cohomology, one needs to take the direct limit of the Floer cohomology groups for admissible Hamiltonians whose slopes tend to infinity using continuation maps. Indeed, for any pair of admissible Hamiltonians $H_{\lambda}$ and $H_{\mu}$ on $\widehat{M}$ such that $\lambda \leq \mu$, there is a chain map 
$$c_{\lambda,\mu} : CF^*(L_0,L_1 ; H_{\lambda}) \to CF^*(L_0,L_1 ;H_{\mu}),$$
called a {\em continuation map}. It is defined by counting the gradient flows of $\mathcal{A}_{H_{s}}$ of index $0$ with respect to the $L^2$-metric discussed above for a 1-parameter family $\{H^s\}_{s\in \mathbb{R}}$ of admissible Hamiltonians satisfying
\begin{itemize}
	\item $H^s = \begin{cases} H_{\mu} &(s \leq -1), \\ H_{\lambda} & (s \geq 1) \end{cases}$ \text{ and}
	\item $\partial_s H^s  \leq 0$ outside a compact subset of $\widehat{M}$.
\end{itemize}
To be more precise, one needs to consider the space of stable popsicle maps as in \eqref{eq:continuationmoduli} to define it. See \eqref{eq:continuation} for a more detailed definition.

In particular, in \cite{abo-sei10}, the authors constructed the wrapped Floer complex of a pair of Lagrangians using a homotopy direct limit of continuation maps. Let us now recall its construction to introduce some notations.

For that purpose, we first consider {\em Floer data associated to pairs of admissible Lagrangians}. 
\begin{dfn}[Floer data associated to Lagrangian pairs]\label{dfn:floerdataforlagrangian}
	For any admissible Lagrangians $L_0$ and $L_1$ of $M$, a Floer datum associated to the pair $(L_0,L_1)$ consists of a collection $\{(\tau^{L_0,L_1}_{w}, H_{\tau^{L_0,L_1}_w},J^{L_0,L_1}_w)\}_{w\in \mathbb{Z}}$ of triples of 
	\begin{itemize}
		\item real numbers $\tau^{L_0,L_1}_w$,
		\item admissible Hamiltonians $H_{\tau^{L_0,L_1}_w}$ and
		\item time-dependent $\omega$-compatible almost complex structures $J^{L_0,L_1}_w = (J^{L_0,L_1}_w(t))_{t \in [0,1]}$ of contact type.
	\end{itemize}
\end{dfn}

Let $\{L_i\}_{i\in I}$ be an at most countable collection of admissible Lagrangians of $M$. Since $I$ is at most countable, we may assume that Reeb chords from $\partial L_i$ to $\partial L_j$ are non-degenerate for all pairs $i,j \in I$ by rescaling the contact form by a generic nonzero function on $\partial M$ if necessary. As a consequence, $\mathrm{Spec}(\partial L_i,\partial L_j)$ is discrete for every pair $i,j \in I$.

We require the sequence of real numbers $\{ \tau^{L_i,L_j}_w\}_{w \in\mathbb{Z}}$ for pairs $i,j \in I$ to satisfy the following:
\begin{enumerate}[label=(\alph*)]
	\item $\tau^{L_i,L_j}_w \notin \mathrm{Spec} (\partial L_i,\partial L_j)$ for all $w\in \mathbb{Z}$.
	\item $\tau^{L_i,L_j}_w < \tau^{L_i,L_j}_{w+1}$ for all $w \in \mathbb{Z}$ and $\tau^{L_i,L_j}_0 <0 < \tau^{L_i,L_j}_1$.
	\item $\lim_{w \to \pm \infty} \tau^{L_i,L_j}_w =\pm \infty$.
	\item $\tau^{L_j,L_k}_w + \tau^{L_i,L_j}_{v} \leq \tau^{L_i,L_k}_{v+w}$ for all $i,j,k \in I$  and all $v,w \in \mathbb{Z}$.
\end{enumerate}

Furthermore, we also require the pair $(H_{\tau^{L_i,L_j}_w},J^{L_i,L_j}_w)$ for any $i,j\in I$ to satisfy the following: 
\begin{enumerate}[label=(\alph*)]
\setcounter{enumi}{4}
\item Every Hamiltonian chord $x \in \mathcal{I} (L_{i},L_{j}; H_{\tau^{L_i,L_j}_w})$ is non-degenerate.
\item The pair $(H_{\tau^{L_i,L_j}_w}, J^{L_i,L_j}_w)$ is regular in the sense that the associated Fredholm operator is surjective.
\end{enumerate}
These two conditions allow us to consider the differential $\delta_{\tau_w}$ as in \eqref{eq:differential}.

	The following lemma shows that such a collection of Floer data associated to pairs $(L_i,L_j)$ of admissible Lagrangians exists.
	\begin{lem}\label{lem:existenceoffloerdata}
		There exists a collection $\{(\tau_w^{L_i,L_j},H_{\tau^{L_i,L_j}_w},J^{L_i,L_j}_w)\}_{w\in \mathbb{Z}, i,j \in I}$ satisfying the conditions {\em (a), (b), (c), (d), (e)} and {\em(f)}.
	\end{lem}
	\begin{proof}
		Let $S = \cup_{i,j \in I} \mathrm{Spec}(\partial L_i,\partial L_j) \subset \R$. Since $S$ is at most countable, there exists a positive real number $a$ such that $wa \in \R \setminus S$ for all $w\in \Z_{\geq 1}$. Indeed, otherwise, we have the equality
		$$\R_{>0} = \cup_{s\in S}\mathbb{Q}_{>0} s,$$
		where $\mathbb{Q}_{> 0}s = \{ qs\in \R |q\in \mathbb{Q}_{>0} \}$ for $s\in S \cap \R_{>0}$, which is absurd since $\R_{>0}$ is uncountable while the right hand side is at most countable.
		Similarly let us choose a sufficiently small $0<\epsilon<a$ such that $wa - \epsilon \in \R \setminus S$ for all $w\in \Z_{\leq -1}$. For every $i,j\in I$, we choose a real number $\tau_{0}^{L_i,L_j}$ such that
		$$ \tau_0^{L_i,L_j} \in (\R \setminus \mathrm{Spec}(L_i,L_j)) \cap (-\epsilon,0)$$
		and we put
		$$\tau_{w}^{L_i,L_j} = \begin{cases} wa & (w\in \mathbb{Z}_{\geq 1}), \\  wa - \epsilon& (w \in \mathbb{Z}_{\leq -1}).\end{cases}$$
		Then it is straightforward to see that the conditions (a), (b), (c) and (d) are satisfied.
		
		For each $i,j \in I$ and $w\in \Z$, the conditions (e) and (f) are satisfied for a generic choice of an admissible Hamiltonian $H_{\tau^{L_i,L_j}_{w}}$ with slope $\tau^{L_i,L_j}_w$ and a time-dependent $\omega$-compatible almost complex structure $J^{L_i,L_j}_w$ of contact type.
	\end{proof}

For the rest of this paper, let $\{L_i\}_{i \in I}$ be an at most countable family of admissible Lagrangians of $M$ and let us fix Floer data $\{(\tau^{L_i,L_j}_w, H_{\tau^{L_i,L_j}_w},J^{L_i,L_j}_w)\}_{w\in \mathbb{Z}}$ associated to all the pairs $i,j \in I$ in such a way that the above conditions (a), (b), (c), (d), (e) and (f) hold. For any $i,j \in I$, we will omit the superscript $(L_i,L_j)$ from the notation for the Floer data associated to the pair $(L_i,L_j)$ whenever the Lagrangian pair is clear from the context. For example, we will just denote the real number $\tau^{L_i,L_j}_{w}$ by $\tau_{w}$ for $w\in \mathbb{Z}$.

Let $L_0,L_1 \in \{L_i\}_{i\in I}$ be given. We define the wrapped Floer complex $CW^*(L_0,L_1)$ of $L_0$ and $L_1$ for the wrapped Floer cohomology $HW^*(L_0,L_1)$ following \cite{abo-sei10}.
First consider the increasing sequence $\{\tau_w \}_{w \in \mathbb{Z}_{\geq 1}} (\subset \{\tau_w\}_{w\in \mathbb{Z}})$ of positive real numbers, which is given as a part of the Floer data associated to the Lagrangian pair $(L_0,L_1)$. These numbers are assumed to satisfy
\begin{itemize}
	\item $\lim_{w\to \infty} \tau_w = \infty$ and
	\item $\tau_w \notin \mathrm{Spec} (\partial L_0, \partial L_1)$ for all $w \in \mathbb{Z}_{\geq 1}$.
\end{itemize}

Then, for each $w \in \mathbb{Z}_{\geq 1}$, we consider the continuation map
$$ c_{w} \coloneqq c_{\tau_w,\tau_{w+1}} : CF^*(L_0,L_1 ; H_{\tau_w}) \to CF^*(L_0,L_1 ; H_{\tau_{w+1}}),$$
which is defined precisely in \eqref{eq:continuation}.

We define the wrapped Floer complex of $L_0$ and $L_1$ by
\begin{equation*}
	CW^*(L_0,L_1) \coloneqq \bigoplus_{w =1}^{\infty} CF^*(L_0,L_1;H_{\tau_w}) \oplus \bigoplus_{w=1}^{\infty} CF^{*+1}(L_0,L_1;H_{\tau_w}) q,
\end{equation*}
where $q$ is a formal variable of degree $-1$ satisfying $q^2=0$. Then we define its differential $\delta$ by
\begin{equation}\label{eq:differential1}
	\begin{split}
		\delta(a) &= (-1)^{\deg a}\delta_{\tau_w} (a),\\
		\delta(bq) &= (-1)^{\deg b} (c_w(b) - b +\delta_{\tau_w}(b) q),
	\end{split}
\end{equation}
for all homogeneous elements $a,b \in CF^*(L_0,L_1;H_{\tau_w})$.

We may represent the differential $\delta$ (up to sign) as follows:
\begin{equation*}
\begin{tikzcd}[column sep=tiny]
CF^*(L_0,L_1;H_{\tau_1}) \ar[out=120,in=60,loop,"\delta_{\tau_1}"] & CF^*(L_0,L_1;H_{\tau_2}) \ar[out=120,in=60,loop,"\delta_{\tau_2}"] & CF^*(L_0,L_1;H_{\tau_3}) \ar[out=120,in=60,loop,"\delta_{\tau_2}"] & \cdots\\
CF^*(L_0,L_1;H_{\tau_1})q \ar[u,"\mathrm{Id}"] \ar[ur,"c_1"] \ar[out=240,in=300,loop,swap,"\delta_{\tau_1}"] & CF^*(L_0,L_1;H_{\tau_2})q \ar[u,"\mathrm{Id}"] \ar[ur,"c_2"] \ar[out=240,in=300,loop,swap,"\delta_{\tau_2}"] & CF^*(L_0,L_1;H_{\tau_3})q \ar[u,"\mathrm{Id}"] \ar[ur,"c_3"] \ar[out=240,in=300,loop,swap,"\delta_{\tau_3}"] & \cdots.
\end{tikzcd}
\end{equation*}

The wrapped Floer cohomolgy of $L_0$ and $L_1$ is defined by the cohomology
\begin{equation*}
	HW^*(L_0,L_1) \coloneqq H^*(CW^*(L_0,L_1) ,\delta).
\end{equation*}

As proved in \cite[Lemma 3.12]{abo-sei10}, we have
\begin{lem}\label{lem:directlimit} There is an isomorphism
	\begin{equation*}
		HW^*(L_0,L_1) \cong \varinjlim_{w\to \infty} HF^*(L_0,L_1;H_{\tau_w}),
	\end{equation*}
	where the latter is the direct limit of the direct system $\{[c_{v-1} \circ \dots \circ c_w] = [c_{\tau_w,\tau_v}] : HF^*(L_0,L_1; H_{\tau_w})  \to HF^*(L_0,L_1;H_{\tau_{v}})\}_{w<v}$.
\end{lem}

\subsection{Wrapped Floer homology}
Now we turn to the construction of the {\em wrapped Floer homology} $HW_*(L_0,L_1)$. As above, consider a decreasing sequence $\left\{\tau_{-w} \right\}_{w \in \mathbb{Z}_{\geq 0}} \left( \subset \left\{ \tau\\
_{w}\right\}_{w\in \mathbb{Z}}\right)$ of negative real numbers, which is also given as a part of the Floer data associated to the Lagrangian pair $(L_0,L_1)$. These numbers are assumed to satisfy
\begin{itemize}
	\item $\lim_{w\to \infty} \tau_{-w} =- \infty$ and
	\item $\tau_{-w} \notin \mathrm{Spec} (\partial L_0,\partial L_1)$ for all $w \in \mathbb{Z}_{\geq 0}$.
\end{itemize}

For each $w \in \mathbb{Z}_{\geq 1}$, we consider the continuation map
$$ c_{-w} \coloneqq c_{\tau_{-w},\tau_{-(w-1)}} : CF^*(L_0,L_1 ; H_{\tau_{-w}}) \to CF^*(L_0,L_1 ; H_{\tau_{-(w-1)}})$$
and, in particular, for $w=0$, we put
$$ c_0 =0.$$

Then we define the wrapped Floer complex $CW_*(L_0,L_1)$ for the wrapped Floer homology by
\begin{equation*}\label{eq:wrappedfloercomplex2}
	CW_*(L_0,L_1) \coloneqq \prod_{w=0}^{\infty} CF^{n-*}(L_0,L_1;H_{\tau_{-w}}) \oplus \prod_{w=0}^{\infty} CF^{n-1-*}(L_0,L_1;H_{\tau_{-w}}) {q}^{\vee},
\end{equation*}
where ${q}^{\vee}$ is a formal variable of degree $1$.
To be consistent with the cohomological degree we discussed above, we say that the degree is given by
$$ \deg x = n-k $$
for all nonzero elements $x \in CW_{k}(L_0,L_1)$, $k \in \mathbb{Z}$.

For a sequence $\{a_w \in CF^{n-*}(L_0,L_1;H_{
\tau_{-w}})\}_{w\in \Z_{\geq 0}}$, we denote by $(a_w)_w \in \prod_{w=0}^{\infty} CF^{n-*}(L_0,L_1;H_{\tau_{-w}})$ the element whose $w$-th entry is given by $a_w$. 
Then we define the differential $\partial$ by
\begin{equation}\label{eq:differential2}
	\begin{split}
		\partial\left((a_w)_w\right) &= ((-1)^{\deg a_w}\delta_{\tau_w} (a_w))_w + ((-1)^{\deg a_{w+1}}c_{-(w+1)}(a_{w+1}) - (-1)^{\deg a_w} a_w)_w q^{\vee},\\
		\partial\left((b_w)_w  q^{\vee}\right) &= ( (-1)^{\deg b_w}\delta_{\tau_w}(b_w))_w {q}^{\vee},
	\end{split}
\end{equation}
for all homogeneous elements $a_w,b_w \in CF^*(L_0,L_1;H_{\tau_{-w}})$.

Once again, we may represent the differential $\partial$ (up to sign) as follows:
\begin{equation*}
\begin{tikzcd}[column sep=tiny]
\cdots  & CF^{n-*}(L_0,L_1;H_{\tau_{-2}})q^\vee \ar[out=120,in=60,loop,"\delta_{\tau_{-2}}"] & CF^{n-*}(L_0,L_1;H_{\tau_{-1}})q^\vee \ar[out=120,in=60,loop,"\delta_{\tau_{-1}}"] & CF^{n-*}(L_0,L_1;H_{\tau_0})q^\vee. \ar[out=120,in=60,loop,"\delta_{\tau_0}"]\\
\cdots  \ar[ur,"c_{-3}"]  & CF^{n-*}(L_0,L_1;H_{\tau_{-2}}) \ar[u,"\mathrm{Id}"] \ar[ur,"c_{-2}"] \ar[out=240,in=300,loop,swap,"\delta_{\tau_{-2}}"] & CF^{n-*}(L_0,L_1;H_{\tau_{-1}}) \ar[u,"\mathrm{Id}"] \ar[ur,"c_{-1}"] \ar[out=240,in=300,loop,swap,"\delta_{\tau_{-1}}"] & CF^{n-*}(L_0,L_1;H_{\tau_{0}}) \ar[u,"\mathrm{Id}"] \ar[out=240,in=300,loop,swap,"\delta_{\tau_{0}}"]
\end{tikzcd}
\end{equation*}

The wrapped Floer homology of $L_0$ and $L_1$ is defined by the homology
\begin{equation*}
	HW_*(L_0,L_1) \coloneqq H_*(CW_*(L_0,L_1) ,\partial).
\end{equation*}

Considering the dual argument of \cite[Lemma 3.12]{abo-sei10}, we have
\begin{lem} \label{lem:inverselimit} There is an isomorphism
	\begin{equation*}
		HW_*(L_0,L_1) \cong \varprojlim_{ w \to \infty} HF^{n-*}(L_0,L_1;H_{\tau_{-w}}),
	\end{equation*}
	where the latter is the inverse limit of the inverse system $\{ [c_{-(v+1)} \circ \dots \circ c_{-w}]= [c_{\tau_{-w},\tau_{-v}}]\}_{v <w}$.
\end{lem}
\begin{proof}
	Observe that, for each $w \in \mathbb{Z}_{\geq 0}$, 
	$$C_w \coloneqq\prod_{u=w+1}^{\infty} CF^{n-*}(L_0,L_1;H_{\tau_{-u}}) \oplus \prod_{u=w}^{\infty} CF^{n-1-*}(L_0,L_1;H_{\tau_{-u}}) {q}^{\vee} $$
	is a subcomplex of $CW_*(L_0,L_1)$. We consider the quotient complex
	$$Q_w = CW_*(L_0,L_1)/C_w.$$
	Since $C_{w+1} \subset C_w$ for each $w\in \mathbb{Z}_{\geq 0}$, there is a natural quotient map $Q_{w+1} \to Q_w$. Moreover these fit into the following commutative diagram
	$$ \xymatrix{CW_*(L_0, L_1)\ar[d] \ar[dr]&\\ Q_{w+1}\ar[r] &Q_w,}$$
	where all the arrows are quotient maps.
	By definition of $CW_*(L_0,L_1)$, this allows us to identify
	\begin{equation*}\label{eq:inverselimit}
		CW_*(L_0,L_1) = \varprojlim_{w\to \infty} Q_w,
	\end{equation*}
	where the latter is the inverse limit for the quotient maps $Q_{w+1} \to Q_w$. Since each of $Q_w$ is finite-dimensional, so is $H_k(Q_w)$ for each degree $k \in \mathbb{Z}$. 
	Therefore both the towers $\{Q_w\}_{w\in \Z_{\geq 0}}$ and $\{H_k(Q_w) \}_{w\in \mathbb{Z}_{\geq 0}}$ satisfy the Mittag-Leffler condition \cite{wei94}, which implies that
	 \begin{equation}\label{eq:homologyinverselimit}
	 	HW_*(L_0,L_1) = H_*(\varprojlim_{w\to \infty} Q_w)\cong \varprojlim_{w\to \infty} H_*(Q_w).
	 \end{equation}

	Next, for each $v\in \mathbb{Z}_{\geq -1}$, consider another subcomplex $D^v$ of $CW_{*}(L_0,L_1)$ defined by
	\begin{equation*}
  D^v = \begin{cases} 0 & (v=-1),\\ 
  \prod_{u=0}^{v} CF^{n-*} (L_0,L_1;H_{\tau_{-u}}) \oplus \prod_{u=0}^{v} CF^{n-1-*}(L_0,L_1;H_{\tau_{-u}}) {q}^{\vee} & (v \geq 0).\end{cases}
	\end{equation*}
	Then, we consider the quotient complex
	$$Q_w^v = CW_*(L_0,L_1)/ (C_w + D^v).$$

	Since $C_w+D^v$ is a subcomplex of $C_w +D^{v+1}$ for each $-1 \leq v <v+1<w$, there is a quotient map
	$$\phi_w^v: Q_w^v \to Q_w^{v+1}$$
	by the subcomplex $(C_w + D^{v+1})/(C_w +D^v)$ of $Q_w^v$.
	The quotient map $\phi_w^v$ is a quasi-isomorphism since the subcomplex $(C_w + D^{v+1})/(C_w +D^v) \cong D^{v+1}/D^v \cong \mathrm{Cone} ( \mathrm{Id} : CF^{n-*}(L_0,L_1;H_{\tau_{-(v+1)}}) \to  CF^{n-*}(L_0,L_1;H_{\tau_{-(v+1)}}))[-1]$ is acyclic.
		
	As a result, we have a sequence of quasi-isomoprhisms
	$$ Q_w=Q_w^{-1} \cong \cdots \cong Q_w^{w-1} \cong CF^{n-*}(L_0,L_1;H_{\tau_{-w}}).$$
	Let us denote by $\phi_w$ the resulting quasi-isomorphism
	$$\phi_w : Q_w \to CF^{n-*}(L_0,L_1;H_{\tau_{-w}}).$$
	Now observe that the following diagram commutes up to a chain homotopy:
	$$ \xymatrix@C45pt{Q_{w+1}\ar[r]\ar[d]_{\phi_{w+1}} & Q_{w} \ar[d]^{\phi_w}\\ CF^{n-*}(L_0,L_1;H_{\tau_{-(w+1)}}) \ar[r]^{\quad c_{-(w+1)}\quad}  &  CF^{n-*}(L_0,L_1;H_{\tau_{-w}}).}$$
	This observation combined with \eqref{eq:homologyinverselimit} proves the assertion.
\end{proof}
	\begin{rmk}
		We remark that our degree convention for wrapped Floer homology is different from that in \cite{ggv22}. We adopt this convention to have a duality between $HW^k (L_0,L_1)$ and $HW_k(L_1,L_0)$ for admissible Lagrangians $L_0$, $L_1$ and an integer $k \in \mathbb{Z}$. This will be discussed in Lemma \ref{lem:wrappednondegenerate} in more datails.
	\end{rmk}

\subsection{Rabinowitz wrapped Fukaya category}\label{subsection:rabinowitz}
We follow the construction in \cite{ggv22} in the sense that we define the {\em Rabinowitz (wrapped) Floer complex} of two given admissible Lagrangians $L_0$ and $L_1$ by the cone of a continuation map from $CW_{n-*}(L_0,L_1)$ to $CW^*(L_0,L_1)$. Namely, for a given two admissible Lagrangians $L_0$ and $L_1$ in the collection $\{L_i\}_{i\in I}$, $RFC^*(L_0,L_1)$ is given by
$$RFC^*(L_0,L_1) = \mathrm{Cone} ( c: CW_{n-*}(L_0,L_1) \to CW^*(L_0,L_1)).$$
However, our approach is different from theirs in that our Floer complexes $CW_{n-*}(L_0,L_1)$ and $CW^*(L_0,L_1)$ are defined using linear Hamiltonians and taking the homotopy direct/inverse limit as in \cite{abo-sei10} and \cite{ven18}.

We will further construct $A_\infty$-products on the Rabinowitz Floer complexes following the idea of \cite{abo-sei10,ggv22}. Indeed, in \cite{abo-sei10}, the authors constructed the wrapped Floer complexes using the homotopy direct limit and defined $A_{\infty}$-products on those by considering stable popsicle maps.

To be more concrete, our main goals of this subsection are
\begin{enumerate}
	\item to construct the Rabinowitz Floer complex $RFC^*(L_0,L_1)$ for any pair $(L_0,L_1)$ of admissible Lagrangians, and
	\item for any integer $k\geq 2$ and any tuple $(L_0,\dots,L_k)$ of admissible Lagrangians, to construct a product structure
	\begin{equation*}
		\widecheck{\mu}^k : RFC^*(L_{k-1}, L_k)\otimes \dots \otimes RFC^*(L_0,L_1) \to RFC^*(L_0,L_k)[2-k]
	\end{equation*}
	satisfying the $A_{\infty}$-equation and an additional property that natural inclusions $CW^* \hookrightarrow RFC^*$ induce an $A_{\infty}$-homomorphism, which will be discussed in Theorem \ref{thm:ainfinity}.
\end{enumerate}
 Assuming that these two have been already done, we define the Rabinowitz (wrapped) Fukaya category $\mathcal{RW}_b(M)=\mathcal{RW}(M)$ of $M$ to be an $A_{\infty}$-category whose objects are admissible Lagrangians of $M$ in $\{L_i\}_{i\in I}$, morphism spaces are given by $RFC^*$ and $A_{\infty}$-products are given by $\widecheck{\mu}^k$.



\subsubsection{Rabinowitz Floer complex}
Let us first construct the Rabinowitz Floer complex. Let $L_0$ and $L_1$ be admissible Lagrangians chosen from the collection $\{L_i\}_{i \in I}$. We assume that the wrapped Floer complexes $CW_*(L_0,L_1)$ and $CW^*(L_0,L_1)$ are defined for the sequence $\{\tau_w \}_{w\in \mathbb{Z}}$ of real numbers chosen in Lemma \ref{lem:existenceoffloerdata}.

Let
$$ \iota : CF^*(L_0,L_1;H_{\tau_1})\to CW^*(L_0,L_1),$$
be the inclusion map and let
\begin{equation}\label{eq:projection}
	\pi : CW_{n-*}(L_0,L_1) \to CF^*(L_0,L_1; H_{\tau_0}),
\end{equation}
be the projection map (or the quotient map onto $Q_0 \cong CF^*(L_0,L_1;H_{\tau_0})$).
If we endow $CF^*(L_0,L_1;H_{\tau_w})$ with a new differential $\delta'_w$ given by
$$ \delta'_w(x) = (-1)^{\deg x} \delta_w(x)$$
for all homogeneous elements $x$, then both $\iota$ and $\pi$ are chain maps.

Moreover, the continuation from $CF^*(L_0,L_1;H_{\tau_0})$ to $CF^*(L_0,L_1;H_{\tau_1})$ defined precisely in 
\eqref{eq:continuation}
$$ c_{\tau_0, \tau_1} : CF^*(L_0,L_1; H_{\tau_0}) \to CF^*(L_0,L_1 ; H_{\tau_1})$$ 
is a chain map with respect to the new differential map.

Then we consider a chain map
\begin{equation*}
	c: CW_{n-*}(L_0,L_1) \to CW^*(L_0,L_1)
\end{equation*}
given by
$$c(b) =  (-1)^{\deg b +1} \iota \circ c_{\tau_0,\tau_1} \circ \pi(b)$$
for all homogeneous elements $b \in CW_{n-*}(L_0,L_1)$. Here, the degree is defined by $\deg b = k$ for all nonzero elements $b \in CW_{n-k}(L_0,L_1)$ for $k\in \mathbb{Z}$.

Since all the factors of $c$ are chain maps, so is $c$. This fits into the following commutative diagram:
\begin{equation*}
	\xymatrix@C60pt{CW_{n-*}(L_0,L_1) \ar[r]^{\quad c} \ar[d]_{\pi}& CW^*(L_0,L_1) \\ CF^{*}(L_0,L_1;H_{\tau_0}) \ar[r]_{(-1)^{\deg (\cdot) +1}c_{\tau_0,\tau_1} }& CF^*(L_0,L_1;H_{\tau_1}). \ar[u]_{\iota}  }
\end{equation*}

As suggested in \cite{ggv22}, we define the Rabinowitz Floer complex $RFC^*(L_0,L_1)$ by the mapping cone
\begin{equation}\label{eq:rabinowitzcomplex}
	RFC^*(L_0,L_1) = \mathrm{Cone} ( c: CW_{n-*}(L_0,L_1) \to CW^*(L_0,L_1)).
\end{equation}
Since the Rabinowitz Floer complex is defined by the mapping cone of a chain map, it naturally inherits a differential, which we will denote by $\widecheck{\mu}^1$. In other words, identifying $RFC^*(L_0,L_1) = CW_{n-*}(L_0,L_1)[1] \oplus CW^*(L_0,L_1)$ as vector spaces, the differential $\widecheck{\mu}^1$ is represented by the following matrix:
\begin{equation}\label{eq:rabinowitzdifferential}
	\widecheck{\mu}^1 = \begin{pmatrix} -\partial & 0\\ c & \delta \end{pmatrix},
\end{equation}
where $\partial$ and $\delta$ are as in \eqref{eq:differential2} and \eqref{eq:differential1}.

Then we define the {\em Rabinowitz Floer cohomology} ${RFH}^*(L_0,L_1)$ of $L_0$ and $L_1$ by its cohomology:
\begin{equation*}
	RFH^*(L_0,L_1) = H^*(RFC^*(L_0,L_1), \widecheck{\mu}^1).
\end{equation*}

As a consequence, there exists a short exact sequence of chain complexes
\begin{equation}\label{eq:ses}
	0\to CW^{*}(L_0,L_1)\xrightarrow{i} RFC^*(L_0,L_1) \xrightarrow{{p}} CW_{n-1-*}(L_0,L_1) \to 0.
\end{equation}
This induces the long exact sequence \eqref{eq:openles}:
\begin{equation*}
	\cdots \to HW_{n-*} (L_0,L_1) \to HW^*(L_0,L_1) \to {RFH}^*(L_0,L_1) \to HW_{n-1-*}(L_0,L_1) \to HW^{*+1}(L_0,L_1) \to \cdots.
\end{equation*}

\subsubsection{Moduli spaces of popsicles}\label{subsection:popsicle}
In order to construct higher products $\widecheck{\mu}^k$ on Rabinowitz Floer complexes, we make use of the notion of popsicle, which was first introduced in \cite{abo-sei10}. Let us recall it following \cite[Section 2.3 and Section 2.4]{abo-sei10}.

Let $k$ be a positive integer. Then let $F$ be a finite set and let $p: F \to \{1,\dots, k\}$ be a function, which we call a {\em flavor}.
{\em A $p$-flavored popsicle} $P = (S, \mathbf{\phi})$ consists of
\begin{itemize}
\item the unit disk $S$ (equipped with the standard complex structure) with $(k+1)$-boundary punctures $z_0,\dots,z_k$ which are cyclically ordered counterclockwise, and one distinguished puncture $z_0$ called a {\em root} and
\item a family $\mathbf{\phi} = \{ \phi_f \}_{f \in F}$ of holomorphic embeddings $\phi_f : S  \to \R \times [0,1]$, each of which extends to a bi-holomorphism $\overline{\phi}_f : \overline{S} \setminus \{z_0,z_{p(f)}\} \to \R \times [0,1]$ sending $z_0$ to $-\infty$ and $z_{p(f)}$ to $\infty$, respectively.
\end{itemize}

Let $P= (S,\phi = \{ \phi_f\}_{f \in F})$ be a $p$-flavored popsicle. For each $f\in F$, a holomorphic embedding $\phi_f : S \to \R \times [0,1]$ that extends to a bi-holomorphism $\overline{\phi}_f : \overline{S} \setminus \{z_0, z_{p(f)}\} \to \R \times [0,1]$ 
as described above is uniquely determined by the pre-image $\phi_f^{-1} (0, \frac{1}{2}) \in S$. The pre-image $\phi_f^{-1}(\R \times \{\frac{1}{2}\})$ is the geodesic between $z_0$ and $z_{p(f)}$ with respect to the hyperbolic metric on the interior of $S$ regarded as the Poincar\'{e} disk, which we will call the {\em popsicle stick} associated to $f$. Consequently, the pre-image $\phi_f^{-1} (0,\frac{1}{2})$ lies on the popsicle stick associated to $f$. Conversely, every point on the geodesic between $z_0$ and $z_{p(f)}$ uniquely determines such a holomorphic embedding $\phi_f : S \to \R \times [0,1]$. In conclusion, the popsicle $P=(S,\phi)$ can be understood as a unit disk with $(k+1)$-boundary punctures equipped with interior markings $\{s_f\}_{f\in F}$, each of which lies on the geodesic between $z_0$ and $z_{p(f)}$. We call such an interior marking $s_f$ a {\em sprinkle}. This description justifies that a popsicle $P = (S,\phi = \{ \phi_f\}_{f \in F})$ with $(k+1)$-boundary punctures is called {\em stable} if $k+ |F| \geq 2.$ See Figure \ref{figure1} for an example of a popsicle.
\begin{center}
\begin{figure}[h!]\label{figure1}
  \includegraphics[width=2in]{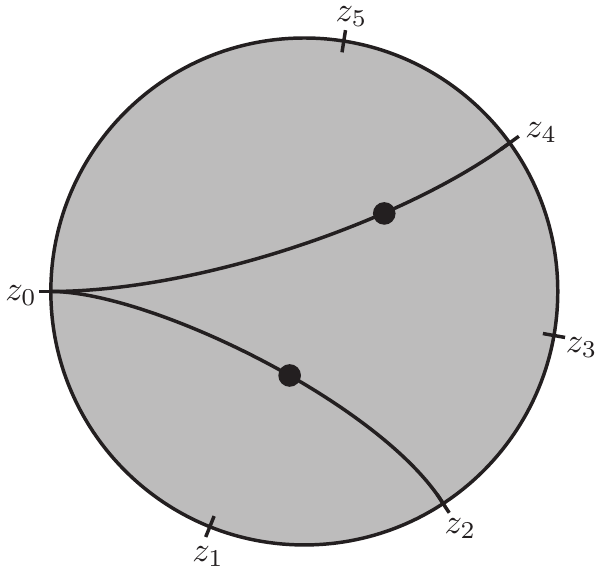}
  \caption{An element of $\mathcal{R}^{6, \{2,4\}}$}
  \label{figure1}
\end{figure}
\end{center}

For a given flavor  $p:F \to \{1,\dots,k\}$, we say that two $p$-flavored popsicles $P=(S,\phi = \{\phi_f\}_{f})$ and $P'=(S',\phi'= \{ \phi'_f\}_{f})$ are isomorphic if there is a bi-holomorphism $\psi : S \to S'$ such that $\phi'_f = \psi \circ \phi_f$ for all $f\in F$.  Let us denote by $\mathcal{R}^{k+1, p}$ the moduli space of stable $p$-flavored popsicles. 

We consider the subgroup of the permutation group of $F$ preserving the flavor $p$:
	\begin{equation*}\label{eq:symmetry}
		\mathbf{Sym}^p =\{ s: F \to F| s \text{ is a bijection satsifying } p \circ s =p \}.
	\end{equation*}
Note that the group $\mathbf{Sym}^p$ is trivial if and only if $p$ is injective or equivalently there is at most one sprinkle on each popsicle stick. The group $\mathbf{Sym}^p$ acts on $\mathcal{R}^{k+1, p}$ by
\begin{equation}\label{eq:action}
s \cdot (S, \{\phi_f\}_{f\in F}) = (S, \{\phi_{s(f)}\}_{f\in F}), \forall s \in \mathbf{Sym}^p, (S,\{\phi_f\})\in \mathcal{R}^{k+1,p}.
\end{equation}
In terms of sprinkles, this action permutes sprinkles on the same popsicle stick.

The moduli space $\mathcal{R}^{k+1,p}$ admits a compactification $\overline{\mathcal{R}}^{k+1,p}$ obtained by adding broken stable popsicles as argued in \cite[Section 2.3 and Section 2.4]{abo-sei10}. To describe it, let $(T,\mathbf{p})$ be a pair consisting of
\begin{itemize}
\item a rooted ribbon tree $T$ with $(k+1)$ semi-infinite ends and
\item a collection $\mathbf{p} = \{ p_v\}_{v \in V(T)}$ of flavors $p_v : F_v \to \{1,\dots, \mathrm{val}(v) -1\}$, where $F = \bigsqcup_{v \in V(T)} F_v$ is a partition of $F$ such that 
\begin{equation}\label{eq:stability}
\mathrm{val}(v)  + |F_v| \geq 3
\end{equation}
for every vertex $v$ of $T$. Here $V(T)$ is the set of vertices of $T$ and $\mathrm{val}(v)$ is the valency of the vertex $v$ in $T$.
\end{itemize}
The set $E(T)$ of edges of $T$ is the union of the set $E^{\mathrm{fin}}(T)$ of finite edges of $T$ and the set $E^{\mathrm{inf}}(T)$ of semi-infinite ends of $T$. Let us label the set $E^{\mathrm{inf}}(T)$ of semi-infinite ends (edges) of $T$ by $\{0,\dots, k\}$, where $0$ is distinguished to be a root and the other semi-infinite ends are ordered according to the cyclic order determined by the ribbon structure on $T$. We call a pair $(v,e)$ consisting of a vertex $v\in V(T)$ and an adjacent edge $e\in E(T)$ a {\em flag}. For each $v\in V(T)$, there are exactly $\mathrm{val}(v)$ flags that involve $v$, each of which corresponds to an edge of $T$ adjacent to $v$. Unless specified otherwise, we denote the set of the edges of $T$ adjacent to $v$ by $\{0,\dots, \mathrm{val}(v)-1\}$, where $0$ is the unique edge toward the root of $T$ and the others are ordered according to the cyclic order determined by the ribbon structure on $T$.
 
Now we associate a {\em broken stable popsicle} to each tuple $(P_v)_{v\in V(T)} \in \prod_{v\in V(T)} \mathcal{R}^{\mathrm{val}(v), p_v}$ of $p_v$-flavored popsicles $P_v=(S_v, \phi_v = \{ \phi_{v,f}\}_{f \in F_v})$ with $\mathrm{val}(v)$ boundary punctures $\{z_{v,0}, \dots, z_{v,\mathrm{val}(v)-1}\}$ in the following three steps:
\begin{enumerate}
\item Take the disjoint union of the closures $\overline{S}_v$ for all $v\in V(T)$.
\item For each finite edge $e$ involved in two distinct flags $(u,e)$ and $(v,e)$ for some vertices $u,v$ of $T$, glue two closed disks $\overline{S}_u$ and $\overline{S}_v$  by identifying the boundary point $z_{u,i(e)}$ of $\overline{S}_u$ and that $z_{v,j(e)}$ of $\overline{S}_v$, where $i(e) \in \{0,\dots,\mathrm{val}(u)-1\}$ and $j(e) \in \{0,\dots,\mathrm{val}(v)-1\}$ are the indices corresponding to the flags $(u,e)$ and $(v,e)$, respectively.
\item For each semi-infinite edge $e$ involved in a flag $(v,e)$ for some vertex $v$ of $T$, remove the boundary point of $\overline{S}_v$ corresponding to the flag.
\end{enumerate}
The condition \eqref{eq:stability} is required for the popsicle $P_v$ corresponding to the vertex $v$ to be stable. The broken stable popsicle obtained this way defines an element of $\mathcal{R}^{T,\mathbf{p}}$. The notion of isomorphism between two broken stable popsicle can be defined analogously to that of isomorphism between two stable popsicles.

As seen above, a stable popsicle is nothing but a disk with interior markings moving along geodesics between the root and another boundary puncture, which is stable in the sense the number of special points on it is greater than or equal to 3. Having this in mind, every broken stable popsicle can be understood as an element of the Deligne--Mumford compactification of $\mathcal{R}^{k+1,p}$.
Therefore, $\mathcal{R}^{k+1,p}$ can be compactified as follows:
\begin{equation*}
	\overline{\mathcal{R}}^{k+1,p} = \bigsqcup_{(T,\mathbf{p})} \mathcal{R}^{T,\mathbf{p}},
\end{equation*}
where the sum runs over all pairs $(T,\mathbf{p})$ given as above. Indeed, for each pair $(T,\mathbf{p})$ given as above, $\mathcal{R}^{T,\mathbf{p}}$ forms a stratum of $\overline{\mathcal{R}}^{k+1,p}$ of codimension $|V(T)|-1 =|E^{\mathrm{fin}}(T)|$. Moreover, for any such a pair $(T,\mathbf{p})$, there is an open set $\mathcal{U}^{T,\mathbf{p}} \subset \mathcal{R}^{T,\mathbf{p}} \times \prod_{e \in E^{\mathrm{fin}}(T)} (-1,0]$ and a gluing map
\begin{equation*}
	\mathcal{U}^{T,\mathbf{p}} \xrightarrow{\mathrm{glue}^{T,\mathbf{p}}} \overline{\mathcal{R}}^{k+1,p}
\end{equation*}
identifying 
$$ \mathcal{R}^{T,\mathbf{p}} \times \prod_{e \in E^{\mathrm{fin}}(T)} \{0\} \xrightarrow{\simeq} \mathcal{R}^{T,\mathbf{p}} \subset \overline{\mathcal{R}}^{k+1,p}.$$\\

\begin{figure}[h]
\includegraphics[scale=0.5]{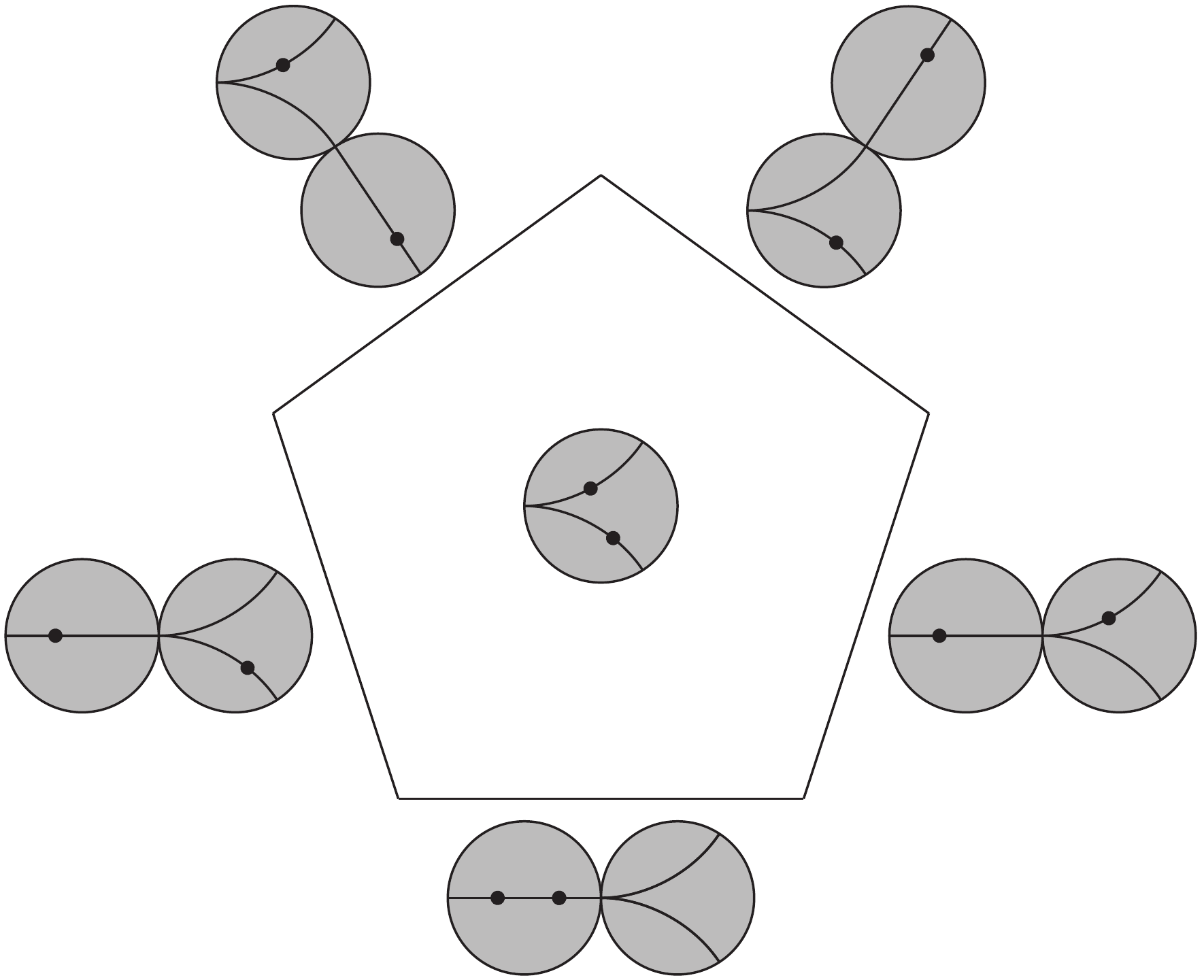}
\caption{The compactified popsicle moduli space $\overline{\mathcal{R}}^{3, \{1,2\}}$}
\centering
\label{fig:popdegen}
\end{figure}

When we define the higher products $\widecheck{\mu}^k$ in the next subsection, we only make use of popsicles, for which $\mathbf{Sym}^p$ is trivial or equivalently, $p : F \to \{1,\dots, k\}$ is injective. If a flavor $p : F \to \{1,\dots,k\}$ is injective, then we may regard $F$ as a subset of $\{1,\dots, k\}$ by assuming that $p$ is an inclusion. This way, for each subset $F$ of $\{1,\dots, k\}$, we can consider the notion of $F$-flavored popsicle. For each subset $F \subseteq \{1,\dots, k\}$, we denote by $\mathcal{R}^{k+1,F}$ the moduli space of $F$-flavored popsicles. The dimension of the moduli space $\mathcal{R}^{k+1,F}$ is given by
$$ k -2 + |F|.$$
This can be deduced from the fact that each of boundary punctures and sprinkles contributes to the dimension by $1$ and the dimension of the group of automorphisms of the disk is $3$.

For later use, let us look into the codimension-one strata of $\overline{\mathcal{R}}^{k+1,F}$ for each subset $F$ of $\{1,\dots, k\}$ more closely. As observed above, the codimension-one strata of $\overline{\mathcal{R}}^{k+1,F}$ are given by the disjoint union of $\mathcal{R}^{T, \mathbf{p}}$ for some rooted ribbon tree $T$ with exactly two vertices $\{v_0,v_1\}$, one finite edge $e$ and some family $\mathbf{p} = \{p_{v_0},p_{v_1}\}$ of flavors $p_{0} :F_{v_0}  \to \{1,\dots, \mathrm{val}(v_0)-1 \}$, $p_{1} : F_{1} \to \{1,\dots,\mathrm{val}(v_1)-1 \}$ such that \eqref{eq:stability} holds for both $v_0$ and $v_1$. Here $v_0$ is the vertex closest to the root $0$ in $T$.

To be more concrete, suppose that there is a sequence $\{P^m = (S^m,\phi^m) \}_{m \in \Z_{\geq 1}}$ of elements of $\mathcal{R}^{k+1,F}$ that converges to a broken stable popsicle $(P_{v_0}, P_{v_1}) \in \mathcal{R}^{T, \mathbf{p}}$. 
For convenience, for each $m \in \Z_{\geq 1}$, we denote the $(k+1)$ boundary punctures on $S^m$ by $\{z^m_{0}, \dots, z^m_{k}\}$ and denote $|F|$ sprinkles on $S^m$ by $\{s^m_{f}\}_{f\in F}$. The convergence means that there are some $0\leq i <j \leq k$ and a subset $F_{1} \subseteq \{i+1,\dots, j\} \cap F$ such that the consecutive boundary punctures $\{z^m_{i+1}, \dots, z^m_{j}\}$ and the sprinkles $\{s^m_{f}\}_{f\in F_{1}}$  get closer as $m$ goes to infinity so that eventually
\begin{itemize}
\item a stable popsicle $P_{v_1} = (S_{v_1}, p_{v_1})$ with $\mathrm{val}(v_1) = j - i +1$-boundary punctures bubbles off and
\item another stable popsicle $P_{v_0} = (S_{v_0}, p_{v_0})$ with $\mathrm{val}(v_0) = k +2 - (j-i)$-boundary punctures is left.
\end{itemize}
Consequently, it makes sense to denote the set of boundary punctures of $P_{v_1}$ by $\{0,i+1,\dots,j\}$ and that of $P_{v_0}$ by $\{0,1,\dots,i,*,j+1,\dots,k\}$, where $*$ is the unique boundary puncture of $P_{v_0}$ connected to the root $0$ of $P_{v_1}$. Moreover, we may regard $F_{1}$ as a subset of $F$ and we may put $F_{0} = F \setminus F_{1}$. 

We deduce that, for every $f\in F_{1}$, the corresponding sprinkle remains in $P_{v_1}$, while, for every $f\in (\{i+1,\dots, j\} \cap F) \setminus F_{1} = \{i+1,\dots, j\}\cap F_{0}$, the corresponding sprinkle moves to $P_{v_0}$. It follows that every geodesic in $S_{v_1}$ between the root and another boundary puncture of $P_{v_1}$ has at most one sprinkle on it. However, for the popsicle $P_{v_0}$, if more than one sprinkle move to $P_{v_0}$ when bubbling off, or equivalently if $|p_{v_0}^{-1}(*)|=|\{i+1,\dots, j\} \cap F_{0}|\geq 2$, then the geodesic in $S_{v_0}$ between the root and the boundary puncture $*$ toward $P_{v_1}$ has more than one sprinkles. But, if $|p_{v_0}^{-1}(*)|=|\{i+1,\dots, j\} \cap F_{0}| \leq 1$, then every geodesic in $S_{v_0}$ between the root and another boundary puncture has at most one sprinkle and hence the associated flavor $p_{v_0}$ is injective as in the case of $P_{v_1}$.

As a consequence, there are different types of codimension-one strata according to the number of elements of $\{i+1,\dots,j\}\cap F_{0} = (\{i+1,\dots,j\}\cap F)\setminus F_1$. As seen above, if $|\{i+1,\dots,j\}\cap F_{0}| \leq 1$, then the associated flavor $p_{v_0}$ is injective. In particular, if $|\{i+1,\dots,j\}\cap F_{0}| = 0$, then $p_{v_0}^{-1}(*) = \emptyset$ and hence $F_0 \subset \{0,1,\dots,i,j+1,\dots, k\}$. If $|\{i+1,\dots,j\}\cap F_{0}| = 1$, then $p_{v_0}^{-1}(*) = \{l\}$ for the unique element $l \in \{i+1,\dots, j\}\cap F_0$. In either case, we may identify $\mathcal{R}^{\mathrm{val}(v_0),p_{v_0}}$ with $\mathcal{R}^{\mathrm{val}(v_0), F_{0}}$. Otherwise, if $|\{i+1,\dots,j\}\cap F_{0}| \geq 2$, then the associated flavor $p_{v_0}$ is not injective and we leave the notation $\mathcal{R}^{\mathrm{val}(v_0),p_{v_0}}$ as it is. On the other hand, since $p_{v_1}$ is always injective, we always identify $\mathcal{R}^{\mathrm{val}(v_1),p_{v_1}}$ with $\mathcal{R}^{\mathrm{val}(v_1),F_{1}}$. 
In conclusion, the codimension-one strata of $\mathcal{R}^{k+1,F}$ are divided into the following two families according to whether $|\{i+1,\dots,j\}\cap F_{0}|\leq 1$ or $|\{i+1,\dots,j\}\cap F_{0}| \geq 2$:
\begin{equation}\label{eq:codimension1strata1}
	\begin{split}
		\bigcup_{0 \leq i < j \leq k} \bigcup_{ |(\{i+1,\dots, j\} \cap F) \setminus F_{1}|\leq 1} \mathcal{R}^{k+2-(j-i),F\setminus F_{1}} \times \mathcal{R}^{j-i+1,F_{1}}
	\end{split}
\end{equation}
and
\begin{equation}\label{eq:codimension1strata2}
\begin{split}
	\bigcup_{0\leq i< j \leq k} \bigcup_{|(\{i+1,\dots, j\} \cap F) \setminus F_{1}|\geq 2} \mathcal{R}^{k+2 - (j-i),p_{v_0}} \times \mathcal{R}^{j-i+1,F_{1}}.
\end{split}
\end{equation}
Here, the above unions are taken for subsets $F_1 \subset \{i+1,\dots, j\} \cap F$ satisfying the stability conditions \eqref{eq:stability}
\begin{equation}\label{eq:stability2}
	k+2-(j-i) + |F\setminus F_1| \geq 3 \text{ and } j-i + |F_1| \geq 3.
\end{equation}

\subsubsection{$A_{\infty}$-structure on Rabinowitz Floer complex}
Let us turn into the construction of higher products $\widecheck{\mu}^k$ on Rabinowitz Floer complexes following the idea of \cite{abo-sei10,cie-oan20}. Let $k \in \mathbb{Z}_{\geq 1}$ and let $p : F\to \{1,\dots, k\}$ be a flavor such that $k+|F| \geq 2.$ For any tuple ${\bf w} =(w_0, w_1,\dots, w_k) $ of integers such that
\begin{equation}\label{eq:weight}
	w_0 = w_1 + \dots +w_k+|F|,
\end{equation}
a {\em $\mathbf{w}$-weighted popsicle} is a $p$-flavored popsicle whose $i$-th boundary puncture is labeled with $w_i$. We consider the moduli $\mathcal{R}^{k+1,p, {\bf w}}$ of weighted popsicles, which is just a copy of $\mathcal{R}^{k+1,p}$. As done in Subsection \ref{subsection:popsicle}, in case $p$ is injective, we will denote the corresponding moduli space by $\mathcal{R}^{k+1,F,\mathbf{w}}$.

There is a universal and consistent choice of strip-like ends $\{ \epsilon_m = \epsilon_m^{P}\}_{P \in \mathcal{R}^{k+1,p}}$ for all pair $(k+1, {p})$ in such a way that the images of $\epsilon_m^{P}$ are mutually disjoint and those do not contain any sprinkles for all $P$. Here $\epsilon_{0}^P$ denotes the negative strip-like end on the popsicle $P$ and $\epsilon_{m}^P$ denotes the $m$-th positive strip-like end for $m =1, \dots, k$.

Now we consider Floer data for $\mathcal{R}^{k+1,p, {\bf w}}$. For the following definition of Floer data with Lagrangian labels, note that the moduli $\mathcal{R}^{2,\{1\}}$ has a unique element $P_w = (S_0,\phi_0)$ and hence so does $\mathcal{R}^{2,\{1\},(w+1,w)}$ for any integer $w$. Here $S_0$ is the infinite strip $\mathbb{R} \times [0,1]$.

\begin{dfn}[{\em Floer data with Lagrangian labels}]\label{dfn:floerdata} Let Floer data associated to Lagrangian pairs be given as in Definition \ref{dfn:floerdataforlagrangian}.
	\begin{enumerate}
		\item For any integer $w$, a Floer datum with Lagrangian labels on $\mathcal{R}^{2,\{1\},(w+1,w)} = \{P_w\}$, which is compatible with the chosen Floer data associated to Lagrangian pairs, consists of 
		\begin{itemize}
			\item a pair $(L_0,L_1)$ of admissible Lagrangians chosen from $\{L_i\}_{i \in I}$,
			\item a domain-dependent Hamiltonian $H^{P_w} :S_0 \to \mathcal{H}(M),z\mapsto H^{P_w}_z$ on $\widehat{M}$ and
			\item a domain-dependent almost complex structure $I^{P_w} : S_0 \to \mathcal{J}(M),z\mapsto I^{P_w}_z$ on $\widehat{M}$
		\end{itemize}
		such that
		\begin{enumerate}[label=(\alph*)]
			\item $H^{P_w}_{(s,t)}$ is independent of $t \in [0,1]$,
			\item $H^{P_w}_{(s,t)} = \begin{cases} H_{\tau^{(L_0,L_1)}_{w+1}} & (s\leq  -1), \\  H_{\tau^{(L_0,L_1)}_w} & (s\geq 1), \end{cases}$
			\item $\partial_s H^{P_w}_{(s,t)} \leq 0$,
			\item $I^{P_w}_{(s,t)} = \begin{cases} J^{(L_0,L_1)}_{w+1}(t) & (s\leq -1), \\ J^{(L_0,L_1)}_w(t) &(s\geq 1),\end{cases}$		
		\end{enumerate}
		where
		\begin{align*}
			\mathcal{H}(M) &= \{  H | H \text{ is an admissible Hamiltonian on } \widehat{M}\} \text{ and }\\
		\mathcal{J}(M) &=  \{ J \,| \,J \text{ is an } \omega \text{-compatible complex structure of contact type on } \widehat{M} \}.
		\end{align*}

		\item For any $k \in \mathbb{Z}_{\geq 2}$, a flavor $p : F \to \{1,\dots,k\}$ and a tuple ${\bf w}= (w_0,w_1,\dots,w_k)$ of integers satisfying \eqref{eq:weight}, Floer data with Lagrangian labels on $\mathcal{R}^{k+1, p, {\bf w}}$, which is compatible with the chosen Floer data associated to Lagrangian pairs, consists of
		\begin{itemize}
			\item a tuple $(L_0,L_1, \dots, L_k)$ of admissible Lagrangians chosen from $\{L_i\}_{i \in I}$,
			\item a smooth family $\{H^P\}_{P = (S,\phi) \in \mathcal{R}^{k+1,p, {\bf w}}}$ of  domain-dependent Hamiltonians,
			\item a smooth family $\{I^P\}_{P = (S,\phi) \in \mathcal{R}^{k+1,p, {\bf w}}}$ of domain-dependent almost complex structures and
			\item a smooth family $\{\beta^P\}_{P = (S,\phi) \in \mathcal{R}^{k+1,p, {\bf w}}}$ of a 1-form $\beta \in \Omega^1(S)$
		\end{itemize}
		such that
		\begin{enumerate}[label=(\alph*)]
			\item $H^{P}_{\epsilon_j (s,t)} = H_{\tau^{L_{j-1},L_j}_{w_j}}$ for $j=0,1,\dots,k$,
			\item $\beta^P|_{\partial S} = 0$,
			\item $\epsilon_j^* \beta^P = dt$, for $j=0,1,\dots,k$,
			\item 
			For all $x \in [2,\infty) \times \partial M \subset \widehat{M}$,
			\begin{equation}\label{eq:formaximumprinciple}
				d_S ( H^P(\cdot,x) \beta^P) \leq 0,
			\end{equation} 
			where $d_S$ is the differential with respect to the domain $S$,
			\item $I^P_{\epsilon_j (s,t)} = J^{L_{j-1},L_j}_{w_j}(t)$ for $j=0,1,\dots,k$ and
			\item $H^P, I^P$ and $\beta^P$ are invariant under the action of $\mathbf{Sym}^p$ \eqref{eq:action} in the sense that
				$$ H^P = H^{s\cdot P}, I^P = I^{s\cdot P} \text{ and } \beta^P = \beta^{s\cdot P}, \forall s\in \mathbf{Sym}^p, P \in \mathcal{R}^{k+1,p,\mathbf{w}}.$$
		\end{enumerate}
	\end{enumerate}
\end{dfn}

\begin{rmk}
	\mbox{}
	\begin{enumerate}
		\item
		Although the first content (1) of Definition \ref{dfn:floerdata} can be regarded as a special case of (2), we state it here separately since the Floer data discussed in (1) are those used to define the continuation maps.
		\item The requirement (d) $\tau_w + \tau_v \leq \tau_{w+v}$, $w,v \in \mathbb{Z}$ in Definition \ref{dfn:floerdataforlagrangian} is necessary to ensure that the conditions (a), (b), (c) and (d) of (2) in Definition \ref{dfn:floerdata} are compatible for $F =\emptyset$.
	\end{enumerate}
\end{rmk}
One may notice that our Floer data are different from those in \cite{abo-sei10} in that we choose a family of Hamiltonians rather than choosing basic 1-forms and sub-closed 1-forms as a part of Floer data. Even so, a standard argument can be applied to showing that the above Floer data can be chosen in a consistent way in the sense that it is compatible with gluing maps in the sense discussed in \cite[Section (9i)]{sei08}. Indeed we have

\begin{thm}\label{thm:floerdata}
	There exists a universal and consistent choice of Floer data with Lagrangian labels on $\mathcal{R}^{k+1, p, {\bf w}}$ for all $k \geq 1$, all flavors $p$ and all weights ${\bf w}$ satisfying \eqref{eq:weight}, which is compatible with given Floer data associated to Lagrangian pairs.
\end{thm}

 The proof of Theorem \ref{thm:floerdata} relies on the fact that the spaces of domain-dependent admissible Hamiltonians, $\omega$-compatible almost complex structures and 1-forms considered in Definition \ref{dfn:floerdata} are convex and hence contractible. In particular, regarding the last requirement (f), note that taking the average of admissible Hamiltonians, Riemannian metrics (corresponding to $\omega$-compatible almost complex structures) and 1-forms over the finite group $\mathbf{Sym}^p$ makes sense. Therefore a universal and consistent family of Floer data can be constructed inductively from lower-dimensional moduli spaces of popsicles to higher-dimensional moduli spaces as discussed in \cite{abo-sei10,ggv22}.

Let us fix universal and consistent Floer data with Lagrangian labels on $\mathcal{R}^{k+1,p,{\bf w}}$ for all Lagrangian labels and for all pairs $(p,{\bf w})$ consisting of a flavor $p$ and a weight ${\bf w}=(w_0,\dots,w_k)$ satisfying \eqref{eq:weight}, which are compatible with the Floer data associated to pairs of admissible Lagrangians chosen above.

Still, to ensure the transversality of the space of popsicle maps that we will discuss below, one further needs an infinitesimal deformation of the almost complex structures $\{ I^P\}_{P \in \mathcal{R}^{k+1,p,{\bf w}}}$ given as a part of the Floer data chosen above. Indeed, we need a family $\{K^P\}_{P \in \mathcal{R}^{k+1,p,{\bf w}}}$ of tangent vectors $K^P_z \in T_{I^P_z} \mathcal{J}(M)$, for which $K^P$ and its derivatives vanish on the strip-like ends of the popsicle $P$ for all popsicles $P$. For such an infinitesimal deformation, we get another family $\{ J^P\}_{P \in \mathcal{R}^{k+1,p,{\bf w}}}$ of almost complex structures obtained by taking the exponential of $K^P$ at $I^P$. We refer the readers to \cite[Section (9k)]{sei08} or \cite[Section 3.2]{abo-sei10} for more details.\\

Let us fix such an infinitesimal deformation $\{K^P\}_{P \in \mathcal{R}^{k+1,p,{\bf w}}}$ of almost complex structures. First, for some integer $w$ and a pair $(L_0, L_1)$ of admissible Lagrangians, let
$$x_0 \in \mathcal{I} (L_{0},L_{1}; H_{\tau_{w+1}}) \text{ and } x_1 \in \mathcal{I} (L_{0}, L_{1}; H_{\tau_{w}})$$
be Hamiltonian chords and let us consider the space of stable popsicle maps
\begin{equation}\label{eq:continuationmoduli}
	\mathcal{R}^{2,\{1\}, (w+1,w)}(x_0,x_1) =\{ u : \mathbb{R} \times [0,1] \to \widehat{M} \,|\, u \text{ satisfies } \mathrm{(i)}, \mathrm{(ii)}, \mathrm{(iii)} \text{ and } \mathrm{(iv)} \text{ below} \}
\end{equation}
\begin{enumerate}[label=(\roman*)]
	\item $u(s,j) \in \widehat{L}_{j}$ for all $s\in \mathbb{R}$ and for $j=0,1$.
	\item $\lim_{s\to -\infty} u(s,t) = x_0(t)$.
	\item $\lim_{s\to \infty} u(s,t) = x_1(t)$.
	\item $\partial_s u + J^{P_w}_{(s,t)} (\partial_t u - X_{H^{P_w}_{(s,t)}} (u)) =0$.
\end{enumerate}
This space is used to define continuation maps $c_{w,w+1}$ as in \eqref{eq:continuation}.

For more general cases, note that, for each $P=(S,\phi) \in \mathcal{R}^{k+1, F,{\bf w}}$, the boundary $\partial S$ has $(k+1)$ components. We will denote by $\partial_m S$ the component between $m$-th and $(m+1)$-th boundary punctures of $\partial S$ for $m=0,\dots,k$. For an integer $k \in \mathbb{Z}_{\geq 1}$, a tuple $(L_0,\dots, L_k)$ of admissible Lagrangians and for some flavor $F \subseteq \{1,\dots, k\}$ and a weight ${\bf w}=(w_0,w_1,\dots, w_k)$ satisfying $\eqref{eq:weight}$, let 
$$x_0 \in \mathcal{I}(L_{0},L_k; H_{\tau_{w_0}}), x_1 \in \mathcal{I}(L_{0}, L_{1};H_{\tau_{w_1}}), \dots, \text{ and }x_k \in \mathcal{I} (L_{k-1},L_{k}; H_{\tau_{w_k}})$$
be Hamiltonian chords. Then we consider the space of stable popsicle maps 
\begin{equation*}
	\mathcal{R}^{k+1,p, {\bf w}} (x_0,\dots,x_k) =\{(P, u)\,|\, P= (S,\phi) \in \mathcal{R}^{k+1,p,{\bf w}}, u : S \to \widehat{M} \text{ satisfies } \mathrm{(i)}, \mathrm{(ii)}, \mathrm{(iii)} \text{ and } \mathrm{(iv)} \text{ below} \}
\end{equation*}
\begin{enumerate}[label=(\roman*)]
	\item $u (\partial_m S) \subset \widehat{L}_{m}$ for $m=0,1,\dots,k$.
	\item $\lim_{s\to -\infty} u(\epsilon_0(s,t)) = x_0(t)$.
	\item $\lim_{s\to \infty} u(\epsilon_m (s,t)) =x_m(t)$ for $m =1,\dots,k$.
	\item $(du - X_{H^P_{z}}(u) \otimes \beta) \circ j+ J^{P}_{z} ( du - X_{H^P_{z}}(u) \otimes \beta) = 0$, where $j$ is the complex structure on $S$.
\end{enumerate}

\begin{rmk}\label{rmk:floerstrips}
Following the manner we defined the space of popsicle maps above, the space $\mathcal{M}(x_0,x_1;H_{\tau_w},J_w)$ of Floer strips defined in \eqref{eq:floerstrip} can be written as $\mathcal{R}^{2,\emptyset,(w,w)}(x_0,x_1)$.	
\end{rmk}

The condition \eqref{eq:formaximumprinciple} guarantees that popsicle maps with fixed inputs and outputs do not escape to infinity. Indeed, we have a variant of \cite[Lemma 7.2]{abo-sei10}, \cite[Lemma 2.2]{cie-oan18} as follows.
	\begin{lem}\label{lem:noescape}
		Let $(P,u) \in \mathcal{R}^{k+1,p,\mathbf{w}}(x_0,\dots,x_k)$ for some Hamiltonian chords $x_1,\dots, x_k$. Then the image of $u$ does not intersect $(2,\infty) \times \partial M$. 
	\end{lem} 
	\begin{proof}
		Suppose the contrary. Namely, suppose that there is a popsicle map $(P=(S,\phi),u) \in \mathcal{R}^{k+1,p,\mathbf{w}}(x_0,\dots,x_k)$ such that the image $u(S)$ intersects $(2,\infty) \times \partial M$. Let $S' = u^{-1}([2,\infty) \times \partial M) \subset S$ and let $u' = u|_{S'}$. Considering that all the Hamiltonian chords $x_i$ are contained in $M \cup [1,2) \times \partial M$, the boundary $\partial S'$ possibly can be decomposed into $\partial_l S'$ and $\partial_n S'$ characterized by
		$$ u'(\partial_l S') \subset \cup_{i=0}^k [2,\infty) \times \partial L_i \text{ and } u'(\partial_n S') \subset \{2\} \times \partial M.$$
		
		Then we will show that the topological energy of $u'$
		\begin{align*}
		E^{top}(u') &= \int_{S'} u'^* d\lambda - d (H^P(\cdot, u'(\cdot )) \beta)
		\end{align*}
		is less than $0$, which is weird since it is supposed to be greater than or equal to 0.
		
		If we denote the slope of the admissible function $H^P(z, \cdot)$ at infinity by $a_z \in \R$ for $z \in S$, then we have
		$$\lambda(X_{H^P}(u'(z))) - H^P(z, u'(z))= a_z.$$
		
		Using this we have
		\begin{align*} 
		E^{top}(u')	&= \int_{S'} d(u'^*\lambda - \lambda(X_{H^p}(u')) \beta^P) + \int_{S'} d(\lambda(X_{H^P}(u')) - H^P(\cdot, u'(\cdot))  \beta^P)\\
		&= \int_{S'} d(u'^*\lambda - \lambda(X_{H^p}(u')) \beta^P) + \int_{S'} d( a_{\cdot}\beta^P).
	\end{align*}
	
		Here \eqref{eq:formaximumprinciple} says that, for $x= (2,y) \in [2,\infty] \times \partial M$,
		$$ d_S (H^P(\cdot,(2,y)) \beta^P) = d (a_{\cdot} \beta^P) \leq 0.$$
		Consequently, we have
	\begin{align*}
		E^{top}(u') & \leq \int_{S'} d(u'^*\lambda - \lambda(X_{H^p}(u')) \beta^P)\\
		&= \int_{\partial_l S'}  u'^*\lambda - \lambda(X_{H^p}(u')) \beta^P  + \int_{\partial_n S'} u'^*\lambda - \lambda(X_{H^p}(u')) \beta^P \\
		&=\int_{\partial_n S'} u'^*\lambda -\lambda(X_{H^p}(u')) \beta^P\\
		&= \int_{\partial_n S'} \lambda (du' - X_{H^p}(u')  \otimes \beta^P)\\
		&= \int_{\partial_n S'} - \lambda \circ J^P \circ (du' - X_{H^p}(u') \otimes \beta^P) \circ j\\
		&= \int_{\partial_n S'} - dr \circ (du' - X_{H^p}(u') \otimes \beta^P) \circ j\\
		&= \int_{\partial_n S'} - dr \circ du' \circ j.
		\end{align*}
	Here, the third equality follows from the fact that the restriction of $\lambda$ to $[2,\infty) \times \partial L_i$ vanishes and $\beta^P$ is assumed to vanish on $\partial S$. The last equality follows since $X^{H^P}(u')$ is a multiple of Reeb vector field over $\partial_n S'$. 
	
	However, the term at the bottom line is less than 0. Indeed, $-dr \circ du' \circ j \leq 0$ with respect to the orientation form on $\partial_n S'$ since $jv$ points inward for a positively oriented tangent vector $v$ along $\partial_n S'$ and $(r \circ u')(z)$ does not decrease as a point $z$ in $S'$ moves from $\partial_n S'$ inward. Moreover, there is at least one point $z\in \partial_n S'$ where the inequality is strict since $u$ is supposed to escape to $(2,\infty) \times \partial M$. This shows that there is no such a popsicle map $(P,u)$ as desired.
	\end{proof}

The smoothness of moduli spaces of popsicle maps are guaranteed for a generic choice of infinitesimal deformations of an almost complex structure given as a part of our universal and consistent Floer data with Lagrangian labels on $\mathcal{R}^{k+1,p,\mathbf{w}}$ as argued in  \cite[Section 3.4]{abo-sei10}. Furthermore, Lemma \ref{lem:noescape} implies that the standard Gromov compactness argument can be applied to the moduli spaces of popsicle maps. As a result, we have the following theorem.
\begin{thm}\label{thm:moduli}
	Let $(k+1,p,{\bf w})$ be a tuple consisting of an integer $k \in \mathbb{Z}_{\geq 1}$, a flavor $p :F \to \{1,\dots, k\}$ and a weight ${\bf w}$ satisfying $k+|F| \geq 2$ and \eqref{eq:weight}. For any tuple $(L_0,\dots, L_k)$ of admissible Lagrangians chosen from the collection $\{L_i\}_{i \in I}$ and Hamiltonian chords $x_0, \dots, x_k$ given as above, we have
	\begin{enumerate}
		\item The space $\mathcal{R}^{k+1,p,{\bf w}}(x_0,x_1,\dots,x_k)$ is a smooth manifold of dimension
		\begin{equation}\label{eq:dimension}
			k-2 + |F| +\deg x_0 -\sum_{i=1}^k \deg x_i,
		\end{equation}
		for a generic choice of infinitesimal deformations $\{K^P\}_{P\in \mathcal{R}^{k+1,p,{\bf w}}}$. 
		\item The space $\mathcal{R}^{k+1,p,{\bf w}}(x_0,x_1,\dots,x_k)$ admits the Gromov compactification obtained by adding broken stable popsicle maps and semi-stable popsicle maps.
	\end{enumerate}
\end{thm}

Let us consider a triple $(k+1,p,\mathbf{w})$ and a tuple $(x_0,x_1,\dots,x_k)$ of Hamiltonian chords given as in Theorem \ref{thm:moduli} such that the expected dimension \eqref{eq:dimension} of the moduli space $\mathcal{R}^{k+1,p,\mathbf{w}} (x_0,x_1,\dots, x_k)$ is zero. Theorem \ref{thm:moduli} implies that $\mathcal{R}^{k+1,p,\mathbf{w}} (x_0,x_1,\dots, x_k)$ is a zero-dimensional manifold with only finitely many points if an infinitesimal deformation is chosen to be generic enough. However, if $\mathbf{Sym}^p$ is not trivial or equivalently if $p$ is not injective, then elements of $\mathcal{R}^{k+1,p,\mathbf{w}} (x_0,x_1,\dots, x_k)$ sum up to zero considering the orientation.
\begin{lem}\label{lem:trivial}\cite[Lemma 3.7]{abo-sei10}
	If $\mathbf{Sym}^p$ is not trivial, then the signed sum of elements of $\mathcal{R}^{k+1,p,\mathbf{w}} (x_0,x_1,\dots, x_k)$ is zero.
\end{lem}
To see why this is the case, observe that $\mathbf{Sym}^p$ acts freely on $\mathcal{R}^{k+1,p,\mathbf{w}} (x_0,x_1,\dots, x_k)$ since the Floer data with Lagrangian labels are assumed to be invariant under the action of $\mathbf{Sym}^p$ in the sense of Definition \ref{dfn:floerdata}. But, any two elements of $\mathcal{R}^{k+1,p,\mathbf{w}} (x_0,x_1,\dots, x_k)$ related by a transposition in $\mathbf{Sym}^p$ contribute with different signs.

This justifies considering the case when $p$ is injective only. In such a case, we denote the corresponding moduli space $\mathcal{R}^{k+1,p,\mathbf{w}}(x_0,x_1,\dots,x_k)$ by $\mathcal{R}^{k+1,F,\mathbf{w}}(x_0,x_1,\dots,x_k)$. Let us denote the Gromov compactification of $\mathcal{R}^{k+1,F,{\bf w}}(x_0,x_1,\dots,x_k)$ by $\overline{\mathcal{R}}^{k+1,F,{\bf w}}(x_0,x_1,\dots,x_k)$ described in Theorem \ref{thm:moduli}.

\bigskip

For later use in the proof of Theorem \ref{thm:ainfinity}, we describe the codimension-one strata of the space $\overline{\mathcal{R}}^{k+1,F,\mathbf{w}}(x_0,x_1,\dots,x_k)$ for a tuple $(x_0,\dots,x_k)$ of Hamiltonian chords such that the expected dimension \eqref{eq:dimension} of the moduli space is one.

First, the codimension-one strata of $\overline{\mathcal{R}}^{k+1,F,\mathbf{w}}(x_0,x_1,
\dots,x_k)$ consisting of stable broken popsicle maps corresponding to \eqref{eq:codimension1strata1} and \eqref{eq:codimension1strata2} are given by 
\begin{equation}\label{eq:codimension1strataofmaps1}
	\bigcup_{1\leq i < j \leq k} \bigcup_{|\{i+1,\dots,j\} \cap F \setminus F_{1}|\leq 1} \bigcup_{y} \mathcal{R}^{k+2-(j-i),F \setminus F_{1},\mathbf{w}_0}(x_0, x_1,\dots,x_{i} ,y, x_{j+1}, \dots, x_k) \times \mathcal{R}^{j-i+1,F_{1},\mathbf{w}_1}(y,x_{i+1}, \dots, x_j),
\end{equation}
and
\begin{equation}\label{eq:codimension1strataofmaps2}
\bigcup_{1\leq i <j \leq k} \bigcup_{|\{i+1,\dots,j\} \cap F \setminus F_{1}| \geq 2} \bigcup_{y} \mathcal{R}^{k+2-(j-i),p_{v_0},\mathbf{w}_0} (x_0,x_1,\dots,x_{i},y,x_{j+1}, \dots,x_k) \times \mathcal{R}^{j-i+1,F_{1},\mathbf{w}_1}(y,x_{i+1},\dots, x_j),
\end{equation}
where 
\begin{itemize}
	\item $\mathbf{w}_0 = (w_0, w_1,\dots, w_i, \sum_{l=i+1}^{j} w_l + |F_{1}|,w_{j+2},\dots,w_k)$, \item $\mathbf{w}_1=(\sum_{l=i+1}^j w_l + |F_{1}|,w_{i+1},\dots,w_j)$.
	\item $y$ runs over all elements of $\mathcal{I}(L_{i}, L_{j}; H_{\tau_{w}})$ of degree $\deg y = \sum_{l=i+1}^j \deg x_l + 2 - (j-i)$ for $w= \sum_{l=i+1}^j w_l + |F_{1}|$ and
	\item $F_1$ runs over all subsets of $\{i+1,\dots, j\} \cap F$ satisfying the prescribed condition
	on the number $|\{i+1,\dots, j\} \cap F \setminus F_1|$ and the stability condition \eqref{eq:stability2}.
\end{itemize}

Second, the codimension-one strata of $\overline{\mathcal{R}}^{k+1,F,\mathbf{w}}(x_0,x_1,\dots,x_k)$ consisting of broken popsicle maps obtained as a semi-stable popsicle map (or equivalently a Floer strip) bubbles off are given by
\begin{equation}\label{eq:codimension1strataofmaps3}
	\bigcup_{y_0} \mathcal{R}^{2, \emptyset, (w_0,w_0)} (x_0, y_0) \times \mathcal{R}^{k+1, F, \mathbf{w}} (y_0, x_1,\dots,x_k),
\end{equation}
where $y_0$ runs over all $y_0 \in \mathcal{I}(L_0,L_k;H_{\tau_{w_0}} )$ of degree $\deg y_0 = \deg x_0 -1$,
and
\begin{equation}\label{eq:codimension1strataofmaps4}
	\bigcup_{y_j} \mathcal{R}^{k+1, F, \mathbf{w}} (y, x_1,\dots, x_{j-1}, y_j, x_j, \dots ,x_k) \times \mathcal{R}^{2, \emptyset, (w_j,w_j)} (y_j, x_j) 
\end{equation}
for $1\leq j \leq k$, where $y_j$ runs over all $y_j \in \mathcal{I}(L_{j-1},L_j;H_{\tau_{w_0}} )$ of degree $\deg y_j = \deg x_j +1$.

We point out that each of the strata \eqref{eq:codimension1strataofmaps1}, \eqref{eq:codimension1strataofmaps2}, \eqref{eq:codimension1strataofmaps3} and \eqref{eq:codimension1strataofmaps4} is a compact zero-dimensional manifold. However, the contribution from \eqref{eq:codimension1strataofmaps2} is zero due to Lemma \ref{lem:trivial}. Therefore, it is enough to consider the strata \eqref{eq:codimension1strataofmaps1}, \eqref{eq:codimension1strataofmaps3} and \eqref{eq:codimension1strataofmaps4} when counting  elements of the codimension-one strata of the moduli space $\mathcal{R}^{k+1,F,\mathbf{w}}(x_0,x_1,\dots,x_k)$. We may combine those three strata \eqref{eq:codimension1strataofmaps1}, \eqref{eq:codimension1strataofmaps3} and \eqref{eq:codimension1strataofmaps4} into
\begin{equation}\label{eq:codimension1strataofmaps}
	\bigcup_{1\leq i < j \leq k} \bigcup_{|\{i+1,\dots,j\} \cap F \setminus F_{1}| \leq 1} \bigcup_{y} \mathcal{R}^{k+2-(j-i),F \setminus F_{1},\mathbf{w}_0}(x_0, x_1,\dots,x_{i} ,y, x_{j+1}, \dots, x_k) \times \mathcal{R}^{j-i+1,F_{1},\mathbf{w}_1}(y,x_{i+1}, \dots, x_j),
\end{equation}
where $\mathbf{w}_0$, $\mathbf{w}_1$ and $y$ are given as right below \eqref{eq:codimension1strataofmaps2} and only the last bullet for the description of $F_1$ changes into
\begin{itemize}
	\item $F_1$ runs over all subsets of $\{i+1,\dots, j\} \cap F$ satisfying $|\{i+1,\dots, j\} \cap F \setminus F_1| \leq 1$.
\end{itemize}
Note that the stability condition \eqref{eq:stability2} is not required here anymore. The above argument says that it is sufficient to count elements of \eqref{eq:codimension1strataofmaps} only among those in the codimension-one strata of  $\mathcal{R}^{k+1,F,\mathbf{w}}(x_0,x_1,\dots,x_k)$. We will make use of this fact in the proof of Theorem \ref{thm:ainfinity} showing the $A_{\infty}$-equations of our higher products on the Rabinowitz Fukaya category.

After choosing an infinitesimal deformation generic enough in the sense of Theorem \ref{thm:moduli}, for any tuple $(x_0,x_1,\dots,x_k)$ of Hamiltonian chords such that the dimension \eqref{eq:dimension} of $\mathcal{R}^{k+1,F,{\bf w}}(x_0,\dots,x_k)$ is zero, one gets a preferred isomorphism
$$ |o_{(P,u)}| :|o_{x_k}|_{\mathbb{K}}\otimes \dots \otimes |o_{x_1}|_{\mathbb{K}} \to |o_{x_0}|_{\mathbb{K}}$$
for every element $(P,u) \in \mathcal{R}^{k+1,F,{\bf w}}(x_0,x_1,\dots,x_k)$. We define the linear map
\begin{equation*}
	m^{k,F,{\bf w}}(x_0,x_1,\dots,x_k)  = \sum_{(P,u) \in \mathcal{R}^{k+1,F,{\bf w}}(x_0,x_1,\dots,x_k)} |o_{(P,u)}|: |o_{x_k}|_{\mathbb{K}}\otimes \dots \otimes |o_{x_1}|_{\mathbb{K}} \to |o_{x_0}|_{\mathbb{K}}
\end{equation*}
as in \cite[Section 3.6]{abo-sei10}.

In particular, the continuation map $c_{\tau_w,\tau_{w+1}} : CF^*(L_0,L_1;H_{\tau_w}) \to CF^*(L_0,L_1;H_{\tau_{w+1}})$ is defined by
\begin{equation}\label{eq:continuation}
 c_{\tau_w,\tau_{w+1}} = \sum_{x_0 \in \mathcal{I}(L_0,L_1;H_{\tau_w}), x_1\in \mathcal{I}(L_0,L_1;H_{\tau_{w+1}}) } m^{1, \{1\}, (w+1,w)}(x_0,x_1).
\end{equation}

Now we are ready to construct higher products
\begin{equation}\label{eq:product1}
	\widecheck{\mu}^k :RFC^*(L_{k-1},L_{k}) \otimes \dots \otimes RFC^*(L_{0},L_{1}) \to RFC(L_{0},L_{k})^*[2-k]
\end{equation}
for $k\in \Z_{\geq 2}$ and tuples $(L_0,\dots,L_k)$ of admissible Lagrangians of $M$.

For that purpose, let us first consider the identifications 
$$ CW_{n-*}(L_0,L_1) = (A^*_{\leq 0}[-1])^* \oplus B_{\leq 0}^*$$
for 
$$A^*_{\leq 0} =\prod_{w=0}^{\infty} CF^* (L_0,L_1; H_{\tau_{-w}}) \text{ and }B^*_{\leq 0}= \prod_{w=0}^{\infty} CF^{*} (L_0,L_1; H_{\tau_{-w}})$$
and 
\begin{equation}\label{eq:wrappedfloercomplexAB}
	CW^*(L_0,L_1) = A_{>0}^* \oplus (B_{>0}[1])^*
\end{equation}
for 
$$A^*_{>0} = \bigoplus_{w=1}^{\infty} CF^*(L_0,L_1;H_{\tau_w}) \text{ and } B^*_{>0}= \bigoplus_{w=1}^{\infty} CF^{*}(L_0,L_1;H_{\tau_w})$$
as graded vector spaces.

Consequently, as a graded vector space, we have the following identifications:
\begin{align*}
	RFC^*(L_0,L_1) &= CW_{n-*}(L_0,L_1) [1] \oplus CW^*(L_0,L_1)\\
	&= A_{\leq 0}^* \oplus (B_{\leq 0}[1])^* \oplus A_{>0}^* \oplus (B_{>0}[1])^*\\
	&= A^* \oplus B[1]^*
\end{align*}
for 
$$A^* = A_{\leq 0}^*\oplus A_{>0}^* \text{ and } B^* = B_{\leq 0}^*\oplus B_{>0}^*.$$
Or equivalently, one may think of this as the formal variables $q$ and $q^{\vee}$ cancel each other so that
$$CW_{n-*}(L_0,L_1)[1] = CW_{n-1-*}(L_0,L_1)q =(A^{*}_{\leq 0}q^{\vee} \oplus B^{*+1}_{\leq 0})q =A^{*}_{\leq 0}  \oplus B^{*+1}_{\leq 0}q$$
and hence
\begin{equation*}
\begin{split}
RFC^*(L_0,L_1) &= CW_{n-1-*}(L_0,L_1)q \oplus CW^*(L_0,L_1)\\
&=(A^{*}_{\leq 0}  \oplus B^{*+1}_{\leq 0}q)\oplus (A^{*}_{> 0}  \oplus B^{*+1}_{> 0}q)\\
&= A^{*} \oplus B^{*+1}q.
\end{split}
\end{equation*}

Here $A^*$ and $B^*$ can be described as follows:
\begin{equation}\label{eq:AB}
	A^*= B^* = \left\{ (x_w)_w \in \prod_{w \in\mathbb{Z}} CF^*(L_0,L_1;H_{\tau_w}) \,\Big| \, \exists \text{ only finitely many } w>0  \text{ such that } x_w \neq 0 \right\}.
\end{equation}

 For later use in the proof of Theorem \ref{thm:ainfinity}, for each $w_0 \in \Z$, we define
 \begin{align*}
  \mathrm{pr}_{A,w_0}&: RFC^*(L_0,L_1)= A^{*} \oplus B^{*+1}q \to CF^{*}(L_0,L_1; H_{\tau_{w_0}})\\
   \mathrm{pr}_{B,w_0}&: RFC^*(L_0,L_1)= A^{*} \oplus B^{*+1}q \to CF^{*+1}(L_0,L_1; H_{\tau_{w_0}})
  \end{align*}
 by
 \begin{equation}\label{eq:prAB}
 	\begin{split}
  \mathrm{pr}_{A,w_0}( (a_w)_w  + (b_w)_w q) &= a_{w_0},\\
  \mathrm{pr}_{B,w_0}( (a_w)_w  + (b_w)_w q) &= b_{w_0}.
  	\end{split}
  \end{equation}

First consider a triple $(k+1=2,F=\emptyset,(w_0,w_0))$ for some weight $w_0\in \Z$, which is not covered in Theorem \ref{thm:moduli}. We define the corresponding operator $$\mu_{A}^{1,\emptyset,(w_0,w_0)} : CF^*(L_0,L_1;w_0) \to CF^{*+1}(L_0,L_1;w_0)$$ by
$$\mu_{A}^{1,\emptyset,(w_0,w_0)}(x_1) = (-1)^{\deg x_1}\delta_{\tau_{w_0}}(x_1), \forall x_1\in \mathcal{I}(L_0,L_1;H_{\tau_{w_0}}),$$
for the Floer differential $\delta_{\tau_{w_0}}$ defined in \eqref{eq:differential}.

Now let $(k+1,F,{\bf w})$ be a triple given as in Theorem \ref{thm:moduli}.
Then let $e_i \in \{0,1\}$ be an exponent given by
\begin{equation*}
	e_i =\begin{cases} 1& (i \in F), \\ 0 & (\text{otherwise})\end{cases}
\end{equation*}
for each $i \in \{ 1, \dots, k\}$. We define an operator
\begin{equation*}
	\mu^{k,F,{\bf w}}_A : CF^*(L_{k-1}, L_{k};H_{\tau_{w_k}})q^{e_k}\otimes \dots \otimes CF^*(L_{0},L_{1};H_{\tau_{w_1}})q^{e_1} \to CF^*(L_{0},L_{k};H_{\tau_{w_0}}),
\end{equation*}
of degree $2-k$ by mapping
$$ x_{k}q^{e_k} \otimes \dots \otimes x_1 q^{e_1} \mapsto \sum_{x_0 \in\mathcal{I}(L_{0},L_{k};H_{\tau_{w_0}})} (-1)^* m^{k,F, {\bf w}}(x_0,\dots,x_k) (x_{k}\otimes \dots \otimes x_1),$$
for $x_j \in \mathcal{I}(L_{j-1}, L_{i_j};H_{\tau_{w_j}})$ and $* = \sum_{j=1}^{j=k} j\deg x_j + \sum_{j \in F} \sum_{m >j} (\deg x_m -1)$, and extending this linearly.

Next, for each $f \in F$, consider
$$F^f = F \setminus \{ f\}$$
and
$${\bf w}^f = (w_0 -1, w_1,\dots,w_k).$$
Then, we define an operator 
\begin{equation*}
	\mu^{k,F,{\bf w}}_B : CF^*(L_{k-1}, L_{k};H_{\tau_{w_k}})q^{e_k}\otimes \dots \otimes CF^*(L_{0},L_{1};H_{\tau_{w_1}})q^{e_1} \to CF^*(L_{0},L_{k};H_{\tau_{w_0-1}})
\end{equation*}
of degree $3-k$ by mapping
$$ x_{k}q^{e_k} \otimes \dots \otimes x_f q^{e_f} \otimes \dots \otimes x_1 q^{e_1} \mapsto \sum_{f \in F} (-1)^{*_f} \mu^{k, F^f, {\bf w}^f}_{A} ( x_{k}q^{e_k} \otimes \dots \otimes x_f \otimes \dots x_1 q^{e_1}),$$
for $x_j \in \mathcal{I}(L_{j-1}, L_{j};H_{\tau_{w_j}})$ and 
$$*_f = \sum_{l=f+1}^k (\deg x_l -1),$$
and extending this linearly.

Let $(e_1,\dots, e_k) \in \prod_{i=1}^k \{0,1\}$ and let $F =\{i\in \{1,\dots,k\}| e_i =1\}$. 
For each $i =1,\dots, k$, let
$$(x_{w_i}^i)_{w_i} q^{e_i} \in \begin{cases} A^{*} &  (e_i=1), \\ B^{*+1}q & (e_i=0). \end{cases}$$
be given. For $k\geq 2$, we define the product $\widecheck{\mu}^k$ \eqref{eq:product1} by
\begin{equation}\label{eq:product2}
	\begin{split}
		\widecheck{\mu}^k\left( (x^k_{w_k})_{w_k} q^{e_k}\otimes \dots \otimes (x^1_{w_1})_{w_1} q^{e_1}\right)&= \widecheck{\mu}^k_A\left((x^k_{w_k})_{w_k} q^{e_k}\otimes \dots \otimes (x^1_{w_1})_{w_1} q^{e_1}\right)\\
		&+\widecheck{\mu}^k_B\left((x^k_{w_k})_{w_k} q^{e_k}\otimes \dots \otimes (x^1_{w_1})_{w_1} q^{e_1}\right) q, 
	\end{split}
\end{equation}
where $\widecheck{\mu}^k_A\left((x^k_{w_k})_{w_k} q^{e_k}\otimes \dots \otimes (x^1_{w_1})_{w_1} q^{e_1}\right)$ and $\widecheck{\mu}^k_A\left((x^k_{w_k})_{w_k} q^{e_k}\otimes \dots \otimes (x^1_{w_1})_{w_1} q^{e_1}\right)$ are defined by
\begin{equation*}\label{eq:productAB}
	\begin{split}
	\widecheck{\mu}^k_A\left((x^k_{w_k})_{w_k} q^{e_k}\otimes \dots \otimes (x^1_{w_1})_{w_1} q^{e_1}\right) & =	\left(\sum_{\sum_{i=1}^k w_i +|F|=w_0}\mu^{k,F, (\sum w_i+|F|,w_k,\dots, w_1)}_A (x^k_{w_k} q^{e_k} \otimes \dots \otimes x^1_{w_1} q^{e_1}) \right)_{w_0},\\
	 \widecheck{\mu}^k_B\left((x^k_{w_k})_{w_k} q^{e_k}\otimes \dots \otimes (x^1_{w_1})_{w_1} q^{e_1}\right)& = \left(\sum_{\sum_{i=1}^k w_i+|F|-1=w_0}\mu^{k,F,(\sum w_i+|F|,w_k,\dots, w_1)}_B (x^k_{w_k} q^{e_k}\otimes \dots \otimes x^1_{w_1} q^{e_1}) \right)_{w_0}.
	\end{split}
\end{equation*}

Recall that $\mu_{A}^{k,F,\mathbf{w}}$ is defined by counting rigid stable popsicle maps that are asymptotic to a Hamiltonian chord at the negative end and Hamiltonian chords at positive ends, and that have sprinkles $\{s_f\}_{f\in F}$. On the other hand, $\mu_{B}^{k,F,\mathbf{w}}$ is a signed sum of $\mu_{A}^{k,F^{f},\mathbf{w}^{f}}$ for all $f\in F$, which means that it is defined counting rigid stable popsicle maps with one sprinkle $s_f$ omitted from $\{s_f\}_{f\in F}$. We will see in the proof of Theorem \ref{thm:ainfinity} that we have to take all such stable popsicle maps into account in order to define higher products that satisfy the $A_{\infty}$-equation once we define the Rabinowitz Floer complex as in \eqref{eq:rabinowitzcomplex} and the differential $\widecheck{\mu}^1$ as in \eqref{eq:rabinowitzdifferential}.

Note further that the definition of $\widecheck{\mu}^k$ given in \eqref{eq:product2} does not extend to the case $k=1$.
Indeed, the expression \eqref{eq:product2} for the case $k=1$ is different from the differential $\widecheck{\mu}^1$ \eqref{eq:rabinowitzdifferential}. To be more precise, for any pair of admissible Lagrangians $(L_0,L_1)$, they differ by the mapping
\begin{equation*}
	\partial_q : RFC^*(L_0,L_1) \to RFC^{*+1}(L_0,L_1), (a_w)_w+(b_w)_w q \mapsto (b_w)_w
\end{equation*}
up to sign, since $\widecheck{\mu}^1$ \eqref{eq:rabinowitzdifferential} involves a map sending $(b_w)q \in B^{*+1}q$ to $(b_w) \in A^{*+1}$.

 We extend the operation $\partial_q$ to tensor products of Rabinowitz Floer complexes by defining $$\partial_q: RFC^*(L_{k-1},L_k) \otimes \dots \otimes RFC^*(L_0,L_1) \to (RFC^*(L_{k-1},L_k) \otimes \dots \otimes RFC^*(L_0,L_1))[1]$$ by
$$\partial_q (c^k\otimes \dots \otimes c^1 ) = \sum_{i=1}^{k} (-1)^{\sum_{j=i+1}^{k} (\deg c^j - 1)} (c^k \otimes \dots \otimes\partial_q(c^i) \otimes \dots \otimes c^1)$$
for all homogeneous elements $c^i \in RFC^{*}(L_{i-1},L_i)$ for $i=1,\dots, k$. Then we have 
\begin{equation}\label{eq:trivial}
	\widecheck{\mu}^k_B(c^k \otimes \dots \otimes c^1) = \widecheck{\mu}^k_A \circ \partial_q(c^k\otimes \dots \otimes c^1).
\end{equation} 
We will see in the proof of Theorem \ref{thm:ainfinity} that this is a part of the $A_{\infty}$-equation involving $\widecheck{\mu}^1$.

\begin{rmk}\label{rmk:wrappedfukayacategory}
	The restriction of the higher products $\widecheck{\mu}^k$ to $CW^*=A_{>0}^* \oplus B_{>0}^{*+1}q \subset A^*\oplus B^{*+1}q = RFC^*$ induces an $A_{\infty}$-structure on the wrapped Fukaya category of $M$. Indeed, following the idea of \cite{abo-sei10}, for a background class $b
 \in H^2(M,\Z/2)$, we may define the wrapped Fukaya category $\mathcal{W}_b(M)$ of $M$ to be an $A_{\infty}$-category with objects in $\{L_i\}_{i\in I}$. For every pair $i,j \in I$, we define the morphism space from $L_i$ to $L_j$ to be the wrapped Floer complex \eqref{eq:wrappedfloercomplexAB}. Furthermore, from the construction \eqref{eq:product2}, one can deduce that if 
	\begin{equation*}
		(x_w^m)_w q^{e_m} \in \begin{cases} A^{*}_{>0}  & (e_m=0),\\  B^{*+1}_{>0}q& (e_m=1) \end{cases}
	\end{equation*} 
	for all $1\leq m \leq k$, then the higher product $\widecheck{\mu}^k((x^k_w)_w q^{e_k}\otimes \dots \otimes (x^1_w)_w q^{e_1})$ lies in $A^{*}_{>0}  \oplus B^{*+1}_{>0}q$ again. We denote the restriction of the higher product $\widecheck{\mu}^k$ to wrapped Floer complexes by
	$$\mu^k : CW^*(L_{k-1},L_k)\otimes \dots \otimes CW^*(L_0,L_1) \to CW^*(L_0,L_k).$$
	Once the higher products $\widecheck{\mu}^k$ are shown to satisfy the $A_{\infty}$-equation, then the same thing can be said for the higher products $\mu^k$. Similarly, we will denote by $\F_b(M)$ the corresponding compact Fukaya category. We will omit the subscript $b$ from the notation whenever it is not crucial.

	Note that our wrapped Fukaya category is essentially the same as that introduced in \cite{abo-sei10} with some technical differences. Indeed, we have chosen a sequence of Hamiltonians $\{H_{\tau_w}^{L_0,L_1}\}_{w \in \mathbb{Z}}$ for each pair of admissible Lagrangians $(L_0,L_1)$, while in \cite{abo-sei10}, the authors chose one Hamiltonian $H$ with slope $1$ at infinity and used sub-closed 1-forms to consider its multiples $wH$ for $w\in \Z_{> 0}$ when defining the wrapped Floer complex and the higher products.
\end{rmk}

The following lemma shows that the product $\widecheck{\mu}^k$ is well-defined for every $k \in \mathbb{Z}$.
\begin{lem}\label{lem:welldefined}
	Let $k\in \mathbb{Z}_{\geq 1}$. For any tuple $((x^k_{w_k})_{w_k} q^{e_k}, \dots, (x^1_{w_1})_{w_1} q^{e_1}) \in RFC^*(L_{k-1}, L_k) \times \dots \times RFC^*(L_0,L_1)$, the following hold.

	\begin{enumerate}
		\item There are only finitely many $w_0 \in \Z_{>0}$ such that
		\begin{equation}\label{eq:onlyfinitelymanytermsA}
			w_0 = \sum_{i=1}^k w_i +|F| \text{ and }
			\mu^{k,F, (w_0 +1,w_1,\dots, w_k)}_A (x^k_{w_k} q^{e_k} \otimes \dots \otimes x^1_{w_1} q^{e_1}) \neq 0
		\end{equation}
		for some $(w_1, \dots, w_k) \in \mathbb{Z}^k$.
		\item For a given $w_0 \in \Z$, there are only finitely many pairs of tuples $(w_1,\dots ,w_k) \in \Z^{k}$ such that 
		\eqref{eq:onlyfinitelymanytermsA} holds.
		\item There are only finitely many $w_0 \in \Z_{>0}$ such that
		\begin{equation}\label{eq:onlyfinitelymanytermsB}
			w_0 = \sum_{i=1}^k w_i +|F|-1 \text{ and }
			\mu^{k,F, (w_0 ,w_1,\dots, w_k)}_B (x^k_{w_k} q^{e_k} \otimes \dots \otimes x^1_{w_1} q^{e_1}) \neq 0
		\end{equation}
		for some $(w_1, \dots, w_k) \in \mathbb{Z}^k$.
		\item For a given $w_0 \in \Z$, there are only finitely many of tuples $(w_1,\dots ,w_k) \in \Z^{k}$ such that 
		\eqref{eq:onlyfinitelymanytermsB} holds.
	\end{enumerate}
\end{lem}
\begin{proof}
	 We prove the statements (1) and (2) only. Then the statements (3) and (4) can be proved similarly.
	 	
	 	For the statement (1), note that the description of $A^*$ and $B^*$ in \eqref{eq:AB} implies that, for each $i\in \{1,\dots, k\}$, there exists a sufficiently large integer $M_i$ such that $x_{w_i}^{i}=0$ for all $w_i \geq M_i$. Therefore, if $w_0 \geq \sum_{i=1}^k M_i$, then at least one of $w_i$ needs to satisfy $w_i \geq M_i$ to have $w_0 = \sum_{i=1}^k w_i +|F|$. Therefore, if $w_0 \geq \sum_{i=1}^k M_i$ and $w_0 =\sum_{i=1}^k w_i + |F|$, then we have
	 $$x^k_{w_k} q^{e_k} \otimes \dots \otimes x^1_{w_1} q^{e_1} =0$$
 for any tuple $(w_1,\dots, w_k) \in \Z^k$ and any subset $F \subseteq \{1,\dots,k\}$.
	 
	 For the statement (2), observe that, for a given $w_0 \in \mathbb{Z}$, there are only finitely many tuples $(w_1,\dots,w_k) \in \Z^k$ such that $w_0 = \sum_{i=1}^k w_i + |F|$ and $w_i <M_i$ for all $i \in \{1,\dots,k\}$. Indeed, if $w_i < w_0 - \sum_{j\neq i} M_j -|F|$ for some $i\in \{1,\dots, k\}$, then $w_0 = \sum_{i=1}^k w_i + |F|$ does not hold unless $w_j \geq M_j$ for at least one $j \in \{1,\dots, k\} \setminus \{i\}$.
\end{proof}

Then we finally have the following theorem.
\begin{thm}[{cf. \cite[Lemma 4.9 and Lemma 4.11]{ggv22}}]\label{thm:ainfinity}
 There exist products $\widecheck{\mu}^k$ on Rabinowitz Floer complexes such that
	\begin{enumerate}
		\item $\widecheck{\mu}^k$ satisfy the $A_\infty$-equation. i.e. for any $k \in \mathbb{Z}_{\geq 1}$ and any tuple $((x^k_{w_k})_{w_k} q^{e_k}, \dots, (x^1_{w_1})_{w_1} q^{e_1}) \in RFC^*(L_{k-1}, L_k) \times \dots \times RFC^*(L_0,L_1)$,
		\begin{equation}\label{eq:ainfinity}
			\begin{split}
			\sum_{1\leq i <j \leq k} (-1)^{\#_{i}}&\widecheck{\mu}^{k+1-(j-i)}\Big((x^k_{w_k})_{w_k}q^{e_k} \otimes \dots \otimes (x^{j+1}_{w_{j+1}})_{w_{j+1}}q^{e_{j+1}} \otimes  \\
		&\widecheck{\mu}^{j-i}\left((x^j_{w_j})_{w_j} q^{e_j}\otimes \dots \otimes (x^{i+1}_{w_{i+1}})_{w_{i+1}} q^{e_{i+1}}\right) \otimes (x^i_{w_i})_{w_i}q^{e_i} \otimes \dots \otimes (x^1_{w_1})_{w_1} q^{e_1}\Big)=0,
			\end{split}
		\end{equation}
		where $\#_i = \sum_{l=1}^{i} (\deg x_l -1)$ for each $1\leq i \leq k$.
		
		\item The natural inclusions $i : CW^*(L_i,L_j) \hookrightarrow RFC^*(L_i,L_j)$ induce an $A_{\infty}$-homomorphism
		$$ \bigoplus_{ i,j} CW^*(L_i,L_j) \to \bigoplus_{i,j} RFC^*(L_i,L_j)$$
		for any collection $\{L_i\}$ of admissible Lagrangians. In particular, the inclusions induce an $A_{\infty}$-functor from the wrapped Fukaya category $\W (M)$ to the Rabinowitz Fukaya category $\mathcal{RW}(M)$.
	\end{enumerate}
	
\end{thm}
\begin{proof}
Let $(e_1,\dots, e_k) \in \prod_{i=1}^k \{0,1\}$ and let $F= \{i\in \{1,\dots,k\}| e_i =1\}$ be the flavor determined by the tuple $(e_1,\dots,e_k)$. 
Also, for each $1\leq i \leq k$, let $(x^i_{w_i})_{w_i}q^{e_i} \in  RFC^*(L_{i-1},L_i)$ be an element given by
$$x^i_{w_i} = \begin{cases} x_{i} &(w_i = a_i)\\ 0 &(w_i \neq a_i) \end{cases}$$
for some $a_i \in \Z$ and $x_i \in \mathcal{I}(L_{i-1}, L_i;H_{\tau_{a_i}})$, which we will denote by $(x_i)_{w_i =a_i}q^{e_i}$ for convenience.

Consider the $A_{\infty}$-equation \eqref{eq:ainfinity} for the tuple $((x_k)_{w_k=a_k}q^{e_k},\dots, (x_1)_{w_1=a_1}q^{e_1})$. For the integer $a_0 = \sum_{i=1}^k a_i +|F|$, its image under $\mathrm{pr}_{A,a_0}$ \eqref{eq:prAB} is given by
\begin{equation}\label{eq:ainfinityA}
	\begin{split}
		&\sum_{0\leq i<j \leq k}(-1)^{\#_i} \mu_A^{k+1-(j-i),F_{0},\mathbf{w}_0}(x_k q^{e_k}\otimes \dots \otimes \mu_A^{j-i, F_{1},\mathbf{w}_1}(x_{j}q^{e_j} \otimes \dots \otimes x_{i+1}q^{e_{i+1}}) \otimes \dots x_1 q^{e_1} ) + \\
		 & \sum_{0\leq i<j \leq k} (-1)^{\#_i}\mu_A^{k+1-(j-i),F_{0}\cup \{*\},\mathbf{w}'_0}(x_k q^{e_k}\otimes \dots \otimes \mu_B^{j-i, F_{1},\mathbf{w}_1}(x_{j}q^{e_j} \otimes \dots \otimes x_{i+1}q^{e_{i+1}})q \otimes \dots x_1 q^{e_1} ) =0,
	\end{split}
\end{equation}
where $F_1=\{i+1,\dots, j\}\cap F$ and $F_0 = F\setminus F_1$.
 Similarly, for the integer $a_0 = \sum_{i=1}^k a_i +|F|-1$, its image under $\mathrm{pr}_{B,a_0}$ is given by
\begin{equation}\label{eq:ainfinityB}
	\begin{split}
		&\sum_{0\leq i<j \leq k} (-1)^{\#_i} \mu_B^{k+1-(j-i),F_{0},\mathbf{w}_0}(x_k q^{e_k}\otimes \dots \otimes \mu_A^{j-i, F_{1},\mathbf{w}_1}(x_{j}q^{e_j} \otimes \dots \otimes x_{i+1}q^{e_{i+1}}) \otimes \dots x_1 q^{e_1} )+\\
		& \sum_{0\leq i<j \leq k} (-1)^{\#_i} {\mu}_B^{k+1-(j-i),F_{0}\cup \{*\},\mathbf{w}'_0}(x_k q^{e_k}\otimes \dots \otimes \mu_B^{j-i, F_{1},\mathbf{w}_1}(x_{j}q^{e_j} \otimes \dots \otimes x_{i+1}q^{e_{i+1}})q \otimes \dots x_1 q^{e_1} ) =0,
	\end{split}
\end{equation}
where the flavors $F_1$ and $F_0$ are given as above.

Lemma \ref{lem:welldefined} implies that, for an arbitrarily given $((x^k_{w_k})_{w_k} q^{e_k}, \dots, (x^1_{w_1})_{w_1} q^{e_1}) \in RFC^*(L_{k-1}, L_k) \times \dots \times RFC^*(L_0,L_1)$, the projection image of the corresponding $A_{\infty}$-equation \eqref{eq:ainfinity} under $\mathrm{pr}_{A,a_0}$(respectively $\mathrm{pr}_{B,a_0}$) is a finite sum of \eqref{eq:ainfinityA} (respectively \eqref{eq:ainfinityB}) for some $((x^k_{a_k})_{w_k=a_k}q^{e_k},\dots,(x^1_{a_1})_{w_1=a_1}q^{e_1})$. This means that it is enough to show that both \eqref{eq:ainfinityA} and \eqref{eq:ainfinityB} hold to prove the statement (1).

Let us first show that \eqref{eq:ainfinityA} holds for any tuple $((x_k)_{w_k=a_k}q^{e_k},\dots, (x_1)_{w_1=a_1}q^{e_1})$ given as above. Consider all the terms of the left hand side of \eqref{eq:ainfinityA} involving at least one of $\widecheck{\mu}^1$. Some of these terms cancel each others for a trivial reason due to \eqref{eq:trivial}. Next, for a given $x_0 \in \mathcal{I}(L_0,L_k;H_{\tau_{a_0}})$, the coefficient of $x_0$ in the left hand side of \eqref{eq:ainfinityA} except those canceled due to \eqref{eq:trivial} is given by the count of rigid stable broken popsicle maps of \eqref{eq:codimension1strataofmaps}. This sums up to zero since the union of \eqref{eq:codimension1strataofmaps} and \eqref{eq:codimension1strataofmaps2} is given by the codimension-one strata of the one-dimensional compactified moduli space $\overline{\mathcal{R}}^{k+1,F,\mathbf{w}}$, but the signed sum of rigid elements of \eqref{eq:codimension1strataofmaps2} is zero as explained above. This proves that \eqref{eq:ainfinityA} holds.

To show that \eqref{eq:ainfinityB} holds, observe that the left hand side of \eqref{eq:ainfinityB} is the (signed) sum of the left hand side of \eqref{eq:ainfinityA} with $F$ replaced by $F^{f} = F\setminus \{f\}$ over all $f\in F$. Since we have already shown that \eqref{eq:ainfinityA} holds for any flavor $F$, the assertion follows.
 
 The statement (2) is a consequence of Remark \ref{rmk:wrappedfukayacategory}.
 \end{proof}

\begin{exa}
We describe the product when $k=2$. 
In particular, consider the following product when two inputs contain the formal variable $q$:
\begin{equation*}
		\widecheck{\mu}^2\left( (x_2)_{w_2=a_2} q \otimes (x_1)_{w_1=a_1} q \right) = \widecheck{\mu}^2_A\left((x_2)_{w_2=a_2} q \otimes (x_1)_{w_1=a_1} q\right)  +\widecheck{\mu}^2_B\left((x_2)_{w_2=a_2} q \otimes (x_1)_{w_1=a_1} q \right)q. 
\end{equation*}

Then, the first term 
$\widecheck{\mu}^2_A\left((x_2)_{w_2=a_2} q \otimes (x_1)_{w_1=a_1} q\right)$ is given by $\mu^{2,\{1,2\},(w_{1}+w_{2}+2,w_{1},w_{2})}(x_2 q \otimes x_1 q)$ (up to sign)  and it is computed by counting rigid stable popsicle maps as in Figure \ref{fig:mu2} (A). The second one $\widecheck{\mu}^2_B\left((x_2)_{w_2=a_2} q \otimes (x_1)_{w_1=a_1} q \right)q$ consists of two terms (up to sign)
$$\mu^{2,\{ 2 \},(w_{1}+w_{2}+1,w_{1},w_{2})}_A (x_2 q \otimes x_1 ) \mbox{ and } \mu^{2,\{ 1 \},(w_{1}+w_{2}+1,w_{1},w_{2})}_A (x_2 \otimes x_1 q)$$
and they are defined by popsicles as in Figure \ref{fig:mu2} (B) and (C). These come from forgetting a sprinkle on each geodesic.

\begin{figure}[h]
\begin{subfigure}[t]{0.31\textwidth}
\includegraphics[scale=1]{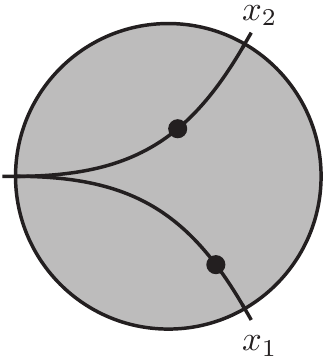}
\centering
\caption{ }
\end{subfigure}
\begin{subfigure}[t]{0.31\textwidth}
\includegraphics[scale=1]{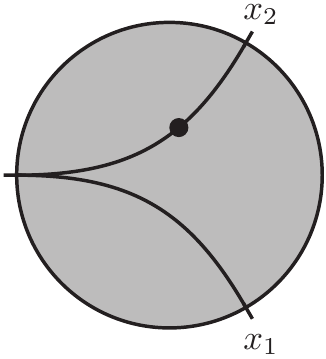}
\centering
\caption{ }
\end{subfigure}
\begin{subfigure}[t]{0.31\textwidth}
\includegraphics[scale=1]{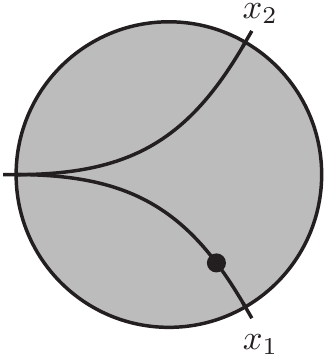}
\centering
\caption{ }
\end{subfigure}
\centering
\caption{ }
\label{fig:mu2}
\end{figure}
\end{exa}

\begin{rmk}
As mentioned in the introduction, our $A_\infty$-structures $\widecheck{\mu}^k$ are defined using linear Hamiltonians rather than quadratic Hamiltonians as in \cite{ggv22}. We adapted the idea of the construction of $A_{\infty}$-structures on the wrapped Fukaya category in \cite{abo-sei10} using linear Hamiltonians to the mapping cone construction of Rabinowitz Fukaya category introduced in \cite{cie-oan20,ggv22} to construct an $A_\infty$-structure on the Rabinowitz Fukaya category. Therefore we expect that our Rabinowitz Fukaya category is quasi-equivalent to the Rabinowitz Fukaya category constructed in \cite{ggv22}, which is parallel to the fact that the wrapped Fukaya category constructed in \cite{abo-sei10} defined using linear Hamiltonians is considered to be quasi-equivalent to the wrapped Fukaya category constructed in \cite{abo10,abo11} defined using quadratic Hamiltonians. Moreover, our Rabinowitz Fukaya category shares a lot of  common features with the Rabinowitz Fukaya category introduced in \cite{ggv22}. In particular, Theorem \ref{thm:ainfinity} can be compared with \cite[Lemma 4.9 and Lemma 4.11]{ggv22}.
\end{rmk}

\section{Calabi--Yau property of derived Rabinowitz Fukaya category}\label{section:cy}
The main goal of this section is to prove that the derived Rabinowitz Fukaya category of a Liouville domain $M$ of dimension $2n$ is $(n-1)$-Calabi--Yau assuming the condition \eqref{condition3} of Theorem \ref{thm:maintheorem}.

 To prove the assertion, we first show that there is a non-degenerate paring 
\begin{equation*}
	\beta^k=\beta^k_{L_i,L_j} : RFH^{n-1-k}(L_j,L_i) \times RFH^{k} (L_i,L_j) \to \mathbb{K},
\end{equation*}
for any pair $i,j \in I$ and any $k\in \mathbb{Z}$. We further show that this non-degenerate pairing can be extended to the triangulated envelope of the Rabinowitz Fukaya category.

Note that, for a given chain complex $C$ over a field $\mathbb{K}$, we denote its dual complex $\Hom_{\mathbb{K}}(C,\mathbb{K})$ by $C^\vee$. Then our strategy to show the existence of a non-degenerate pairing $\beta$ is to construct chain maps
\begin{itemize}
	\item $\overline{\alpha}: CW^*(L_i,L_j) \to CW_*(L_j,L_i)^{\vee}$,
	\item $\overline{\gamma} : CW_{n-1-*}(L_i,L_j) \to (CW^{n-1-*}(L_j,L_i))^{\vee}$ and 
	\item $\overline{\beta} : RFC^*(L_i,L_j) \to (RFC^{n-1-*}(L_j,L_i))^{\vee}$
\end{itemize}
for $i,j \in I$ such that the following diagram commutes:
\begin{equation}\label{eq:morphismbewteenses}
	\xymatrix{ 0\ar[r] & CW^*(L_i,L_j)\ar[r]^{i} \ar[d]^{\overline{\alpha}}&RFC^*(L_i,L_j) \ar[r]^{p} \ar[d]^{\overline{\beta}}&CW_{n-1-*}(L_i,L_j)\ar[r] \ar[d]^{\overline{\gamma}}&0 \\ 0\ar[r]& CW_*(L_j,L_i)^{\vee} \ar[r]^{p^*} &RFC^{n-1-*}(L_j,L_i)^{\vee} \ar[r]^{i^*}& CW^{n-1-*}(L_j,L_i)^{\vee} \ar[r]&0. }
\end{equation}
We will further show that $\overline{\alpha}$ is a quasi-isomorphism if the condition \eqref{condition3} of Theorem \ref{thm:maintheorem} is satisfied and $\overline{\gamma}$ is a quasi-isomorphism even without such an assumption. Then the commutativity of the above diagram and the five-lemma imply that $\overline{\beta}$ is a quasi-isomoprhism as well. 
Since taking dual and taking homology commute over a field, we have 
$$ RFH^{n-1-*}(L_j,L_i)^{\vee} = H^*( RFC^{n-1-*}(L_j, L_i)^{\vee}).$$

This implies that the pairing
$$ \beta^k : RFH^{n-1-k}(L_j,L_i) \times RFH^{k}(L_i,L_j) \to \mathbb{K}, k\in \Z$$
defined by
$$ \beta^k([y],[x]) = \overline{\beta}(x)(y)$$
for all cycles $x\in RFC^{k}(L_i,L_j)$ and $y\in RFC^{n-1-k}(L_j,L_i)$
is non-degenerate.

\subsection{Construction of $\overline{\alpha}$ and $\overline{\gamma}$}
Let us begin the construction of the chain maps $\overline{\alpha}$ and $\overline{\gamma}$. First, for any admissible Lagrangian $L$ of $M$, let $\epsilon_{L} >0$ be sufficiently small so that
\begin{equation*}
	\epsilon_{L} < \mathrm{min} ~\mathrm{Spec} (\partial L, \partial L) \cap \mathbb{R}_{>0}.
\end{equation*}

The proof of Lemma \ref{lem:existenceoffloerdata} implies that the Floer data $\{(\tau^{L_0,L_1}_{w}, H_{\tau^{L_0,L_1}_w},J^{L_0,L_1}_w)\}_{w\in \mathbb{Z}}$ associated to Lagrangian pairs (Definition \ref{dfn:floerdataforlagrangian}) have been chosen in such a way that the following two additional conditions are satisfied:
\begin{enumerate}[label=(\alph*)]
	\setcounter{enumi}{6}
	\item $-\epsilon_L < \tau^{L,L}_{0} <0$ for all admissible Lagrangian $L$. 
	\item $H_{\tau^{L,L}_0}|_L$ and $H_{\tau^{L,L}_1}|_L$ are $C^2$-small Morse functions for any admissible Lagrangian $L$.
\end{enumerate}

Denoting $\tau^{L,L}_w$ by $\tau_w$ for $w\in \mathbb{Z}$, the Floer complex $CF^*(L,L;H_{\tau_0})$ (respectively $CF^*(L,L;H_{\tau_1})$) can be identified with the cohomological Morse complex $CM^*(L;H_{\tau_0}|_L)$ (respectively $CM^*(L;H_{\tau_1}|_L)$), as argued in \cite[Proposition 2.10]{cie-oan18} and \cite[Section 5]{rit13}.

Furthermore, the Morse complex $CM^*(L;H_{\tau_0}|_L)$ is canonically identified with the dual of the Morse complex $CM^{n-*}(L;-H_{\tau_0}|_L)$ by definition. Namely, there is an identification:
\begin{equation*}\label{eq:dual}
	CM^{*}(L;H_{\tau_0}|_L)^{\vee} = CM^{n-*}(L;-H_{\tau_0}|_L).
\end{equation*}

	 But, the homology of the latter $CM^{n-*}(L;-H_{\tau_0}|_L)$ is isomorphic to $H^{n-*}(L;\mathbb{K})$ or $H_*(L,\partial L;\mathbb{K})$ as $-H_{\tau_0}|_L$ has a positive slope less than $\epsilon_L$ at infinity. Consequently, the homology of $CM^*(L;H_{\tau_0}|_L)$ is isomorphic to $H^*(L,\partial L;\mathbb{K})$ or $H_{n-*}(L;\mathbb{K})$.

Let $\psi_L \in CM^{n}(L;H_{\tau_0}|_L)^{\vee} = CM^{0}(L;-H_{\tau_0}|_L)$ be a cycle that represents the unit class $1 \in H^0(L;\mathbb{K})$ or the relative fundamental class $[L,\partial L] \in H_n(L,\partial L;\mathbb{K})$. Equipping the ground field $\mathbb{K}$ with the trivial differential, this gives rise to a chain map $\langle [L,\partial L], -  \rangle : CM^{*}(L;H_{\tau_0}|_L) \to \mathbb{K}$ defined by
\begin{equation*}
\langle [L,\partial L], x \rangle = \psi_L(x).
\end{equation*}

Now we are ready to define chain maps
\begin{equation*}
	\begin{split}
		\alpha &=\alpha_{L_0,L_1} : CW_*(L_1,L_0) \otimes CW^*(L_0,L_1) \to \mathbb{K},\\
		\gamma &=\gamma_{L_0,L_1} : CW^*(L_1,L_0) \otimes CW_*(L_0,L_1) \to \mathbb{K}.
	\end{split}
\end{equation*}

Here we equip the tensor product $CW_*(L_1,L_0)\otimes CW^*(L_0,L_1)$ with the differential $\delta'$ given by
\begin{equation*}\label{eq:tensordifferential}
	\delta'(y\otimes x) =  (-1)^{\deg x} \partial(y)\otimes x + y\otimes \delta(x)
\end{equation*}
for all homogeneous elements $x \in CW^*(L_0,L_1)$ and $y\in CW_*(L_1,L_0)$ and equip the tensor product $CW^*(L_1,L_0)\otimes CW_*(L_0,L_1)$ with a differential defined similarly. This definition of $\delta'$ ensures that the maps $\alpha$ and $\gamma$ defined below are chain maps. It can be shown that the differential $\delta'$ is equivalent to the more usual differential on a tensor product in the sense that their associated cochain complexes are canonically chain isomorphic.

Recall from \eqref{eq:ses} that there exist the inclusion map
$$ i: CW^*(L_0,L_1) \to RFC^*(L_0,L_1)$$
and the quotient map
$$ p: RFC^*(L_0,L_1) \to CW_{n-1-*}(L_0,L_1),$$
which are chain maps. Even though there is no chain map from $CW_{n-1-*}(L_0,L_1)$ to $RFC^*(L_0,L_1)$ that is the inverse of $p$, $CW_{n-1-*}(L_0,L_1)$ is still a subspace of $RFC^*(L_0,L_1)$. Let us denote the inclusion map by
$$ p^{-1} : CW_{n-1-*}(L_0,L_1) \to RFC^*(L_0,L_1),$$
which is also a right inverse of $p$.

We define the map $\alpha$ and $\gamma$ 
\begin{align*}
	\alpha( y\otimes x)& = \left\langle [L_0,\partial L_0], \pi \circ p \circ \widecheck{\mu}^2 ( p^{-1}(y)\otimes i(x))\right\rangle,\\
	\gamma(x\otimes y) &= \left\langle [L_0,\partial L_0], \pi \circ p \circ \widecheck{\mu}^2 ( i(x)\otimes p^{-1}(y))\right\rangle
\end{align*}
for $y \in CW_{*}(L_1,L_0)$ and $x \in CW^*(L_0,L_1)$, where 
$\pi : CW_{n-*}(L_0,L_0) \to CF^*(L_0,L_0;H_{\tau_0})\cong CM^*(L_0;H_{\tau_0}|_{L_0})$ is the projection map discussed in \eqref{eq:projection}. Then we have:
\begin{lem}\label{lem:chainmap}
	Both $\alpha$ and $\gamma$ are chain maps.
\end{lem}
\begin{proof}
	Let us prove the assertion for $\alpha$ only since the other one can be proved similarly.
	Since the operators $\langle [L_0,\partial L_0] , - \rangle$ and $\pi$ are already chain maps, it suffices to show that the map $\eta : CW_*(L_1,L_0) \otimes CW^*(L_0,L_1) \to CW_*(L_0,L_0)$ defined by
	$$ \eta (y\otimes x) = p \circ \widecheck{\mu}^2( p^{-1}(y)\otimes i(x))$$
	is a chain map. 
	
	Let $\pi^{-1} : CF^{n-*}(L_1,L_0;H_{\tau_0}) \to CW_*(L_1,L_0)$ be the inclusion map, which is also a right inverse of $\pi$. Then every element $y\in CW_*(L_1,L_0)$ is the sum
	$$ y = y' + y'' = (y - \pi^{-1} (\pi(y))) + \pi^{-1} (\pi(y))$$
	of an element $y' =y - \pi^{-1} (\pi(y))$ that does not involve any summand from $CF^{n-*}(L_1,L_0;H_{\tau_0}) \subset CW_*(L_1,L_0)$ and an element $y'' = \pi^{-1}(\pi(y)) \in CF^{n-*}(L_1,L_0;H_{\tau_0}) \subset CW_*(L_1,L_0)$. Hence it also suffices to show that the assertion holds for such elements $y'$ and $y''$ separately.
	
	Let us first show that the assertion holds for $y'\otimes x$, for which $y'$ does not involve any summand from $CF^{n-*}(L_1,L_0;H_{\tau_0})$. Indeed, for such elements $y'$, one can observe
	$$ \widecheck{\mu}^1 \circ p^{-1}(y') = - p^{-1} \circ \partial(y')$$
	from the definition of the differential $\partial$ on $CW_*(L_1,L_0)$ and \eqref{eq:rabinowitzdifferential}.
	Therefore, we have
	\begin{equation}\label{eq:longargument1}
		\begin{split}
			\eta \circ \delta'(y'\otimes x)  &= (-1)^{\deg x} \eta ( \partial(y') \otimes x) + \eta( y'\otimes \delta(x)) \\
			&= (-1)^{\deg x} p\circ\widecheck{\mu}^2(p^{-1}\left(\partial(y')) \otimes i(x)\right)+ p\circ \widecheck{\mu}^2\left( p^{-1}(y') \otimes \delta(i(x)) \right)\\
			&= (-1)^{\deg x -1} p\circ \widecheck{\mu}^2\left( \widecheck{\mu}^1(p^{-1}(y')) \otimes i(x) \right) +p\circ \widecheck{\mu}^2 \left( p^{-1}(y') \otimes \widecheck{\mu}^1(i(x)) \right)\\
			&= p \left( (-1)^{\deg x-1}\widecheck{\mu}^2\left( \widecheck{\mu}^1(p^{-1}(y')) \otimes i(x)\right)  + \widecheck{\mu}^2 \left( p^{-1}(y') \otimes \widecheck{\mu}^1(i(x)) \right) \right)\\
			&= - p \circ \widecheck{\mu}^1 \circ \widecheck{\mu}^2 \left( p^{-1}(y') \otimes i(x) \right)\\
			&= \partial \circ p \circ \widecheck{\mu}^2 \left( p^{-1}(y') \otimes i(x) \right)\\
			&= \partial\circ \eta(y'\otimes x).
		\end{split}
	\end{equation}
	Here the fifth equality follows from the $A_{\infty}$-equation between $\widecheck{\mu}^{k}$'s. This proves the assertion.
	
	Now, for all homogeneous elements $y'' \in CF^{n-1-*}(L_1,L_0;H_{\tau_0}) \subset CW_*(L_1,L_0)$, observe
	$$ \widecheck{\mu}^1 \circ p^{-1}(y'') = -p^{-1} \circ \partial (y'') + (-1)^{\deg y''+1} i\circ c_{\tau_0,\tau_1}(y''),$$
	where $i\circ c_{\tau_0,\tau_1} : CF^{n-1-*}(L_1,L_0;H_{\tau_0}) \to  RFC^{n-1-*}(L_1,L_0)$ factors through
	$$CF^{n-1-*}(L_1,L_0;H_{\tau_0})\xrightarrow{c_{\tau_0,\tau_1}} CF^{n-1-*}(L_1,L_0;H_{\tau_1}) \subset CW^{n-1-*}(L_1,L_0)\xrightarrow{i} RFC^{n-1-*}(L_1,L_0).$$
	
	The above argument \eqref{eq:longargument1} can be applied here with the only difference that a multiple of the following term
	$$ p \circ \widecheck{\mu}^2 \left( i(c_{\tau_0,\tau_1}(y'')) \otimes i(x) \right)$$
	is added to the lines following from the third line in \eqref{eq:longargument1}.
	But the above term is zero since the inclusion $i$ induces an $A_\infty$-homomorphism as argued in Theorem \ref{thm:ainfinity} and hence the above term equals
	$$ 0=p\circ i \circ \mu^2(c_{\tau_0,\tau_1}(y'')\otimes x).$$
	This completes the proof.
\end{proof}

	\begin{rmk}
		To understand Lemma \ref{lem:chainmap} more intuitively, recall that, for any oriented manifold $L$ with boundary, the singular cochain complex $C^*(L;\mathbb{K})$ together with the coboundary map and the cup product defines a dg algebra and the relative singular cochain complex $C^*(L,\partial L;\mathbb{K})$ carries a dg module structure over $C^*(L;\mathbb{K})$ given by the cup product:
		\begin{equation}\label{eq:dgmodule}
		C^*(L,\partial L;\mathbb{K}) \otimes C^*(L;\mathbb{K}) \to C^*(L,\partial L;\mathbb{K}).
		\end{equation}
		The map $\eta : CW_*(L_1,L_0) \otimes CW^*(L_0,L_1) \to CW_*(L_0,L_0)$ defined in the proof of Lemma \ref{lem:chainmap} is a Floer-theoretic enhancement of the cup product \eqref{eq:dgmodule}. Since the cup product is a chain map, it is reasonable to expect $\eta$ to be a chain map as well.
	\end{rmk}

As a consequence of Lemma \ref{lem:chainmap}, $\alpha$ and $\gamma$ induce the pairings
\begin{align*}
	\alpha^k : HW_k(L_1,L_0) \times HW^k(L_0,L_1)& \to \mathbb{K},\\
	\gamma^k: HW^k(L_1,L_0) \times HW_k(L_0,L_1)& \to \mathbb{K}
\end{align*}
for each $k\in \mathbb{Z}$.

Furthermore, Lemma \ref{lem:chainmap} allows us to consider the following chain maps:
$$\overline{\alpha} : CW^*(L_0,L_1) \to CW_*(L_1,L_0)^{\vee} \text{ and } \overline{\gamma} : CW_*(L_0,L_1) \to CW^*(L_1,L_0)^{\vee} $$ defined by
\begin{align*}
	\overline{\alpha}(x) (y) &= \alpha(y \otimes x),\\
	\overline{\gamma}(y) (x) &= \gamma(x \otimes y)
\end{align*}
for all $ x \in CW^*(L_0,L_1)$ and $y \in CW_*(L_1,L_0)$.

Let us denote their induced maps on homologies by
\begin{equation*}
\begin{split}
 \overline{\alpha}^k &: HW^k(L_0,L_1) \to HW_{k}(L_1,L_0)^\vee,\\
   \overline{\gamma}^k &: HW_k(L_0,L_1) \to HW^{k}(L_1,L_0)^\vee
\end{split}
\end{equation*}
for each $k\in \Z$.\\

Now let us turn into the proof that both $\overline{\alpha}$ and $\overline{\gamma}$ induce isomorphisms on homologies assuming that the condition \eqref{condition3} of Theorem \ref{thm:maintheorem} is satisfied. For that purpose, first observe that, for each $w\in \mathbb{Z}_{\geq 1}$, there are embeddings
$$ CF^*(L_0, L_1;H_{\tau_w}) \subset CW^*(L_0,L_1) \text{ and } CF^{n-*}(L_1,L_0; H_{\tau_{-w}} ) \subset CW_{*}(L_1,L_0)$$
as subspaces. This allows us to consider the restriction of the product $\widecheck{\mu}^2$ to $CF^{n-*}(L_1,L_0;H_{\tau_{-w}})\otimes CF^*(L_0,L_1;H_{\tau_w})$ given by
$${\mu}^{2,\emptyset,(0,w,-w)}_B : CF^{n-*}(L_1,L_0;H_{\tau_{-w}}) \otimes CF^*(L_0,L_1;H_{\tau_w}) \to CF^{n}(L_0,L_0;H_{\tau_0}).$$

But, for $x_0 \in CF^{n}(L_0,L_0;H_{\tau_0})$, $x_1 \in CF^{*}(L_0,L_1; H_{\tau_w})$ and $x_2 \in CF^{n-*}(L_1,L_0;H_{\tau_{-w}})$, the space $\mathcal{R}^{3,\emptyset,(0,w,-w)}(x_0,x_1,x_2)$ of popsicle maps coincides with the space of pair-of-pants curves with one negative end $x_0$ and two positive ends $x_1$ and $x_2$. Therefore $\mu^{2,\emptyset,(0,-w,w)}_B$ coincides with the pair-of-pants product up to sign. 

Similarly, the product $\widecheck{\mu}^2$ restricts to
$$\mu^{2,\emptyset,(0,-w,w)}_B : CF^{n-1-*}(L_1, L_0;H_{\tau_w}) \otimes CF^{*+1} (L_0,L_1;H_{\tau_{-w}} )\to CF^{n} (L_0,L_0;H_{\tau_0}),$$
which also coincides with the pair-of-pants product up to sign. Let
\begin{equation}\label{eq:tildepairing}
\begin{split}
	\widetilde{\alpha}_w &: CF^{n-*}(L_1,L_0: H_{\tau_{-w}}) \otimes CF^* (L_0,L_1;H_{\tau_{w}}) \to \mathbb{K},\\ \widetilde{\gamma}_w&: CF^{n-1-*}(L_1,L_0;H_{\tau_{w}}) \otimes CF^{*+1}(L_0,L_1;H_{\tau_{-w}}) \to \mathbb{K}
\end{split}
\end{equation}
be the chain maps defined by
\begin{align*}
	\widetilde{\alpha}_w (y\otimes x) &= \langle [L_0,\partial L_0], \mu^{2,\emptyset,(0,w,-w)}_B (y\otimes x) \rangle,\\
	\widetilde{\gamma}_w (x \otimes y)&= \langle [L_0,\partial L_0], \mu^{2,\emptyset,(0,-w,w)}_B (x\otimes y) \rangle.
\end{align*}
Let us denote the induced pairings on homologies by
\begin{align*}
	\widetilde{\alpha}^k_w &: HF^{n-k}(L_1,L_0: H_{\tau_{-w}}) \times HF^k (L_0,L_1;H_{\tau_{w}}) \to \mathbb{K},\\ \widetilde{\gamma}^k_w&: HF^{k}(L_1,L_0;H_{\tau_{w}}) \times HF^{n-k}(L_0,L_1;H_{\tau_{-w}}) \to \mathbb{K}
\end{align*}
for each $k \in \Z$ by abuse of notation.
\begin{rmk}
In general, these pairings $\widetilde{\alpha}_{w+1}^k$ and $\widetilde{\gamma}_{w+1}^{n-k}$ are not non-degenerate since $\tau_{-w}$ may not be close to $-\tau_w$ so that $[\tau_{-w}, -\tau_w] \cap \mathrm{Spec} (\partial L_1,\partial L_0) \ \neq \emptyset$. This is the reason why we choose other slopes $\sigma_{-w}$ in the proof of Lemma \ref{lem:limitnondegenerate}. 
\end{rmk}

The pairings $\widetilde{\alpha}_{w+1}^k$ and $\widetilde{\gamma}_{w+1}^{n-k}$ are compatible with the continuation maps in the sense that
\begin{equation}\label{eq:compatiblitywithcontinuation}
	\begin{split}
	\widetilde{\alpha}_{w+1}^k ( y , c_w(x) ) &= \widetilde{\alpha}_w^k ( c_{-(w+1)} (y) , x),\\
	\widetilde{\gamma}_{w+1}^{n-k} (c_w(x) , y) &= \widetilde{\gamma}_w^{n-k} ( x , c_{-(w+1)}(y) )
	\end{split}
\end{equation}
for all $w\in \Z_{\geq 0}$, $k\in \Z$, $x\in HF^k(L_0,L_1; H_{\tau_w})$ and $y \in HF^{n-k}(L_1,L_0;H_{\tau_{-(w+1)}})$. This implies that these induce pairings on limits as follows.
\begin{equation*}
\begin{split}
	\widetilde{\alpha}^k &: \varprojlim_{w \to \infty}  HF^{n-k}(L_1,L_0;H_{\tau_{-w}}) \times \varinjlim_{w \to \infty} HF^k (L_0,L_1; H_{\tau_w}) \to \mathbb{K},\\  
	\widetilde{\gamma}^{k} &:  \varinjlim_{w \to \infty} HF^{k} (L_1,L_0; H_{\tau_w}) \times  \varprojlim_{w \to \infty}  HF^{n-k}(L_1,L_0;H_{\tau_{-w}})  \to \mathbb{K}.
	\end{split}
\end{equation*}

Finally, these induce maps
\begin{equation*}\label{eq:limitdual}
\begin{split}
	\overline{\widetilde{\alpha}}^k &:  \varinjlim_{w \to \infty} HF^k (L_0,L_1; H_{\tau_w}) \to (\varprojlim_{w \to \infty}  HF^{n-k}(L_1,L_0;H_{\tau_{-w}}))^{\vee},\\  
	\overline{\widetilde{\gamma}}^{k} &:  \varprojlim_{w \to \infty}  HF^{n-k}(L_1,L_0;H_{\tau_{-w}})  \to (\varinjlim_{w \to \infty} HF^{k}(L_1,L_0; H_{\tau_w}))^{\vee}.
	\end{split}
\end{equation*}

\begin{lem}\label{lem:limitnondegenerate}
	\mbox{}
	\begin{enumerate}
	\item	For $k\in \Z$, if $\dim HW^k(L_0,L_1) = \dim  \varinjlim_{w \to \infty} HF^k (L_0,L_1; H_{\tau_w})  < \infty$, then $\overline{\widetilde{\alpha}}^k$ is an isomorphism.
	\item For $k \in \Z$, $\overline{\widetilde{\gamma}}^k$ is an isomorphism.
	\end{enumerate}
\end{lem}
\begin{proof}	
	Let us prove the first assertion first.
	For each $w\in \Z_{\geq 1}$, we may choose $\sigma_{-w} \in \mathbb{R}$ in such a way that
	\begin{itemize}
	\item $ \tau_{-w} \leq \sigma_{-w}  \leq -\tau_w +\tau_0$,
	\item $[\sigma_{-w}, -\tau_w] \cap \mathrm{Spec} (\partial L_1, \partial L_0) = \emptyset$.
	\end{itemize}
	Then the pairing $a^k_w : HF^{n-k}(L_1,L_0; H_{\sigma_{-w}}) \times HF^k (L_0,L_1;H_{\tau_{w}}) \to \mathbb{K}$ given as above is non-degenerate as proved in \cite[Section 6.1]{cie-oan20}. This implies that there is an isomorphism
	$$\overline{a}^k_w : HF^k (L_0,L_1;H_{\tau_{w}}) \cong HF^{n-k}(L_1,L_0; H_{\sigma_{-w}})^{\vee},$$
	which makes sense since both $HF^* (L_0,L_1;H_{\tau_{w}})$ and $HF^{n-*}(L_1,L_0; H_{\sigma_{-w}})$ are finite-dimensional.
	
	Furthermore, $\sigma_{-w}$ can be chosen in such a way that the sequence $(\sigma_{-w})_{w \in \Z_{\geq 1}}$ is decreasing so that we may identify
	\begin{equation}\label{eq:identification}
	\varprojlim_{w \to \infty}  HF^{n-k}(L_1,L_0;H_{\tau_{-w}}) = \varprojlim_{w \to \infty}  HF^{n-k}(L_1,L_0;H_{\sigma_{-w}}).
	\end{equation}
	
	Then since the isomorphisms $\overline{a}^k_w$ are compatible with continuation maps once again, we have
	\begin{equation*}
	\overline{a}^k : \varinjlim_{w \to \infty} HF^k (L_0,L_1; H_{\tau_w}) \to (\varprojlim_{w \to \infty}  HF^{n-k}(L_1,L_0;H_{\sigma_{-w}}))^{\vee}.
	\end{equation*}
	Here note that the assumption $\dim  \varinjlim_{w \to \infty} HF^k (L_0,L_1; H_{\tau_w}) < \infty$ implies that 
	\begin{equation*}\label{eq:doubledual}
	\begin{split}
	 \varinjlim_{w \to \infty} HF^k (L_0,L_1; H_{\tau_w}) &\cong \left((\varinjlim_{w \to \infty} HF^k (L_0,L_1; H_{\tau_w}))^{\vee}\right)^{\vee} \\
	 &\cong \left(\varprojlim_{w \to \infty} HF^k (L_0,L_1; H_{\tau_w})^{\vee}\right)^{\vee}.
	\end{split}
	\end{equation*}
	Since $HF^{n-k}(L_1,L_0;H_{\sigma_{-w}})$ is identified with $HF^k (L_0,L_1; H_{\tau_w})^{\vee}$ via $(\overline{a}_w^k)^{\vee}$ for each $w$ and the dual of the continuation map $c_{\tau_w,\tau_{w+1}} : HF^k (L_0,L_1;H_{\tau_w})\to HF^k (L_0,L_1;H_{\tau_{w+1}})$ is identified with the continuation map $c_{-\sigma_{w+1},-\sigma_w}$ under the identification $(\overline{a}_w^k)^{\vee}$ due to a compatibility as in \eqref{eq:compatiblitywithcontinuation}, this implies that 
	$$\varinjlim_{w \to \infty} HF^k (L_0,L_1; H_{\tau_w}) \cong \left(\varprojlim_{w \to \infty} HF^{n-k} (L_1,L_0; H_{\sigma_{-w}})\right)^{\vee}.$$
	Finally \eqref{eq:identification} proves the first assertion.
	
	The second assertion can be proved similarly even without any assumption on dimension since the dual of a direct limit is given by the inverse limit of duals in general.
\end{proof}

Now we are ready to prove the following lemma.
\begin{lem}\label{lem:wrappednondegenerate}
	\mbox{}
	\begin{enumerate}
		\item For $k\in \mathbb{Z}$, if $\dim HW^k(L_0,L_1) < \infty$, $\overline{\alpha}^k : HW^k (L_0,L_1) \to HW_k(L_1,L_0)^{\vee}$ is an isomorphism.
		\item For $k\in \mathbb{Z}$, $\overline{\gamma}^k : HW_k (L_0,L_1) \to  (HW^k(L_1,L_0))^\vee$ is an isomorphism.
		\end{enumerate}
\end{lem}
\begin{proof}
	Let us prove the assertion for $\alpha$ only again since the other one can be proved similarly.
	
	Due to Lemma \ref{lem:limitnondegenerate}, it suffices to show that the map $\overline{\alpha}^k$ is identified with $\overline{\widetilde{\alpha}}^k$ under the identification $HW^k(L_0,L_1) =  \varinjlim_{w \to \infty} HF^k (L_0,L_1; H_{\tau_w})$ and $HW_k(L_1,L_0) = \varprojlim_{w\to \infty} HF^{n-k}(L_1,L_0;H_{\tau_{-w}})$.
	
	For that puprose, for each $w\in \mathbb{Z}_{>0}$, let $E_w \xrightarrow{\iota_w} CW^*(L_0,L_1)$ be a subcomplex defined by
	$$E_w = \bigoplus_{u=1}^{w} CF^*(L_0,L_1;H_{\tau_v}) \oplus \bigoplus_{u=1}^{w-1} CF^{*+1} (L_0,L_1;H_{\tau_v}) q.$$
	Then we have
	$$ CW^*(L_0,L_1) = \varinjlim_{w \to \infty} E_w,$$
	where the direct limit is taken with respect to the natural inclusions 
	$$E_w \to E_{w+1}.$$
	
	Furthermore, in the proof of Lemma \ref{lem:inverselimit}, we have introduced a quotient complex $CW_*(L_1,L_0) \xrightarrow{\pi_w} Q_w$ for a subcomplex $C_w \subset CW_*(L_1,L_0)$ for each $w\in \mathbb{Z}_{\geq 0}$. These quotient complexes admit further quotient maps
	$$Q_{w+1} \to Q_w.$$
	
	For each $w\in \mathbb{Z}_{>0}$, observe that the restriction of $\alpha$ to
	$$ C_w \otimes E_w \subset CW_*(L_1,L_0)\otimes CW^*(L_0,L_1)$$
	is zero since any element of the image $p \circ \widecheck{\mu}^2( C_w \otimes E_w)$ has a zero summand from $CF^*(L_0,L_0;H_{\tau_0})$. 
	Therefore $\alpha$ induces a chain map
	$$\alpha_{w} : Q_w \otimes E_w \to \mathbb{K},$$
	which further induces a chain map
	$$\overline{\alpha}_{w} : E_w \to Q_w^{\vee}.$$

	For each $0\leq v<w$, let $E^v_w \subset E_w$ be a subcomplex given by
	$$ E^v_w = \bigoplus_{u=v+1}^w CF^*(L_0,L_1;H_{\tau_u}) \oplus \bigoplus_{u=v+1}^{w-1} CF^{*+1}(L_0,L_1;H_{\tau_u})q.$$
	
	As argued in the proof of Lemma \ref{lem:inverselimit}, since the quotient complex of the inclusion $E^{v}_w \to E^{v-1}_w$ is acyclic, all the inclusions
	$$ CF^{*} (L_0,L_1;H_{\tau_w}) \cong E^{w-1}_w \hookrightarrow E^{w-2}_w \hookrightarrow \dots \hookrightarrow E^0_w = E_w$$
	are quasi-isomorphisms.
	
	Recall that for any $0\leq v<w$, there is a further quotient complex
	$$ Q_w^v = CW_*(L_0,L_1)/(C_w\oplus D^v)$$
	for some subcomplex $D^v \subset CW_*(L_0,L_1)$. Then once again, one can show that the restriction of $\alpha$ to
	$$ (C_w \oplus D^v) \otimes E^v_w \subset CW_*(L_1,L_0)\otimes CW^*(L_0,L_1) $$
	is zero and hence $\alpha$ induces a chain map
	$$ \alpha^v_w : Q^v_w \otimes E^v_w \to \mathbb{K}.$$
	This further induces a chain map
	$$ \overline{\alpha}^v_w : E^v_w \to (Q^v_w)^{\vee}.$$
	
	In particular, identifying
	$$ E^{w-1}_w = CF^*(L_0,L_1;H_{\tau_w}) \text{ and } Q^{w-1}_w = CF^{n-*}(L_1,L_0; H_{\tau_{-w}}),$$
	the chain map $$\alpha^{w-1}_w : CF^{n-*}(L_1,L_0; H_{\tau_{-w}}) \otimes CF^*(L_0,L_1;H_{\tau_w}) \to \mathbb{K}$$
	is given by
	$$ \alpha^{w-1}_w ( y \otimes x) = \langle [L_0,\partial L_0], \mu^{2,\emptyset,(0,w,-w)}_B (y\otimes x)\rangle,$$
	which is the same as the chain map $\widetilde{\alpha}_w$ discussed in \eqref{eq:tildepairing}.

	Now observe that the following diagram commutes:
	\begin{equation}\label{eq:commutativediagram1}
		\xymatrix{  CF^*(L_0,L_1;H_{\tau_w})=E_w^{w-1} \ar[r]^{\qquad\qquad \simeq} \ar[d]^{\overline{\alpha}^{w-1}_w} &E^{w-2}_w \ar[d]^{\overline{\alpha}^{w-2}_w} \ar[r]^{\quad\simeq} & \cdots \ar[r]^{\simeq\qquad} & E^0_w =E_w \ar[d]^{\overline{\alpha}^{0}_w}\\CF^{n-*}(L_1,L_0; H_{\tau_{-w}})^{\vee}= (Q^{w-1}_w)^{\vee} \ar[r]^{\qquad \qquad \qquad \simeq}& (Q^{w-2}_w)^{\vee} \ar[r]^{\quad \simeq} &\cdots \ar[r]^{\simeq \qquad}& (Q^0_w)^{\vee},}
	\end{equation}
	since all $\alpha_{w}^{v}$ come from the restriction of the pairing $\alpha$. 
	
	Here recall from the proof of Lemma \ref{lem:inverselimit} that $\phi_{w}^{-1} : Q_w = Q^{-1}_w \to  Q^0_w$ is a quasi-isomorphism and it commutes with the quotient maps $Q_{w+1}^{\cdot} \to Q_w^{\cdot}$. As a consequence, we have the following quasi-isomorphism 
	\begin{equation*}
		\varinjlim_{w\to \infty} (Q^0_w)^{\vee} \xrightarrow{\lim_{w}(\phi_w^{-1})^{\vee}}  \varinjlim_{w \to \infty} (Q^{-1}_w)^{\vee} =\varinjlim_{w\to \infty} Q_w^{\vee}. 
	\end{equation*}
	This further fits into the following commutative diagram
	\begin{equation}\label{eq:commutativediagram2}
			\xymatrix@C+1pc{ &\varinjlim_{w\to \infty} E_w \ar[dl]_{\lim_w \overline{\alpha}^0_w} \ar[d]^{\lim_w \overline{\alpha}_w} \ar[r]^{\simeq} & CW^*(L_0,L_1)\ar[d]^{\overline{\alpha}} \\   \varinjlim_{w\to \infty} (Q_w^0)^{\vee} \ar[r]_{\lim_{w}(\phi_w^{-1})^{\vee}} & \varinjlim_{w\to \infty} Q_w^{\vee} \ar[r]& CW_*(L_1,L_0)^{\vee},}
	\end{equation}
	where the bottom-right arrow from $\varinjlim_{w\to \infty} Q_w^{\vee}$ to $CW_*(L_1,L_0)^{\vee}$ is the natural inclusion map $\varinjlim_{w\to \infty} Q_w^{\vee} \to (\varprojlim_{w \to \infty}Q_w)^\vee = CW_*(L_1,L_0)^{\vee}.$ 
	
	Finally, under the assumption that $\dim HW^k(L_0,L_1) <\infty$ for all $k\in \Z$, the commutative diagrams \eqref{eq:commutativediagram1} and \eqref{eq:commutativediagram2} prove that the bottom-right arrow is a quasi-isomorphism and furthermore that $\overline{\alpha}^k$ is identified with $\overline{\widetilde{\alpha}}^k$ as asserted.
\end{proof}

	\subsection{Construction of $\overline{\beta}$}
Now let us define a chain map
$$ \beta = \beta_{L_0,L_1} : RFC^{n-1-*} (L_1,L_0) \otimes RFC^*(L_0,L_1) \to \mathbb{K} $$
by
$$ \beta( y\otimes x) = \langle [L_0,\partial L_0], \pi \circ p \circ \widecheck{\mu}^2( y\otimes x) \rangle$$
for all $x\in  RFC^*(L_0,L_1), y \in RFC^{n-1-*} (L_1,L_0)$. Furthermore, we define a chain map
$$ \overline{\beta} : RFC^*(L_0,L_1) \to RFC^{n-1-*}(L_1,L_0)^{\vee}$$
by 
$$ \overline{\beta}(x)(y) = \beta(y\otimes x)$$
for all $x \in RFC^*(L_0,L_1)$ and $y\in RFC^{n-1-*}(L_1,L_0)$.

Then we have the following lemma.
\begin{lem}\label{lem:commutativity}
	All the squares in the diagram \eqref{eq:morphismbewteenses} commute.
\end{lem}
\begin{proof}
	Let us prove the assertion for the first square only since that for the other square can be proved similarly.
	The commutativity of the first square means
	$$ p^{*} \circ \overline{\alpha}  =  \overline{\beta} \circ i,$$
	or equivalently
	$$ \alpha ( p(y) \otimes x) = \beta(y \otimes i(x)) $$
	for all $x\in CW^*(L_0,L_1)$ and $y \in RFC^{n-1-*}(L_1,L_0)$. 
	
	Let $x\in CW^*(L_0,L_1)$ and $y \in RFC^{n-1-*}(L_1,L_0)$ be given. Then $y$ can be written as
	$$ y = i(y') + p^{-1}(y'')$$
	for some $y' \in CW^{n-1-*}(L_1,L_0)$ and $y'' = p(y) \in CW_{*}(L_1,L_0)$. Then, as in the proof of Lemma \ref{lem:chainmap}, since the inclusion $i$ induces $A_\infty$-homomorphism, we have
	\begin{align*}
		p\circ \widecheck{\mu}^2( y \otimes i(x)) &= p\circ \widecheck{\mu}^2 ( i(y') + p^{-1}(y''), i(x) ) \\
		&= p \circ \widecheck{\mu}^2 ( p^{-1}(y''), i(x) ). 
	\end{align*}
	
	Consequently, we have
	\begin{align*}
		\beta( y \otimes i(x) ) &= \langle [L_0,\partial L_0], \pi \circ p \circ \widecheck{\mu}^2( y\otimes i(x)) \rangle \\
		&= \langle [L_0,\partial L_0], \pi \circ p \circ \widecheck{\mu}^2( p^{-1}(y'') \otimes i(x)) \rangle\\
		&= \alpha ( y'' \otimes x) = \alpha (p(y) \otimes x).
	\end{align*}
	This proves the assertion.
\end{proof}

Combining Lemma \ref{lem:wrappednondegenerate} and Lemma \ref{lem:commutativity}, we conclude that 
\begin{itemize}
\item $\dim RFH^k(L_0,L_1) < \infty$ for all $k \in \mathbb{Z}$ and
\item  there exists a quasi-isomorphism
\begin{equation}\label{eq:beta}
	\overline{\beta} : RFC^*(L_0,L_1) \to RFC^{n-1-*}(L_1, L_0)^{\vee}
\end{equation}
\end{itemize}
assuming $\dim HW^{k} (L_0, L_1) <\infty$ for all $k \in \mathbb{Z}$.

This implies that $\overline{\beta}$ induces a non-degenerate pairing 
$$ \beta^k : RFH^{n-1-k}(L_1,L_0) \times RFH^k(L_0,L_1) \to \mathbb{K}$$
for every $k \in \Z$.

\subsection{Extension of $\overline{\beta}$ to the triangulated envelope}
For the rest of this subsection, we assume that the condition \eqref{condition3} of Theorem \ref{thm:maintheorem} is satisfied.
Combining Theorem \ref{thm:ainfinity} with the assumption (in particular, that $\{L_i\}_{i \in I}$ generates $\mathcal{W}_b(M)$), we see that the triangulated envelope of $\mathcal{RW}$ (usually constructed as the category of one-sided twisted complexes on $\mathcal{RW}$) coincides with the dg category $\mathrm{Perf}\,\mathcal{RW}$ of perfect $A_\infty$-modules over $\mathcal{RW}$.

Now fix an admissible Lagrangian $L$ and consider $A_\infty$-functors
$$RFC^*(L,-),RFC^{n-1-*}(-,L)^\vee : \mathcal{RW} \to \mathrm{Mod}\,\mathbb{K}.$$
Then one can show that the quasi-isomorphisms \eqref{eq:beta} assemble to a natural quasi-isomorphism
$$RFC^*(L,-) \to RFC^{n-1-*}(-,L)^{\vee}$$
between them.
More concretely, for $k\in \Z_{\geq 1}$, we define
$$\overline{\beta}^k : RFC^*(L_{k-1},L_k) \otimes \cdots \otimes RFC^*(L_0,L_1) \to \mathrm{Hom}_{\mathrm{Mod}\,\mathbb{K}}(RFC^*(L,L_0),RFC^{n-1-*}(L_k,L)^\vee)[-k]$$
by
$$\overline{\beta}^k(z_1,\dots,z_k)(x)(y) = (-1)^{\deg z_1 + \dots + \deg z_k -k} \langle [L,\partial L],\pi \circ p \circ \widecheck{\mu}^{k+2}(y \otimes z_k \otimes \cdots \otimes z_1 \otimes x) \rangle$$
where $z_i \in RFC^*(L_{i-1},L_i),x \in RFC^*(L,L_0)$ and $y \in RFC^{n-1-*}(L_k,L)$.

Using \cite[Section (3m)]{sei08}, we can extend this to a natural quasi-isomorphism
$$\mathrm{Hom}_{\mathrm{Perf}\,\mathcal{RW}}(L,-) \to \mathrm{Hom}_{\mathrm{Perf}\,\mathcal{RW}}(-,L[n-1])^\vee$$
where $\mathrm{Hom}_{\mathrm{Perf}\,\mathcal{RW}}(L,-)$ and $\mathrm{Hom}_{\mathrm{Perf}\,\mathcal{RW}}(-,L[n-1])^\vee$ are regarded as dg functors from $\mathrm{Perf}\,\mathcal{RW}$ to $\mathrm{Perf}\,\mathbb{K}$.

Similarly, fix an object $L'$ of $\mathrm{Perf}\,\mathcal{RW}$ and consider $A_\infty$-functors
$$\mathrm{Hom}_{\mathrm{Perf}\,\mathcal{RW}}(-,L'),\mathrm{Hom}_{\mathrm{Perf}\,\mathcal{RW}}(L',-[n-1])^\vee : \mathcal{RW}^\mathrm{op} \to \mathrm{Mod}\,\mathbb{K}.$$
As above, one can see that the quasi-isomorphisms \eqref{eq:beta} can eventually be extended to a natural quasi-isomorphism
$$\mathrm{Hom}_{\mathrm{Perf}\,\mathcal{RW}}(-,L') \to \mathrm{Hom}_{\mathrm{Perf}\,\mathcal{RW}}(L',-[n-1])^\vee$$
where $\mathrm{Hom}_{\mathrm{Perf}\,\mathcal{RW}}(-,L')$ and $\mathrm{Hom}_{\mathrm{Perf}\,\mathcal{RW}}(L',-[n-1])^\vee$ are regarded as dg functors from $\mathrm{Perf}\,\mathcal{RW}^\mathrm{op}$ to $\mathrm{Perf}\,\mathbb{K}$.

Recall that, for an $A_\infty$-category $\mathcal{A}$, its (idempotent complete) derived category $D^\pi(\mathcal{A})$ is defined by the idempotent completion of the localization of $H^0(\mathrm{Mod}\,\mathcal{A})$ with respect to $A_\infty$-quasi-isomorphisms. Here $\mathrm{Mod}\,\mathcal{A}$ denotes the dg category of $A_\infty$-modules over $\mathcal{A}$. For more details, see \cite[Section 2]{bjk22}.

Consequently, we have
\begin{thm}\label{thm:maintheorem2}
Let $M$ be a $2n$-dimensional Liouville domain.
If the condition \eqref{condition3} of Theorem \ref{thm:maintheorem} is satisfied, $D^\pi\mathcal{RW}(M)$ is a $(n-1)$-Calabi--Yau triangulated category.
\end{thm}

This completes the proof of Theorem \ref{thm:maintheorem}.

\section{Examples}\label{section:examples}
In this section, we provide examples of Liouville domains for which condition \eqref{condition3} of Theorem \ref{thm:maintheorem} is satisfied. As a result, for such a Liouville domain, its derived Rabinowitz Fukaya category is Calabi--Yau.
\subsection{Cotangent bundles of manifolds with a finite fundamental group}
Let $N$ be a connected closed smooth manifold whose fundamental group is finite. For any field $\mathbb{K}$ of characteristic 0, the Serre spectral sequence  \cite{hat02} for the fibration $\Omega_{x}(N) \hookrightarrow \mathcal{P} (N) \to N$ is a spectral sequence $\{E^{r+1}_{p,q},d^r\}$ converging to $H_*(\mathcal{P}(N);\mathbb{K})$ such that
$$E^2_{p,q} \cong H_p(N; H_q(\Omega_{x} (N);\mathbb{K})).$$
Here $\mathcal{P}(N) =\mathrm{Map} \left(([0,1],0),(N,x)\right)$ is the space of based paths in $(N,x)$ and the map $\mathcal{P}(N) \to N$ is the evaluation map at $1\in [0,1]$. Consequently, the fiber of the evaluation map over a point in $N$ is homotopy equivalent to the space $\Omega_{x} (N)$ of loops in $N$ based at a point $x \in N$.

Considering that the path space $\mathcal{P}(N)$ is a finite union of contractible components, the above spectral sequence actually shows that the homology $H_k (\Omega_{x} (N);\mathbb{K})$ is finite-dimensional for all $x\in N$ and $k \in \Z_{\geq 0}$. This can be deduced by a result of Serre \cite{ser51}.

 Indeed, this can be proved by an induction on $k$. First, for $k=0$, the assumption that $\pi_1(N)$ is finite implies that $\Omega_{x} (N)$ is a finite union of path-connected components and hence $\dim H_0(\Omega_{x} (N);\mathbb{K})$ is finite. Suppose that $\dim H_j (\Omega_{x}(N);\mathbb{K}) <\infty$ for all $0 \leq j <k$ for some $k\in \Z_{\geq 1}$. Then, considering that $\dim E^{\infty}_{0,k} = 0$ and that the differential $d^r$ changes the bidegree $(p,q)$ into $(p-r,q+r-1)$ on the $r$-th page of the Serre spectral sequence, we have
$$\dim H_k (\Omega_{x}(N);\mathbb{K}) \leq \sum_{1<j \leq \dim N} \dim H_{k+1-j} (\Omega_{x}(N);\mathbb{K}) \cdot \dim H_{j}(N;\mathbb{K}).$$
Here we used the fact that $\dim H_1(N;\mathbb{K}) = 0$. The assertion follows from this inequality.

Now recall the following two theorems.
\begin{thm}[\cite{abo11}]
A cotangent fiber $T^*_{x} N$ generates the wrapped Fukaya category $\W_b(T^*N)$, where $b$ is the pullback of the second Stiefel--Whitney class of $N$ via the projection $T^* N\to N$.
\end{thm}
\begin{thm}[\cite{abb-sch10,abo12}]
There is an $A_{\infty}$-isomorphism
$$ HW^*(T^*_{x} N,T^*_{x} N) \cong H_{-*} (\Omega_{x} (N);\mathbb{K}).$$
\end{thm}

These two theorems and the observation that $H_{k} (\Omega_{x} (N);\mathbb{K})$ is finite-dimensional for every $k\in \mathbb{Z}_{\geq 0}$ say that $\W_b(T^*N)$ is generated by $T^*_{x} N$ and its endomorphism algebra $ HW^{-k}(T^*_{x} N,T^*_{x} N)$ is finite-dimensional for every $k \in \mathbb{Z}$. Thus, the wrapped Fukaya category $\W_b(T^*N)$ satisfies the condition \eqref{condition3} of Theorem \ref{thm:maintheorem} for the generator $T^*_{x} N$.

\subsection{Plumbings of $T^*S^{n}$ along trees}
Let $Q=(Q_0,Q_1)$ be a finite quiver whose underlying graph is a tree. For an integer $n \geq 3$, let $X^{n}_Q$ be the plumbing of cotangent bundles $T^* S^{n}$ along the underlying graph of $Q$. Indeed, for each $v\in Q_0$, we associate a copy of the cotangent bundle $T^* S^{n}$, which we denote by $T^* S_v$. Then we take the plumbing between $T^*S_v$ and $T^* S_w$ whenever $v$ and $w$ are connected by an arrow of $Q$. The resulting space is $X^{n}_Q$. See \cite{etg-lek17} for a construction of such a space.

For each vertex $v \in Q_0$, let $L_v$ be a cotangent fiber of at a base point on $S_v \subset T^* S_v$. Then it was shown  in \cite{etg-lek17,lek-ued20,asp21} that the $A_{\infty}$-algebra $\oplus_{v,w \in Q_0} CW^*(L_v,L_w)$ is quasi-isomorphic to the $n$-Calabi--Yau Ginzburg dg algebra $\Gamma_Q$ \cite{gin06} associated with $Q$ once the cotangent fibers $L_v$ are graded in a certain way.

To describe the Ginzburg dg algebra $\Gamma_Q$, first consider a quiver $\widehat{Q}$, which we will call {\em Ginzburg double} of $Q$, obtained from the quiver $Q$ by adding a reverse arrow $\alpha^*$ for each $\alpha \in Q_1$ and a loop $t_i$ for each $i\in Q_0$. Then, for each $\alpha \in Q_1$, grade $\alpha$ by $\deg \alpha =0$ and its reverse $\alpha^*$ by $\deg \alpha^* = 2-n$ and for each vertex $i\in Q_0$, grade $t_i$ by $\deg t_i = 1-n$. Then, $\Gamma_Q$ is defined by the path algebra $\mathbb{K}\widehat{Q}$ as a graded associative algebra and equipped with a dg algebra structure given by
\begin{itemize}
	\item $d\alpha = 0 = d\alpha^*$ for $\alpha \in Q_0$ and
	\item $dt_i = \sum_{s(\alpha) =i} \alpha^*\alpha - \sum_{t(\beta) = i} \beta \beta^*$. 
\end{itemize}

Since Ginzburg algebra $\Gamma_Q$ satisfies the conditions of \cite[Proposition 2.5]{kal-yan16} when $n \geq 3$, the degree $k$ part $\Gamma^k_Q$ of $\Gamma_Q$ is finite-dimensional for every $k\in \Z$. Since the results of \cite{etg-lek17,cdrgg17,gps20} say that the cotangent fibers $L_v$ generate the wrapped Fukaya category $\W(X^{n}_Q)$, this observation implies that $\W(X^{n}_Q)$ satisfies the condition \eqref{condition3} of Theorem \ref{thm:maintheorem} for the set of Lagrangian cocores $\{L_v\}_{v\in Q_0}$.\\

This case is particularly interesting since it is known that the Rabinowitz Fukaya category $\mathcal{RW}(X^{n}_Q)$ is quasi-equivalent to the quotient category $\W(X^{n}_Q)/\mathcal{F}(X^{n}_Q)$ due to \cite[Proposition 6.4]{ggv22}. On the other hand, we proved that the derived category of the quotient $D^{\pi}(\W(X^{n}_Q)/\mathcal{F}(X^{n}_Q))$ carries a $(n-1)$-Calabi--Yau cluster structure in \cite{bjk22}.

Now we show that the $(n-1)$-Calabi--Yau pairing $\beta_{D^{\pi}(\W/\F)}$ on the derived quotient category $D^{\pi} (\W/\F) \coloneqq D^{\pi} (\W(X^{n}_Q)/\mathcal{F}(X^{n}_Q))$ is equivalent to the pairing $\beta_{D^{\pi}\RW}$ constructed in Section \ref{section:cy} on the derived Rabinowitz Fukaya category $D^{\pi} \RW \coloneqq D^{\pi} \mathcal{RW}(X^{n}_Q)$ in the sense that there exists a nonzero constant $c\in\mathbb{K}$ such that
$$ \beta_{D^\pi(\W/\F)} : \Hom_{D^{\pi}(\W/\F)} ( Y,X[n-1]) \times \Hom_{D^{\pi}(\W/\F)}(X,Y) \to \mathbb{K}$$
is identified with
$$ c\cdot \beta_{D^{\pi}\RW} : \Hom_{D^{\pi}\RW} (\Phi(Y), \Phi(X)[n-1]) \times  \Hom_{D^\pi\RW} (  \Phi(X), \Phi(Y)) \to \mathbb{K}$$
for all objects $X,Y$ of $D^{\pi}(\W/\F)$ under the equivalence $\Phi : D^{\pi}(\W/\F) \to D^{\pi}\RW$ that was shown to exist by \cite{ggv22}. This assertion will be verified by Lemma \ref{lem:equivalence} below.

For that purpose, we first recall some properties of a Calabi--Yau pairing on a triangulated category. Let $\mathcal{C}$ be a triangulated category over a field $\mathbb{K}$ with a $(n-1)$-Calabi--Yau pairing $\beta$.
\begin{dfn}\label{dfn:pairingtriangle}
For any two objects $X,Y$ of $\mathcal{C}$, let $\overline{\beta}_{X,Y} : \Hom_{\mathcal{C}} (X,Y) \to \Hom_{\mathcal{C}}(Y,X[n-1])^{\vee}$ be defined by
$$ \overline{\beta}_{X,Y} (f) (g) = \beta_{X,Y} (g,f)$$
for $f \in \Hom_{\mathcal{C}}(X,Y)$ and $g\in \Hom_{\mathcal{C}}(Y,X[n-1])$.
\end{dfn}

The following lemma follows from the bifunctoriality of Calabi--Yau pairing.
\begin{lem}\label{lem:bifunctoriality}
For any three objects $X$, $Y$, $Z$ of $\mathcal{C}$, $f\in \Hom_{\mathcal{C}} (X,Y)$ and $g\in \Hom_{\mathcal{C}}(Y,X[n-1])$, we have
\begin{enumerate}
\item$\beta_{X,Y} (g,f) = \beta_{X,X[n-1]}(\mathrm{Id}_{X[n-1]}, g\circ f)$.
\item $\beta_{X,Y} (g, f) = \beta_{Y,X[n-1]} ( f[n-1], g).$
\item There exists a commutative diagram between two long exact sequences as follows.
\begin{equation*}\label{eq:commutativediagram}
\begin{tikzcd}
	\cdots \ar[r] &\Hom_{\mathcal{C}} (Z,Y) \ar[r] \ar[d, "\overline{\beta}_{Z,Y}"] & \Hom_{\mathcal{C}} (Z ,\mathrm{Cone}(X \xrightarrow{f} Y)) \ar[r]   \ar[d,"\overline{\beta}_{Z,\mathrm{Cone}(X \xrightarrow{f} Y)} "] & \Hom_{\mathcal{C}} (Z, X[1]) \ar[r]  \ar[d, "\overline{\beta}_{Z,X[1]}"]& \cdots \\
	\cdots \ar[r] &\Hom_{\mathcal{C}} (Y,Z[n-1])^\vee \ar[r] & \Hom_{\mathcal{C}} (\mathrm{Cone}(X \xrightarrow{f} Y),Z[n-1])^{\vee} \ar[r] & \Hom_{\mathcal{C}} (X[1],Z[n-1])^{\vee} \ar[r]& \cdots
	\end{tikzcd}
\end{equation*}

\end{enumerate}
\end{lem}

For the following lemma, for objects $X$ and $Y$ of $\mathcal{C}$, we define
$$\Hom^*_{\mathcal{C}} (X,Y) = \bigoplus_{d \in \Z} \Hom_{\mathcal{C}} ( X, Y[d])[-d].$$
\begin{lem}\label{lem:equivalence}
Let $\mathcal{C}$ be a triangulated category over a field $\mathbb{K}$.
Suppose that $\mathcal{C}$ admits two $(n-1)$-Calabi--Yau pairings $\alpha$ and $\beta$. Suppose further that $\mathcal{C}$ has a collection $\{X_i\}_{i \in I}$ of objects generating $\mathcal{C}$ for some set $I$, such that
\begin{itemize}
\item $\dim \Hom_{\mathcal{C}}(X_i,X_i[n-1])=1$ or equivalently $\dim \Hom_{\mathcal{C}}(X_i,X_i)=1$ for each $i\in I$ and 
\item for any $i,j \in I$, there exists a finite chain $\{{i_k} \in I\}_{k=0}^m$ such that
${i_0} =i$, $i_m = j$ and
$$ \Hom^*_{\mathcal{C}}(X_{i_k},X_{i_{k+1}}) \neq 0$$ for all $0 \leq k \leq m-1$.
\end{itemize}
Then there exists a nonzero constant $c \in \mathbb{K}$ such that
$$ \beta = c \cdot \alpha.$$
\end{lem}
\begin{proof}
For each $i \in I$, let $f_i \in \Hom_{\mathcal{C}}(X_i, X_i[n-1])$ be a generator. Then, since $\alpha_{X_i,X_i[n-1]} (\mathrm{Id}_{X_i[n-1]}, f_i) \neq 0 \neq \beta_{X_i,X_i[n-1]} (\mathrm{Id}_{X_i[n-1]}, f_{i})$, there exists a nonzero constant $c_i \in \mathbb{K}$ such that
\begin{equation}\label{eq1}
\beta_{X_i,X_i[n-1]} (\mathrm{Id}_{X_i[n-1]}, f_{i}) =c_i \alpha_{X_i,X_i[n-1]} (\mathrm{Id}_{X_i[n-1]}, f_i).
\end{equation}
We first show that $c_i = c_j$ for all $i,j \in I$. Since our assumption implies that it suffices to prove the assertion for any $i,j \in I$ such that
$$ \Hom^*_{\mathcal{C}}(X_i,X_j) \neq 0,$$
we assume that $i$ and $j$ satisfy the above condition.

Then there exists $m \in \Z$ such that $\Hom_{\mathcal{C}} (X_i,X_j[m]) \neq 0$. Since $\alpha$ is a $(n-1)$-Calabi--Yau pairing on $\mathcal{C}$, for any nonzero $f_{ij} \in \Hom_{\mathcal{C}} (X_i,X_j[m])$, there exists $f_{ji} \in \Hom_{\mathcal{C}} ( X_j[m], X_i[n-1])$ such that
$$\alpha_{X_i, X_{j}[m]} (f_{ji},f_{ij}) \neq 0.$$

But Lemma \ref{lem:bifunctoriality} (1) implies that
$$\alpha_{X_i, X_{j}[m]} (f_{ji},f_{ij})  = \alpha_{X_i, X_i} (\mathrm{Id}_{X_i[n-1]} ,f_{ji} \circ f_{ij}) \neq 0.$$
This means that $f_{ji} \circ f_{ij} \neq 0 \in \Hom_{\mathcal C} (X_i,X_i[n-1])$ and hence it is a nonzero multiple of $f_i$.
Similarly, $f_{ij}[n-1] \circ f_{ji}$ is a nonzero multiple of $f_j[m]$. Let us say that 
$$ f_i = a_i f_{ji} \circ f_{ij} \text{ and } f_j[m] = a_j f_{ij}[n-1] \circ f_{ji}$$
for some nonzero elements $a_i$ and $a_j$ of $\mathbb{K}$.

Now Lemma \ref{lem:bifunctoriality} (2) says that
$$\alpha_{X_i,X_j[m]} (f_{ji}, f_{ij}) = \alpha_{X_j[m],X_i[n-1]} ( f_{ij}[n-1], f_{ji}) $$

From this, it follows that
\begin{equation}\label{eq2}
\begin{split}
\alpha_{X_i,X_i[n-1]} (\mathrm{Id}_{X_i[n-1]}, f_i) &= a_i \alpha_{X_i,X_i[n-1]} (\mathrm{Id}_{X_i[n-1]}, f_{ji} \circ f_{ij}) \\
&= a_i  \alpha_{X_i,X_j[m]} ( f_{ji} , f_{ij})\\
&= a_i \alpha_{X_j[m], X_i[n-1]} ( f_{ij}[n-1], f_{ji})\\
&= a_i \alpha_{X_j[m], X_j[m+n-1]} (  \mathrm{Id}_{X_j[m]},f_{ij}[n-1]\circ f_{ji})\\
&= a_i a_j^{-1} \alpha_{X_j[m],X_j[m+n-1]} (  \mathrm{Id}_{X_j[m]}, f_j[m])\\
&=  a_i a_j^{-1}  \alpha_{X_j,X_j[n-1]} (  \mathrm{Id}_{X_j[n-1]}, f_j).
\end{split}
\end{equation}
Similarly, we have
\begin{equation}\label{eq3}
\beta_{X_i,X_i[n-1]} (\mathrm{Id}_{X_i[n-1]}, f_i) =   a_i a_j^{-1}  \beta_{X_j,X_j[n-1]} (  \mathrm{Id}_{X_j[n-1]}, f_j).
\end{equation}

Now \eqref{eq1}, \eqref{eq2} and \eqref{eq3} imply that $c_i = c_j$ as desired. Let us now call $c =c_i$ for $i \in I$.
This actually means that
$$\beta_{X_i,X_j[m]} (f_{ji} ,f_{ij})  = c\alpha_{X_i,X_j[m]} (f_{ji}, f_{ij})$$
for all $f_{ij} \in \Hom_{\mathcal{C}} (X_i,X_j[m])$ and $f_{ji} \in \Hom_{\mathcal{C}}( X_j[m],X_i[n-1])$ due to Lemma \ref{lem:bifunctoriality} (1) again.

Since $\{X_i\}_{i\in I}$ is assumed to generate $\mathcal{C}$, any object of $\mathcal{C}$ is obtained as an iterated cone from objects $\{X_i\}$. Then applying Lemma \ref{lem:bifunctoriality} (3) iteratively, one can show that for any object $X$, $Y$ of $\mathcal{C}$, the following holds:
$$\beta_{X,Y} = c \alpha_{X,Y}.$$
The assertion follows.
\end{proof}

Turning back to the case of the derived category $D^{\pi} (\W/\F)$, the cocores $\{L_v\}_{v\in Q_0}$ are known to satisfy
\begin{itemize}
	\item $\dim \Hom_{D^{\pi}(\W/\F)} (L_v ,L_v) = \dim H^0(e_v \cdot \Gamma_Q \cdot e_v) = 1$ and
	\item $\Hom^*_{D^{\pi}(\W/\F)} (L_v,L_w) = H^*(e_w \cdot \Gamma_Q \cdot e_v) \neq 0$
\end{itemize}
for all $v,w\in \widehat{Q}_0 = Q_0$ when $n \geq 3$. Here, $e_v$ is the constant path at $v$ for each $v\in \widehat{Q}_0$.
 Indeed, for each $v,w \in \widehat{Q}_0$, there is a unique shortest path from $v$ to $w$ in the Ginzburg double $\widehat{Q}$. The shortest path 
\begin{itemize}
\item has the minimal length and
\item has the maximal degree
\end{itemize}
 among all paths from $v$ to $w$ when $n \geq 3$. The second property implies that the shortest path is a cocycle and the first property says that it is not a coboundary since, for a given path in $\widehat{Q}$, the differential of $\Gamma_Q$ increases its length or maps it to zero. Therefore, it defines a nonzero cohomology class in $H^*(e_w \cdot \Gamma_Q \cdot e_v)$. In particular, if $v=w$, then the shortest path is the constant path $e_v$ at $v$, which has degree $0$, and all the other loops at $v$ have negative degrees. As a result, Lemma \ref{lem:equivalence} can be applied to prove our assertion.

\subsection{Weinstein domains with periodic Reeb flow on boundary}
Let $(W,\omega=d\lambda)$ be a Weinstein domain of dimension $2n$ for some $n \in \Z_{\geq 2}$ such that $2c_1(W) =0$. 

Let $Z$ denote the Liouville vector field on $(W,\omega=d\lambda)$, i.e., $\iota_Z \omega =\lambda$. Furthermore, let $\xi = \ker \lambda|_{\partial W}$ denote the contact distribution on the boundary $\partial W$.
Since the Reeb vector field $R$ and the Liouville vector field $Z$ are linearly independent everywhere on $TW|_{\partial W}$, the restriction of the tangent bundle $TW$ to the boundary $\partial W$ decompose into
\begin{equation*}\label{eq:decomposition}
TW|_{\partial W} = (\mathbb{R} \langle Z \rangle \oplus \mathbb{R} \langle R \rangle) \bigoplus  \xi.
\end{equation*}
Furthermore, since there is an almost complex structure $J$ on $W$ such that $J$ preserves both summands $(\mathbb{R} \langle Z \rangle \oplus \mathbb{R} \langle R \rangle)$ and $\xi$, the assumption $2c_1(W)=0$ implies that $2c_1(\xi) =0$ as well.
This allows us to choose a quadratic complex volume form 
\begin{equation*}\label{eq:quadratic}
\eta = \eta_{1} \wedge \eta_{2}
\end{equation*}
on $TW$ for some quadratic complex volume forms $\eta_1$ on $(\mathbb{R} \langle Z \rangle \oplus \mathbb{R} \langle R \rangle)$ and $\eta_2$ on $\xi$.

Assume further that the Reeb flow $R$ on the contact boundary $(\partial W, \xi = \ker \lambda|_{\partial W})$ is periodic in the sense that there exists $T >0$ such that
$$\mathrm{Fl}_R^T = \mathrm{Id}_{\partial W},$$
where $\mathrm{Fl}_R^t$ is the time-$t$ flow of the Reeb vector field $R$.
Let $T_0$ be the minimal period among such periods.

Let $\gamma_0 :S^1 = \mathbb{R}/\Z \to \partial W$ be a principal Reeb orbit given by
\begin{equation*}\label{eq:gamma1}
\gamma_0 (t) = \mathrm{Fl}_R^{T_0t} (y)
\end{equation*}
for some $y\in \partial W$. 

As done in \cite{rit13}, let us consider a symplectic trivialization
\begin{equation*}\label{eq:trivialization}
\Psi : (S^1 \times \R^{2(n-1)}, \sum_{j=1}^{n-1} dx_j \wedge dy_j) \cong (\gamma_0^* \xi, (d\lambda|_{\partial W})|_{\xi})
\end{equation*}
such that  
\begin{equation}\label{eq:compatibility}
\Psi^* \eta_2 = (\wedge_{j=1}^{n-1} (dx_j + \sqrt{-1} dy_j))^{\otimes 2},
\end{equation}
where $(x_1,y_1,\dots,x_{n-1},y_{n-1})$ denotes the coordinate of $\R^{2(n-1)}$. Then consider the loop $\Gamma_0 :S^1=\R/\Z \to \mathrm{Sp}(2(n-1))$ of symplectic matrices given by
\begin{equation*}\label{eq:Gamma}
{\Gamma}_0(t) = \Psi^{-1}_t \circ d \mathrm{Fl}_{R}^{tT_0} \circ \Psi_0,
\end{equation*}
where $\Psi_t =\Psi|_{\{t\} \times \R^{2(n-1)}}$.

Let us assume further that the Maslov index $\mu_{T_0}$ given by
\begin{equation}\label{eq:maslovindex}
\mu_{T_0} \coloneqq \mu_{RS} (\mathrm{Graph}\,{\Gamma}_0, \Delta)
\end{equation}
is nonzero. Here the right hand side means the Robbin--Salamon index of the path of Lagrangians subspaces of $(\R^{2(n-1)}\times \R^{2(n-1)}, -\omega_{std}\oplus \omega_{std})$ given by the graphs of ${\Gamma}_0$ with respect to the diagonal Lagrangian subspace $\Delta$.

On the other hand, the wrapped Fukaya category $\W(W)$ is shown to be generated by Lagrangian cocores, say $\{L_i\}_{i=1}^k$, by \cite{cdrgg17,gps20}. We will prove the following proposition in this subsection.
\begin{prop}\label{prop:periodic}
Let $(W,\omega=d\lambda)$ be a Weinstein domain with $2c_1(W) =0$ and with periodic Reeb flow such that the Maslov index $\mu
_{T_0}$ in \eqref{eq:maslovindex} is nonzero. Then the condition \eqref{condition3} of Theorem \ref{thm:maintheorem} is satisfied for $W$ with respect to the Lagrangian cocores $\{L_i\}_{i=1}^k$ and any background class $b\in H^2(W,\Z/2)$. 
\end{prop}

To prove Proposition \ref{prop:periodic}, we need to show that for any pair $1\leq i,j \leq k$ and any degree $d \in \Z$, $HW^{d}(L_i,L_j)$ is finite-dimensional. For that purpose, for a given $1\leq i, j \leq k$, we introduce a new Lagrangian submanifold $L'_j$ obtained from $L_j$ by perturbing near $\partial W$ with a generic Hamiltonian isotopy. To be more precise, we define $L'_j$ to be $\phi_1(L_j)$ for some path $\{\phi_t\}_{t\in [0,1]}$ of Hamiltonian diffeomorphisms on $W$ satisfying
\begin{itemize}
	\item $\phi_0 =\mathrm{Id}$,
	\item $\phi_t^*\lambda = \lambda$ for all $t\in [0,1]$,
 and
	\item that every Reeb chord from $\partial L_i$ to $\partial L'_j$ is non-degenerate.
\end{itemize}
Note that the last requirement can be satisfied even when $i=j$. Since wrapped Floer cohomology is invariant under such Hamiltonian perturbations, we will consider the wrapped Floer cohomology $HW^*(L_i,L'_j)$ instead of $HW^*(L_i,L_j)$.
We may further assume that the primitive $h_{L_i}$ of $\lambda|_{\widehat{L}_i}$ is zero outside a compact subset of $W$ for any $1\leq i \leq k$. This is possible for $n\geq 2$ since the boundary $\partial L_i$ of a Lagrangian cocore $L_i$ is topologically $(n-1)$-dimensional sphere and hence is connected. Then the primitive $h_{L'_j}$ of $\lambda|_{\widehat{L}'_j}$ can be chosen to be zero outside a compact subset of $W$ for any $1\leq j \leq k$ as well.

Let $H_{\nu} : \widehat{W} \to \R$ be an admissible Hamiltonian for $\nu \notin \mathrm{Spec} (\partial L_i, \partial L'_j)$ such that
$$H_{\nu}(r,y) = h_\nu(r), \forall (r,y) \in [1,\infty) \times \partial W$$
 for some function $h_{\nu} : [1,\infty) \to \R$ satisfying
 \begin{itemize}
 \item $h_{\nu}'(r) \geq 0$,
 \item $h'_{\nu}(r) = \nu, r \geq r_{\nu}$ for some $r_{\nu} >1$ and
 \item $h''_{\nu}(r) >0$ for $1< r < r_{\nu}$.
 \end{itemize}
 
 To make our arguments more convenient, for given such Hamiltonians $H_{\mu}$ and $H_{\nu}$ for some $\mu <\nu$, we further consider a certain requirement for the pair $(H_{\mu},H_{\nu})$ following the idea of \cite[Section 8]{kwo-van16}:
\begin{dfn}\label{dfn:good}
Let $\mu < \nu$. A pair ($H_{\mu}$,$H_{\nu}$) of admissible Hamiltonians on $\widehat{W}$ satisfying the above conditions for some functions $h_{\nu}$ and $h_{\mu}$ is said to be a {\em good pair} if
\begin{itemize}
\item $r_{\nu} > r_{\mu}$ and
\item $h_{\nu} = h_{\mu}$ on $\widehat{W} \setminus (r_{\mu}, \infty) \times \partial W$.
\end{itemize}

\end{dfn}

\begin{center}
\begin{figure}[h!]
  \includegraphics[width=2in]{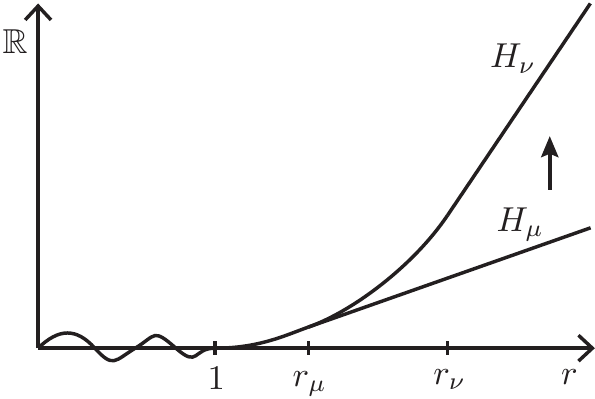}
  \caption{A good pair of Hamiltonians}
  \label{figure2}
\end{figure}
\end{center}

Let $1\leq i, j \leq k$ be given. For an admissible Hamiltonian $H_{\nu}$ as above, any Hamiltonian chord $x \in \mathcal{I}(L_i,L'_j;H_{\nu})$ is one of the following:
\begin{enumerate}
\item[(a)] a short Hamiltonian chord from $L_i$ to $L'_j$ in the interior of $W$ or
\item[(b)] a Hamiltonian chord in $[1,\infty) \times \partial W$ corresponding to a Reeb chord in the sense that
$$x(t) = (r_x, \gamma (t)) \coloneqq \left(r_x, \mathrm{Fl}_R^{t h_{\nu}'(r_x)} ( \gamma (0))\right) \in [1,\infty) \times \partial W, t\in [0,1],$$ 
for some $r_x\in [1,\infty)$ such that $h'_{\nu}(r_x) \in \mathrm{Spec}(\partial L_i,\partial L'_j)$.
\end{enumerate}
The non-degeneracy assumption on Reeb chords from $\partial L_i$ to $\partial L'_j$ and the assumption that $h''_{\nu}(r) >0$ imply that every Hamiltonian chord $x \in \mathcal{I}(L_i,L'_j;H_{\nu})$ of the form (b) is non-degenerate.

For any Hamiltonian chord $x \in \mathcal{I}(L_i,L'_j;H_{\nu})$ of the form (b) for some $r_x \in [1,\infty)$, since the primitive functions $h_{L_i}$ and $h_{L'_j}$ are chosen to be zero near the boundary, the action $\mathcal{A}_{H_{\nu}}(x)$ is given by
$$\mathcal{A}_{H_{\nu}}(x) = -r_xh'_{\nu} (r_x) +h_{\nu} (r_x).$$ 

This leads to consider a function $A_{\nu} : [1,\infty) \to \R$ defined by
$$A_{\nu} (r) = - r h'_{\nu} ( r) + h_{\nu} (r),$$
which gives the action value at the radius $r$. Observe that $A_{\nu}$ is decreasing since $A_{\nu}'(r) = -r h''_{\nu}(r) <0$.

For any good pair $(H_{\mu}, H_{\nu})$ of admissible Hamiltonians on $\widehat{W}$, any Hamiltonian chord $x$ for $H_{\mu}$ is again a Hamiltonian chord for $H_{\nu}$ and their actions coincide, i.e.,
$$\mathcal{A}_{H_{\mu}}(x) = \mathcal{A}_{H_{\nu}}(x).$$
Moreover, $A_{\nu} (r) = A_{\mu}(r)$ for all $r\in [1,r_{\mu}]$.

For any $-\infty \leq  a \leq b  \leq \infty$ and $\nu >0$, consider the set of Hamiltonian chords with actions in the window $(a,b)$:
$$\mathcal{I}_{(a,b)}(L_i,L'_j; H_{\nu})=   \{ x \in \mathcal{I}(L_i,L'_j; H_{\nu}) \;|\; a<  \mathcal{A}_{H_{\nu}} (x)<b \}.$$
Then the above observation implies the following.
\begin{lem}\label{lem:action}
For any good pair $(H_{\mu}, H_{\nu})$ of admissible Hamiltonians on $\widehat{W}$ and any $-\infty< A_{\mu}(r_\mu) \leq a < b \leq \infty$, we have
$$ \mathcal{I}_{(a ,b)}(L_i,L'_j; H_{\mu}) =  \mathcal{I}_{(a ,b)}(L_i,L'_j; H_{\nu}).$$
\end{lem}

For any regular increasing homotopy from $H_{\mu}$ to $H_{\nu}$, the associated continuation from $CF^*(L_i,L'_j;H_{\mu})$ to $CF^*(L_i,L'_j;H_{\nu})$ maps every $x\in \mathcal{I}(L_i,L'_j;H_{\mu})$ to itself $x\in \mathcal{I}(L_i,L'_j;H_{\nu})$. Indeed, this can be deduced by observing that the continuation map does not decrease the action and hence that the constant solution $u: \R \times [0,1] \to \widehat{W}$ given by $u(s,t) = x(t)$ is the unique solution contributing to the continuation map.

Now, for each $l \in \mathbb{Z}_{\geq 1}$, let $\nu_l$ be a real number such that
\begin{itemize}
\item $\nu_l > lT_0$ and
\item $(lT_0,\nu_l] \cap \mathrm{Spec}(\partial L_i,\partial L'_j) = \emptyset$.
\end{itemize}

We may choose a sequence of admissible Hamiltonians $(H_{\nu_l})_{l \in \mathbb{Z}_{\geq 1}}$ on $\widehat{W}$ given as above in such a way that the pair $(H_{\nu_l}, H_{\nu_{l+1}})$ is good for every $l$ in the sense of Definition \ref{dfn:good}. Denote $r_l = r_{\nu_l}$ for each $l \in \Z_{\geq 1}$.

Then we have
\begin{equation}\label{eq:directlimit}
HW^*(L_i,L'_j) =\varinjlim_{l \to \infty} HF^*(L_i,L'_j;H_{\nu_l}).
\end{equation}

For each $l \in \mathbb{Z}_{\geq 1}$, consider the function $A_{\nu_l} : [1,\infty) \to \R$ defined by
$$A_{\nu_l} (r) = - r h'_{\nu_l} ( r) + h_{\nu_l} (r)$$
as above. Then, for each $2\leq m \leq l$, let $I_m^l$ denote
$$I_m^l = \mathcal{I}_{(A_{\nu_l}(r_{m}), A_{\nu_l}  (r_{m-1}) )}(L_i,L'_j;H_{\nu_l})$$
and for $m=1$, let $I_m^l=I_1^l$ denote
 $$I_1^l = \mathcal{I}_{(A_{\nu_l}(r_{1}), \infty)}(L_i,L'_j;H_{\nu_l}).$$
Then Lemma \ref{lem:action} implies that $I_m^l$ is independent of $l$. Furthermore, for $x=(r_x, \gamma) \in I_m^l$, $r_x \in (r_{m-1},r_m)$ since $A_{\nu_l}$ is decreasing.

Since the Reeb flow is periodic, the Hamiltonian chords in $\mathcal{I} (L_i,L'_j; H_{\nu_l})$, $n \in \Z_{\geq 2}$, occur periodically as well. Indeed, for all $2 \leq m < l$, there exists a bijection between
\begin{equation*}\label{eq:bijection}
I_{m}^l \xrightarrow{1 : 1} I_{m+1}^l,x \mapsto \overline{x}.
\end{equation*}
More precisely, for each $x = (r_x, \gamma) \in I_{m}^l$, there exist $\overline{r}_x \in (r_m, r_{m+1})$ and a Reeb chord $\overline{\gamma}$ such that
\begin{itemize}
\item $h'_{\nu_l}(\overline{r}_x) = h'_{\nu_l}(r_x) +T_0$
\item $\overline{\gamma}$ is obtained from $\gamma$ by concatenating a periodic Reeb chord $\gamma_0$ of period $T_0$ to its end point and
\item $\overline{x} \coloneqq (\overline{r}_x, \overline{\gamma}) \in I_{m+1}^l.$
\end{itemize}
Here $\gamma_0 : [0,1]\to \partial W$ is the Reeb chord given by
\begin{equation*}\label{eq:gamma2}
\gamma_0 (t) =  \mathrm{Fl}_R^{T_0 t} ( \gamma(1)).
\end{equation*}

The requirement \eqref{eq:compatibility} and the definition \eqref{eq:maslovindex} of $\mu_{T_0}$ imply that the degrees of $x$ and $\overline{x}$ are related as follows:
\begin{equation}\label{eq:degree}
	\deg \overline{x} = \deg x - \mu_{T_0}.
\end{equation}

Since $I_2^l$ is finite, there exist some integers $A \leq B$ such that $A \leq \deg x \leq B$ for all $x\in I_2^l$, $l \geq 2$. By applying \eqref{eq:degree} inductively, we have:
\begin{lem}\label{lem:degree}
There exist integers $A \leq B$ such that for any Hamiltonian chord $x\in I_m^l$, $2\leq m \leq l$, the following holds:
$$ A-(m-2) \mu_{T_0} \leq \deg x \leq B-(m-2) \mu_{T_0}.$$
\end{lem}
Now we are ready to complete the proof of Proposition \ref{prop:periodic}. From Lemma \ref{lem:degree}, it follows that, for each $d \in \mathbb{Z}$, there exist finitely many $m \in \mathbb{Z}_{\geq 1}$ such that
$d =\deg x$ for some $x\in I_m^l$, $l \geq m$ since $\mu_{T_0}$ is nonzero by assumption. If there is no such an integer $m$, then $HW^d(L_i,L'_j) \cong HW^d(L_i,L_j)$ is zero.  Otherwise, let $m$ be the maximal integer among those. This means that every element of $HW^d(L_i,L'_j)$ lies in the image of $HF^d(L_i,L'_j;H_{\nu_m})$ in the direct limit \eqref{eq:directlimit}. But, since $HF^d(L_i,L'_j;H_{\nu_m})$ is finite-dimensional, this proves Proposition \ref{prop:periodic}.

\bibliographystyle{amsalpha}
	\bibliography{rabinowitz.bib}
\end{document}